\documentclass[a4paper,10pt]{smfart}

\usepackage[applemac]{inputenc}
\RequirePackage[T1]{fontenc}
\RequirePackage{amsfonts,latexsym,amssymb}

\RequirePackage{smfenum}
\RequirePackage[frenchb]{babel}
\usepackage{xy}
\addto\extrasfrenchb{\mathbfl@nonfrenchitemize
\mathbfl@nonfrenchspacing}

\newtheoremstyle{theorem}{11pt}{11pt}{\slshape}{}{\mathbfseries}{.}{.5em}{}
\newtheoremstyle{note}{11pt}{11pt}{}{}{\mathbfseries}{.}{.5em}{}

\theoremstyle{plain}
  \newtheorem{theoreme}{Th\'{e}or\`{e}me}[section]
  \newtheorem{proposition}[theoreme]{Proposition}
  \newtheorem{lemme}[theoreme]{Lemme}
  \newtheorem{corollaire}[theoreme]{Corollaire}
  \newtheorem{conj}[theoreme]{Conjecture}

\theoremstyle{definition}
  \newtheorem{definition}[theoreme]{D\'{e}finition}

\theoremstyle{remark}
  \newtheorem{example}[theoreme]{Exemple}
  \newtheorem{remarque}[theoreme]{Remarque}

\usepackage{url}

\let\mathcal\mathcal

\let\mathbf\mathbf

\def\Q{{\mathbf Q}} \def\Z{{\mathbf Z}}

\def\O{{\mathcal O}}

\def\dual{{\boldsymbol *}}

\def\zp{{\Z_p}}
\def\zpet{{\Z_p^\dual}}
\def\qp{{\Q_p}}
\def\qpet{{\Q_p^\dual}}
\def\p1{{\mathbf P}^1}
\def\qp{\mathbf{Q}_p}
\def\zp{\mathbf{Z}_p}
\def\z{\mathbf{Z}}
\def\qpp{\mathbf{Q}_{p^2}}
\def\cp{\mathbf{C}_p}
\def\q{\mathbf{Q}}
\def\ho{\mathrm{Hom}}

\def\con{\mathcal{C}^0}
\def\ob{\bar{B}}
\def\ext{\mathrm{Ext}}
\def\m{\mathcal{M}}

\def\oo{\O}

\def\epsilon{\varepsilon}

\begin{document}
\title[Rev\^etements de Drinfeld et correspondance de Langlands $p$-adique]
{Rev\^etements du demi-plan de Drinfeld et correspondance de Langlands $p$-adique}

\author{Gabriel Dospinescu}
\address{CNRS, UMPA, \'Ecole Normale Sup\'erieure de Lyon, 46 all\'ee d'Italie, 69007 Lyon, France}
\email{gabriel.dospinescu@ens-lyon.fr}

\author{Arthur-C\'esar Le Bras}
\address{Institut de Math\'ematiques de Jussieu, 4 place Jussieu, 75005 Paris, France}
\email{arthur-cesar.le-bras@imj-prg.fr \\ lebras@dma.ens.fr}

\begin{abstract} Nous d\'ecrivons le complexe de de Rham des rev\^etements du demi-plan de Drinfeld 
pour ${\rm GL}_2(\qp)$. Cette description, conjectur\'ee par Breuil et Strauch, fournit une r\'ealisation g\'eom\'etrique de la correspondance de Langlands locale $p$-adique pour 
certaines repr\'esentations de de Rham de dimension $2$ de 
${\rm Gal}(\overline{\qp}/\qp)$.
 \end{abstract}

\begin{altabstract} We describe the de Rham complex of the \'etale coverings of Drinfeld's $p$-adic upper half-plane for 
${\rm GL}_2(\qp)$. Conjectured by Breuil and Strauch, this description gives a geometric realization of the $p$-adic local Langlands correspondence for 
certain two-dimensional de Rham representations of ${\rm Gal}(\overline{\qp}/\qp)$. 
\end{altabstract}
\setcounter{tocdepth}{3}

\maketitle

\stepcounter{tocdepth}
{\Small
\tableofcontents
}

\section{Introduction}

La correspondance de Langlands locale \og classique\fg{} pour ${\rm GL}_n$ entretient un lien \'etroit avec la cohomologie du demi-espace de Drinfeld de dimension $n-1$ et de ses rev\^etements : la \textit{th\'eorie de Lubin-Tate non ab\'elienne} de Carayol \cite{Carayol} pr\'edit que la correspondance pour les repr\'esentations supercuspidales se r\'ealise dans la cohomologie \'etale $\ell$-adique de la tour de Drinfeld \cite{Drinfeld} (ou de Lubin-Tate, selon les go\^uts, cf. \cite{faltings} \cite{FGL}) et la mise en forme de ce principe joue un r\^ole crucial dans la preuve d'Harris et Taylor \cite{HT} de la correspondance.

Par contraste, l'existence de la correspondance de Langlands $p$-adique, qui n'est \`a l'heure actuelle formul\'ee et prouv\'ee que pour le groupe $G={\rm GL}_2(\qp)$, repose \cite{Cbigone} sur la th\'eorie de Fontaine \cite{FoGrot} des $(\varphi,\Gamma)$-modules. Elle n'a donc \`a premi\`ere vue aucune relation avec la g\'eom\'etrie des rev\^etements du demi-plan de Drinfeld\footnote{Pour certaines repr\'esentations galoisiennes, on sait cependant que la correspondance se r\'ealise dans la cohomologie compl\'et\'ee de la tour des courbes modulaires \cite{Emcomp} ; ceci joue d'ailleurs un r\^ole capital dans cet article.}. Pourtant, le cas classique laisse esp\'erer que celle-ci puisse expliquer la structure des repr\'esentations de $G$ associ\'ees aux repr\'esentations galoisiennes de de Rham non triangulines. Cet article se propose de montrer que c'est effectivement le cas. Les r\'esultats obtenus ont \'et\'e directement inspir\'es par une conjecture non publi\'ee de Breuil et Strauch \cite{BS}, qui donnait un sens pr\'ecis \`a cet espoir, en d\'ecrivant le complexe de de Rham de ces rev\^etements\footnote{Plus pr\'ecis\'ement, la conjecture \'etait faite pour le premier rev\^etement.} en termes de la correspondance de Langlands $p$-adique\footnote{La conjecture de Breuil-Strauch et les r\'esultats de ce travail ne disent rien sur les repr\'esentations de de Rham triangulines : pour une explication, voir la remarque \ref{cas triangulin}.}. Les r\'esultats obtenus, que nous d\'ecrivons plus en d\'etail dans la suite de cette introduction, indiquent que le probl\`eme analogue pour 
${\rm GL}_2(F)$ ($F$ \'etant une extension non triviale de $\qp$) ne sera pas une mince affaire, mais sugg\`erent des pistes de recherche int\'eressantes. En particulier, ils 
sont utilis\'es dans \cite{CDN} pour calculer la cohomologie \'etale $p$-adique des rev\^etements du demi-plan de Drinfeld, obtenant ainsi l'analogue 
$p$-adique (pour ${\rm GL}_2(\qp)$ pour l'instant...) de \cite{Carayol}.  

\subsection{Les r\'esultats principaux}
   Nous aurons besoin de quelques pr\'eliminaires pour \'enoncer notre premier r\'esultat principal.    
   Soit $D$ l'unique alg\`ebre de quaternions ramifi\'ee sur $\qp$
(\`a isomorphisme pr\`es), $\O_D$ son unique ordre maximal et soit $\varpi_D$ une uniformisante de $D$. 
Soit $\breve{\mathcal{M}}_n$ l'espace rigide analytique sur $\widehat{\qp^{\rm nr}}$, fibre g\'en\'erique du sch\'ema formel classifiant 
les d\'eformations par quasi-isog\'enie $\O_D$-\'equivariante d'un $\O_D$-module formel sp\'ecial de dimension $2$ et hauteur $4$ sur $\overline{\mathbf{F}_p}$, avec structure de niveau $1+p^n \O_D$. Les espaces $\breve{\mathcal{M}}_n$ 
forment une tour d'espaces analytiques, les morphismes de transition 
$\breve{\mathcal{M}}_{n+1}\to \breve{\mathcal{M}}_n$ \'etant finis \'etales. Chacun\footnote{L'action \og horizontale\fg{}, i.e. sur chaque \'etage de la tour, est celle de $G$, le groupe 
$D^*$ agissant \og verticalement\fg{} sur la tour. Puisque $1+p^n\O_D$ est distingu\'e dans $D^*$, l'action par correspondances de Hecke de
$D^*$ sur la tour pr\'eserve chaque $\breve{\mathcal{M}}_n$.} de ces espaces est muni 
 d'actions qui commutent des groupes 
$G={\rm GL}_2(\qp)$ et $D^*$, compatible avec 
 les morphismes de transition $\breve{\mathcal{M}}_{n+1}\to \breve{\mathcal{M}}_n$. De plus, les espaces 
 $\breve{\mathcal{M}}_n$ sont munis de donn\'ees de descente canoniques \`a la Weil, qui ne sont pas effectives, mais qui le deviennent 
 sur le quotient \footnote{On voit $p$ comme \'el\'ement du centre de $G$.} de 
     $\breve{\mathcal{M}}_n$ par $p^{\mathbf{Z}}$. On note $\Sigma_n$ le mod\`ele de $p^{\mathbf{Z}}\backslash\breve{\mathcal{M}}_n$ sur $\qp$ qui s'en d\'eduit. 
    L'espace rigide analytique $\Sigma_n$ est  
    un rev\^etement \'etale de $\Sigma_0$, de groupe de Galois le quotient
    $${\rm Gal}(\Sigma_n/\Sigma_0)=\O_D^*/(1+p^n \O_D).$$

    Soit $\Omega$ le demi-plan de Drinfeld, un espace rigide sur $\qp$ dont les $\cp$-points sont $$\Omega(\cp)=\mathbf{P}^1(\cp)- \mathbf{P}^1(\qp).$$ 
  Il admet une action de $G$, via l'action naturelle de $G$ sur $\mathbf{P}^1(\cp)$, donn\'ee par $g.z=\frac{az+b}{cz+d}$ si 
  $g=\left(\begin{smallmatrix} a & b \\ c & d\end{smallmatrix}\right)\in G$, o\`u $z$ est \og la \fg{} variable sur $\p1$. 
    L'espace $\Sigma_0$ n'est pas bien myst\'erieux : il s'agit simplement de deux copies de $\Omega$, avec action triviale
     de $\O_D^*$, l'\'el\'ement $\varpi_D$ permutant les deux copies. L'action de $g \in G$ est l'action naturelle sur $\Omega$ et \'echange ou non les deux copies de $\Omega$ selon que le d\'eterminant de $g$ est impair ou pair. La g\'eom\'etrie des rev\^etements $\Sigma_n$ et l'action de $G \times D^*$ sur $\Sigma_n$ sont par contre bien plus compliqu\'ees.
\\

   Fixons maintenant une repr\'esentation lisse supercuspidale $\pi$ de $G$, de caract\`ere central trivial (pour simplifier)
         et notons $$\rho={\rm JL}(\pi)$$ la repr\'esentation lisse irr\'eductible (de dimension finie) de $D^*$ qui lui est attach\'ee par la correspondance de Jacquet-Langlands locale. Il existe\footnote{L'existence est un fait standard de la th\'eorie, utiliser par exemple le fait que $\pi$ est l'induite d'une repr\'esentation de dimension finie 
   \`a partir d'un sous-groupe ouvert compact modulo le centre. Voir aussi le chapitre 4 de \cite{Breuil-Schneider} pour une discussion de la rationalit\'e dans un cadre beaucoup plus g\'en\'eral.}       
          une extension finie $L$ de $\qp$ telle que ces repr\'esentations soient d\'efinies sur $L$. Il est sous-entendu dans la suite que le corps des coefficients de toutes les repr\'esentations qui apparaissent est $L$ ; en particulier, si $X$ est un espace rigide sur $\qp$ et $\mathcal{F}$ est un faisceau coh\'erent sur $X$, on notera simplement $\mathcal{F}(X)$ pour $H^0(X,\mathcal{F})\otimes_{\qp} L$.
    
   Soit ${\rm Ban}^{\rm adm}(G)$ la cat\'egorie des repr\'esentations de $G$ sur des $L$-espaces de Banach $\Pi$, qui ont un r\'eseau ouvert, born\'e et $G$-invariant, dont la r\'eduction modulo $p$
      est lisse admissible au sens usuel (i.e. le sous-espace des vecteurs invariants par un sous-groupe ouvert compact arbitraire de 
      $G$ est fini).
        Si $\Pi\in {\rm Ban}^{\rm adm}(G)$, on note 
   $\Pi^{\rm an}$ (resp. $\Pi^{\rm lisse}$) le sous-espace de $\Pi$ form\'e des vecteurs localement analytiques 
 (resp. localement constants), i.e. des vecteurs dont l'application orbite\footnote{Si $v\in \Pi$ est un tel vecteur, son application orbite est 
 $G\to \Pi$, $g\mapsto g.v$.} est localement analytique (resp. localement constante). 
 Les espaces $\Pi^{\rm an}$ et $\Pi^{\rm lisse}$ sont stables sous l'action de $G$ et 
 $\Pi^{\rm an}$ est dense dans $\Pi$ \cite{STInv} (alors que
   $\Pi^{\rm lisse}$ est la plupart du temps nul). 
                  
   \begin{definition} On note $\mathcal{V}(\pi)$ l'ensemble des 
   repr\'esentations absolument irr\'eductibles $\Pi\in  {\rm Ban}^{\rm adm}(G)$ telles que
   $\Pi^{\rm lisse}\simeq \pi$.
   \end{definition} 
   
      La correspondance de Langlands locale $p$-adique pour $G$ fournit une description compl\`ete de 
      $\mathcal{V}(\pi)$ (voir la discussion suivant la remarque \ref{co ad}). 
           
     Le th\'eor\`eme suivant, qui est le premier r\'esultat principal de ce texte, fournit une description g\'eom\'etrique de la repr\'esentation localement analytique $\Pi^{\rm an}/\Pi^{\rm lisse}$
     quand $\Pi\in \mathcal{V}(\pi)$. Alternativement, on peut le voir comme une description de la $G\times D^*$ repr\'esentation 
     $\O(\Sigma_n)$ en termes de la correspondance de Jacquet-Langlands et surtout de la correspondance de Langlands $p$-adique. Si $V$ est un $L$-espace vectoriel localement convexe, on note $V^*$ son dual topologique. On pose aussi $$\sigma^{\rho}={\rm Hom}_{D^*}(\rho, \sigma)$$ pour toute $L$-repr\'esentation $\sigma$ de 
     $G\times D^*$.
  
   \begin{theoreme}\label{main1}
    Soit $\pi$ une repr\'esentation supercuspidale de $G={\rm GL}_2(\qp)$, \`a caract\`ere central trivial, et soit $\rho={\rm JL}(\pi)$ comme ci-dessus. 
    Pour tout $\Pi\in\mathcal{V}(\pi)$ et 
     pour tout $n$ assez grand (il suffit que 
    $\rho$ soit triviale sur $1+p^n\O_D$),
    il existe un isomorphisme de $G$-modules topologiques unique \`a scalaire pr\`es
    $$(\O(\Sigma_n)^{\rho})^*\simeq \Pi^{\rm an}/\Pi^{\rm lisse}.$$
   \end{theoreme}
   
\begin{remarque}\label{co ad} 

a) Au lieu de partir de $\pi$, on aurait pu plus g\'en\'eralement partir d'une repr\'esentation localement alg\'ebrique de la forme $\pi \otimes \mathrm{Sym}^{k}$, avec $\pi$ supercuspidale\footnote{C\^ot\'e Galois, cela revient \`a passer des poids de Hodge-Tate $0, 1$ aux poids $0, k+1$, comme on le verra plus bas.}. Tous les r\'esultats de ce texte s'\'etendent, \`a condition de consid\'erer des fibr\'es vectoriels diff\'erents sur $\Sigma_n$ : voir la remarque \ref{poids generaux}. De m\^eme, l'hypoth\`ese que le caract\`ere central de $\pi$ est trivial n'est pas essentielle, contrairement \`a l'hypoth\`ese que $\pi$ est supercuspidale (voir la remarque \ref{cas triangulin} pour plus de d\'etails concernant ce dernier point). 
      
b)  Une cons\'equence importante du th\'eor\`eme \ref{main1} est que 
       $(\O(\Sigma_n)^{\rho})^*$ est une $G$-repr\'esentation localement analytique admissible, au sens de 
Schneider et Teitelbaum \cite{STInv}. Le caract\`ere localement analytique s'\'etablit sans trop de mal, mais l'admissibilit\'e semble nettement plus d\'elicate. 
Qu'en est-il pour ${\rm GL}_2(F)$, ou m\^eme ${\rm GL}_n(F)$ ? Notre m\'ethode ne fournit aucune approche pour ce probl\`eme. 
        On peut esp\'erer que la th\'eorie des $D$-modules $p$-adiques permette de dire quelque chose de l'admissibilit\'e de ces repr\'esentations ind\'ependamment de la correspondance de Langlands $p$-adique (voir \cite{PSS} pour des r\'esultats dans cette direction). 
\end{remarque}

      Le th\'eor\`eme \ref{main1} affirme en particulier que le quotient $\Pi^{\rm an}/\Pi^{\rm lisse}$ ne d\'epend pas du choix de $\Pi \in \mathcal{V}(\pi)$. Cela n'est cependant que de la poudre aux yeux: la {\it{preuve}} du th\'eor\`eme \ref{main1} utilise de mani\`ere essentielle 
 cette ind\'ependance, qui a \'et\'e d\'emontr\'ee par voie tr\`es d\'etourn\'ee dans \cite[th. VI.6.43]{Cbigone}, et dont la preuve a \'et\'e consid\'erablement simplifi\'ee dans un travail r\'ecent de Colmez \cite{Colmezpoids}. Nous donnons aussi une preuve\footnote{En utilisant encore un certain nombre de r\'esultats de \cite{Cbigone}, ainsi qu'une astuce de \cite{Colmezpoids}.} 
  dans le chapitre $8$ de cet article, car nous avons besoin d'un certain nombre d'ingr\'edients de cette preuve pour montrer le th\'eor\`eme 
  \ref{main2} ci-dessous.
  
        Pour comprendre l'importance de l'ind\'ependance discut\'ee dans le paragraphe pr\'ec\'edent, il convient d'expliciter davantage l'ensemble $\mathcal{V}(\pi)$. 
      La correspondance de Langlands \og classique\fg{} pour $G$, normalis\'ee \`a la Tate, combin\'ee avec une recette de Fontaine \cite{per adiques} (compl\'et\'ee par la proposition $4.1$ de \cite{Breuil-Schneider}, qui permet \og d'inverser\fg{} cette recette pour attacher des 
      $(\varphi, N, \mathcal{G}_{\qp})$-modules \`a des repr\'esentations de Weil-Deligne, cf. le dernier paragraphe des notations et conventions) 
      permet d'associer \`a
$\pi$ un $(\varphi, \mathcal{G}_{\qp}:={\rm Gal}(\overline{\qp}/\qp))$-module $M(\pi)$, libre de rang $2$ sur 
$L\otimes_{\qp} \qp^{\rm nr}$, ainsi qu'un $L$-espace vectoriel de dimension $2$
$$M_{\rm dR}(\pi)=(\overline{\qp}\otimes_{\qp^{\rm nr}} M(\pi))^{\mathcal{G}_{\qp}}.$$
Un des r\'esultats principaux de \cite{PCD}, qui utilise la compatibilit\'e entre les correspondances 
de Langlands \og classique\fg{} et $p$-adique \cite{Emcomp}, montre que le foncteur de Colmez \cite{Cbigone} induit une bijection 
$$\Pi\to V(\Pi)$$
entre $\mathcal{V}(\pi)$ et l'ensemble des $L$-repr\'esentations 
absolument irr\'eductibles $V$ de dimension $2$ de $\mathcal{G}_{\qp}$, potentiellement cristallines \`a poids de Hodge-Tate 
$0,1$ et telles que 
$$D_{\rm pst}(V)\simeq M(\pi).$$
On a $\det V(\Pi)=\chi_{\rm cyc}$ pour tout 
$\Pi\in \mathcal{V}(\pi)$, car le caract\`ere central de $\pi$ est trivial 
(aussi innocente qu'elle puisse para\^itre, cette assertion est en fait 
la partie la plus technique de \cite{PCD}...). 
En combinant cela 
avec le th\'eor\`eme de Colmez-Fontaine \cite{CF}, et l'observation que, la repr\'esentation de Weil attach\'ee \`a $M(\pi)$ par la recette de Fontaine \'etant irr\'eductible, toutes les filtrations 
sur $M(\pi)$ sont automatiquement faiblement admissibles (voir la preuve du th\'eor\`eme 5.2 de \cite{Breuil-Schneider}), on en d\'eduit une bijection canonique
$$\mathcal{V}(\pi)\simeq {\rm Proj}(M_{\rm dR}(\pi)), \quad \Pi\mapsto {\rm Fil}^0(D_{\rm dR}(V(\Pi))),$$
 en consid\'erant $ {\rm Fil}^0(D_{\rm dR}(V(\Pi)))$ comme une $L$-droite 
de $M_{\rm dR}(\pi)$ via l'isomorphisme $D_{\rm dR}(V(\Pi))\simeq M_{\rm dR}(\pi)$ induit par\footnote{Ce dernier isomorphisme est unique \`a scalaire pr\`es, dont tout est \og canonique \`a scalaire pr\`es\fg{} dans ce qui pr\'ec\`ede, et l'identification
$\mathcal{V}(\pi)\simeq {\rm Proj}(M_{\rm dR}(\pi))$ est canonique tout court.... } $D_{\rm pst}(V(\Pi))\simeq M(\pi)$.  Notons $\mathcal{L}\to \Pi_{\mathcal{L}}$ l'inverse de cette bijection. Ainsi, $\mathcal{L}$ est la filtration de Hodge sur 
$D_{\rm dR}(V(\Pi_{\mathcal{L}}))$. L'ind\'ependance de $\Pi^{\rm an}/\Pi^{\rm lisse}$ pour $\Pi\in \mathcal{V}(\pi)$ \'equivaut alors 
\`a l'ind\'ependance       
   de $\Pi_{\mathcal{L}}^{\rm an}/\Pi_{\mathcal{L}}^{\rm lisse}$ par rapport \`a la filtration de Hodge $\mathcal{L}$ sur $M_{\rm dR}(\pi)$. Cela est surprenant et n'a rien de gratuit : la repr\'esentation $\Pi^{\rm an}$ permet de r\'ecup\'erer $\Pi$, et donc la filtration de Hodge sur $M_{\rm dR}(\pi)$, gr\^ace au r\'esultat principal de \cite{CD}, tandis que le quotient par les vecteurs lisses ne d\'epend que de $M(\pi)$. Autrement dit, les repr\'esentations $\Pi^{\rm an}$ pour $\Pi\in \mathcal{V}(\pi)$ sont des extensions
   $$0\to \pi\to \Pi^{\rm an}\to \Pi(\pi,0)\to 0$$
   d'une repr\'esentation $\Pi(\pi,0)$ qui ne d\'epend que de $M(\pi)$, par $\pi$. La filtration de Hodge encode cette extension. Nous verrons plus loin comment r\'ecup\'erer cette filtration. Une cons\'equence essentielle de l'ind\'ependance par rapport \`a la filtration de Hodge est qu'elle nous permet de choisir 
         un $\Pi\in \mathcal{V}(\pi)$ convenable, et ainsi d'utiliser des m\'ethodes globales pour montrer l'existence d'un morphisme non nul entre les deux objets du th\'eor\`eme \ref{main1}. Nous montrons ensuite par voie locale que tout tel morphisme est un isomorphisme, et qu'il est unique \`a scalaire pr\`es : c'est le coeur technique de l'article, voir la section suivante pour plus de d\'etails. 
\\

       La description g\'eom\'etrique de $\Pi^{\rm an}/\Pi^{\rm lisse}$ \'etant acquise gr\^ace au th\'eor\`eme \ref{main1}, nous voulons maintenant d\'ecrire
      $\Pi^{\rm an}$ g\'eom\'etriquement, et r\'ecup\'erer ainsi la filtration de Hodge. C'est ici qu'intervient le complexe de de Rham de $\Sigma_n$.

 L'espace $\Sigma_n$ est Stein, ce qui fournit une suite exacte d'espaces de Fr\'echet avec action de $G\times D^*$
          $$0\to H_{{\mathrm{dR}}}^0(\Sigma_n) \to \O(\Sigma_n)\to \Omega^1(\Sigma_n)\to H_{\rm dR}^1(\Sigma_n)\to 0.$$
 En passant aux composantes $\rho:={\rm JL}(\pi)$-isotypiques\footnote{Cela utilise de mani\`ere cruciale le fait que 
 $\rho$ est de dimension $>1$, ce qui fournit $H_{{\mathrm{dR}}}^0(\Sigma_n)^{\rho}=0$, gr\^ace \`a \cite{Strauch cc} et 
 \cite{FGL}.}, on obtient une suite exacte de 
 $G$-repr\'esentations sur des espaces de Fr\'echet
      $$0\to \O(\Sigma_n)^{\rho}\to \Omega^{1}(\Sigma_n)^{\rho} \to H^1_{\rm dR}(\Sigma_n)^{\rho} \to 0.$$

\begin{theoreme}\label{main2} Il existe un isomorphisme canonique (\`a scalaire pr\`es) de $G$-modules de Fr\'echet 
$$H^1_{\rm dR}(\Sigma_n)^{\rho}\simeq M_{\rm dR}^*\otimes_{L} \pi^*,$$
tel que pour toute $L$-droite $\mathcal{L}$ de $M_{\rm dR}$, l'image inverse de 
$\mathcal{L}^{\perp}\otimes_{L} \pi^*\subset H^1_{\rm dR}(\Sigma_n)^{\rho}$ dans $\Omega^1(\Sigma_n)^{\rho}$ est isomorphe \`a
$(\Pi_{\mathcal{L}}^{\rm an})^*$ et la suite exacte 
$$0\to \O(\Sigma_n)^{\rho}\to (\Pi_{\mathcal{L}}^{\rm an})^*\to \mathcal{L}^{\perp}\otimes_{L} \pi^*\simeq \pi^*\to 0$$
qui s'en d\'eduit est duale de celle fournie par le th\'eor\`eme \ref{main1}.
\end{theoreme}
Ce th\'eor\`eme, qui est le deuxi\`eme r\'esultat principal de l'article, donne donc une recette g\'eom\'etrique simple pour construire $\Pi^{\rm an}$ \`a partir de la donn\'ee de $M(\pi)$ (ou de fa\c con \'equivalente, de $\pi$) et de la filtration de Hodge, \`a partir du complexe de de Rham. C'\'etait l'objet de la conjecture originale de Breuil-Strauch \cite{BS} (qui \'etait toutefois formul\'ee de mani\`ere un peu diff\'erente, voir la remarque \ref{difference}). 

\begin{remarque}\label{cas triangulin}
Si $\pi$ est une repr\'esentation de la s\'erie principale, et $\Pi \in \mathrm{Ban}^{\rm adm}(G)$ contient $\pi$, on ne dispose pas d'une telle description g\'eom\'etrique de $\Pi^{\rm an}$ : la situation est bien s\^ur similaire \`a celle de la correspondance de Langlands locale classique, o\`u les repr\'esentations de la s\'erie principale n'apparaissent pas dans la cohomologie $\ell$-adique \`a supports de la tour de Drinfeld. Si $\pi$ est un twist de la Steinberg, les premiers travaux de Breuil \cite{Br} sur la correspondance de Langlands $p$-adique fournissent une description partielle de $\Pi^{\rm an}$ utilisant le demi-plan $\Omega$, mais la situation est plus compliqu\'ee : il ne suffit pas de copier les \'enonc\'es pr\'ec\'edents avec $\rho$ triviale. Cela s'explique en partie par le fait que dans ce cas la cohomologie de de Rham est non triviale aussi en degr\'e $0$, cf. le paragraphe \ref{cat\'egorie d\'eriv\'ee}. 

Toutefois, dans ces deux cas, la repr\'esentation galoisienne attach\'ee \`a une telle $\Pi$ est trianguline, et on dispose donc d'une description tr\`es pr\'ecise des vecteurs localement analytiques \cite{Cvectan} \cite{LXZ}.
\end{remarque}

\begin{remarque} Rempla\c cons $\Sigma_n$ par le premier rev\^etement $\Sigma_{1+\varpi_D \O_D}$ et consid\'erons comme dans l'\'enonc\'e du th\'eor\`eme le dual de l'image inverse de $\mathcal{L}^{\perp}\otimes_{L} \pi^*\subset H^1_{\rm dR}(\Sigma_{1+\varpi_D \O_D})^{\rho}$ dans $\Omega^1(\Sigma_{1+\varpi_D \O_D})^{\rho}$, pour $\rho$ repr\'esentation lisse irr\'eductible non triviale de $D^*$, triviale sur $1 + \varpi_D \O_D$. Dans un travail r\'ecent \cite{luepan} et dans une formulation un peu diff\'erente, Lue Pan montre que le compl\'et\'e unitaire universel de cette repr\'esentation est admissible et que sa r\'eduction modulo $\pi_L$ co\"incide avec la r\'eduction modulo $\pi_L$ de $V(\Pi_{\mathcal{L}})$ par la correspondance de Langlands semi-simple \og modulo $p$\fg{}. Sa preuve exploite la g\'eom\'etrie d'un mod\`ele formel explicite de $\Sigma_{1+\varpi_D \O_D}$.
\end{remarque}

\subsection{Survol de la preuve} 

    Comme nous l'avons d\'ej\`a pr\'ecis\'e, la preuve du th\'eor\`eme \ref{main1} combine des arguments globaux et locaux. La plupart des ingr\'edients apparaissant dans sa preuve sont
    aussi utilis\'es pour d\'emontrer le th\'eor\`eme \ref{main2}, donc nous allons nous concentrer uniquement sur la preuve du th\'eor\`eme \ref{main1} dans la suite. 
\\
    
      Commençons par la partie globale. Le but est de construire dans un premier temps un morphisme non nul, $G$-\'equivariant et continu, de $(\Omega^1(\Sigma_n)^{\rho})^*$ dans $\Pi^{\rm an}$, pour un \textit{certain}\footnote{Ce genre de strat\'egie avait \'et\'e employ\'ee par Emerton \cite{Emcomp} pour d\'emontrer la compatibilit\'e entre les correspondances de Langlands locale $p$-adique et classique pour $G$.} $\Pi \in \mathcal{V}(\pi)$. Consid\'erons une alg\`ebre de quaternions 
      $B$ sur $\mathbf{Q}$ ramifi\'ee en $p$ et d\'eploy\'ee \`a l'infini. Elle donne naissance \`a une tour de courbes de Shimura
      $({\rm Sh}_K)_{K}$ index\'ee par les sous-groupes ouverts compacts 
      $K$ de $B^*(\mathbf{A}_f)$ (on voit $B^*$ comme un groupe alg\'ebrique sur $\mathbf{Q}$ dans la suite). 
      Fixons un sous-groupe ouvert compact suffisamment petit $K^p$ de $B^*(\mathbf{A}_f^p)$ et consid\'erons 
      $K=(1+p^n\O_D)K^p\subset B^*(\mathbf{A}_f)$. Le th\'eor\`eme d'uniformisation 
      de Cerednik-Drinfeld (plus quelques contorsions topologiques) permet d'obtenir un isomorphisme
            \begin{eqnarray}\Omega^1({\rm Sh}_K\otimes_{\mathbf{Q}} \qp)^{\rho} \simeq {\rm Hom}_G^{\rm cont}( (\Omega^1(\Sigma_n)^{\rho})^*, {\rm LA}(X(K^p))), \label{unif} \end{eqnarray}
      o\`u 
      $$X(K^p)=\bar{B}^*(\mathbf{Q})\backslash \bar{B}^*(\mathbf{A}_f)/K^p,$$
     $\bar{B}$ \'etant l'alg\`ebre de quaternions sur $\mathbf{Q}$ ayant les m\^emes invariants que $B$ aux places diff\'erentes  
     de $p$ et $\infty$, et des invariants \'echang\'es en ces places ($\bar{B}$ est donc compacte modulo centre \`a l'infini). L'espace 
     $X(K^p)$ est une vari\'et\'e analytique, au sens na\"if du terme, compacte, avec une action 
     localement analytique de $G$. L'espace ${\rm LA}(X(K^p))$ des fonctions localement analytiques sur $X(K^p)$ \`a valeurs dans $L$ est muni d'une action de l'alg\`ebre de Hecke hors $p$, et cette action commute \`a l'action de 
     $G$. 
      
       Ce qui pr\'ec\`ede n'utilise pas le fait que l'on travaille avec ${\rm GL}_2(\qp)$, mais \`a partir de maintenant nous allons pleinement exploiter 
       ce qu'on conna\^it sur ce groupe.
              Le point cl\'e est de comprendre les espaces Hecke-propres 
       dans ${\rm LA}(X(K^p))$. Comme $\bar{B}$ est d\'eploy\'ee en $p$, cela se fait en reprenant mot \`a mot les arguments qui ont permis \`a Emerton \cite{Emcomp} de comprendre la 
       cohomologie compl\'et\'ee de la tour des courbes modulaires. En fait, en globalisant convenablement\footnote{De telle sorte que la repr\'esentation galoisienne associ\'ee \`a la forme automorphe globalisant $\pi$ soit irr\'eductible en r\'eduction mod $p$ et en restriction \`a 
       $\mathcal{G}_{\qp}$} la repr\'esentation 
       $\pi$ nous pouvons nous placer dans une situation relativement simple - mais qui demande quand m\^eme toute la force de la correspondance de Langlands locale 
       $p$-adique ! Ainsi, en regardant les espaces $\mathfrak{p}$-propres des deux c\^ot\'es de l'isomorphisme $\eqref{unif}$ pour un id\'eal maximal 
       convenable $\mathfrak{p}$ de l'alg\`ebre de Hecke sph\'erique et en utilisant la version\footnote{Notons que l'on a besoin d'une version forte de cette compatibilit\'e, i.e. il ne suffit pas de savoir que $\Pi^{\rm an}$ appara\^it dans l'espace propre
${\rm LA}(X(K^p))[\mathfrak{p}]$, mais qu'en plus $\Pi^{\rm an}$ est l'unique sous-quotient irr\'eductible de cet espace.}
        du th\'eor\`eme de compatibilit\'e local-global d'Emerton pour comprendre l'espace propre
       ${\rm LA}(X(K^p))[\mathfrak{p}]$, on obtient un morphisme $G$-\'equivariant non nul continu
       $(\Omega^1(\Sigma_n)^{\rho})^*\to \Pi^{\rm an}$, pour un certain $\Pi\in \mathcal{V}(\pi)$, ce qui induit par dualit\'e un morphisme
       $(\Pi^{\rm an})^*\to \Omega^1(\Sigma_n)^{\rho}$. 
       On v\'erifie sans mal que ce morphisme se restreint en un morphisme non nul $G$-\'equivariant continu 
       $$\Phi: (\Pi^{\rm an}/\Pi^{\rm lisse})^*\to \O(\Sigma_n)^{\rho}.$$
       D'apr\`es les r\'esultats de Colmez \'evoqu\'es plus haut, le membre de gauche, tout comme le membre de droite, ne d\'epend que de $M$ (ou, de fa\c con \'equivalente, de $\pi$, ou de $\rho$), et pas du choix de $\Pi \in \mathcal{V}(\pi)$ : on notera d\'esormais $\Pi(\pi,0)=\Pi^{\rm an}/\Pi^{\rm lisse}$.
\\ 
       
          La suite de la preuve, qui repr\'esente la partie la plus technique de l'article, consiste \`a montrer que 
          $\Phi$ est un isomorphisme, et qu'il est unique \`a scalaire pr\`es. L'argument est un peu acrobatique. 
         Nous commençons par munir 
          $\Pi(\pi,0)^*$ d'une structure de $\O(\Omega)$-module telle que $\Phi$ soit $\O(\Omega)$-lin\'eaire. Cela se fait en exploitant 
          la construction explicite de $\Pi$ via les $(\varphi,\Gamma)$-modules, et la comprehension de l'action de l'alg\`ebre de Lie 
          $\mathfrak{gl}_2$ de $G$ sur $\Pi^{\rm an}$.      
     Pour motiver un peu la construction, notons que l'op\'erateur 
       $\partial: \O(\Sigma_n)^{\rho}\to \O(\Sigma_n)^{\rho}$ de multiplication par $z\in \O(\Omega)$ 
       encode la structure de $\O(\Omega)$-module de $\O(\Sigma_n)^{\rho}$, et est uniquement caract\'eris\'e par 
       l'\'egalit\'e d'op\'erateurs sur $\O(\Sigma_n)^{\rho}$ 
        $$a^+-1=u^+\circ \partial,$$
       o\`u $a^+$ (respectivement $u^+$) d\'esigne l'action infinit\'esimale de 
        $\left(\begin{smallmatrix} \zpet & 0 \\ 0 & 1\end{smallmatrix}\right)$ (respectivement
        $\left(\begin{smallmatrix} 1 & \zp \\ 0 & 1\end{smallmatrix}\right)$) sur $\O(\Sigma_n)^{\rho}$. Notons que l'op\'erateur $u^+$ agit comme $-\frac{d}{dz}$. 
                          
          Le point est alors de refaire ces constructions du c\^ot\'e des $(\varphi,\Gamma)$-modules (tout cela est fortement inspir\'e d'un travail en cours de 
          Colmez \cite{Colmezpoids}). On d\'emontre ainsi l'existence d'un automorphisme $\partial$ du $L$-espace vectoriel topologique
          $\Pi(\pi,0)^*$ uniquement caract\'eris\'e par 
          le fait que 
            $$a^+-1=u^+\circ \partial.$$
             L'existence de $\partial$ est un th\'eor\`eme d\'elicat de Colmez \cite{Colmezpoids}, dont on donne une nouvelle preuve. 
      La notation $\partial$ peut para\^itre pour le moins \'etrange, sachant qu'il s'agit d'un op\'erateur de \og multiplication par $z$\fg{} : elle vient du fait que 
      $\partial$ encode la connexion sur le $(\varphi,\Gamma)$-module (sur l'anneau de Robba) attach\'e \`a $V(\Pi)$. 
       Au vu des remarques pr\'ec\'edentes, le th\'eor\`eme suivant ne devrait pas surprendre le lecteur, mais nous insistons sur le fait qu'il requiert un certain nombre d'estim\'ees pas totalement triviales et qu'il joue un r\^ole d\'ecisif dans la preuve des r\'esultats principaux de l'article. 
        
                  \begin{theoreme} \label{Module} Pour tout 
                  $\Pi\in \mathcal{V}(\pi)$
            il existe une unique structure de $\O(\Omega)$-module sur 
            $\Pi(\pi,0)^*$ qui \'etend sa structure de $L$-espace vectoriel, et telle que 
            $z\in \O(\Omega)$ agit comme $\partial$. 
          \end{theoreme}
          
          Un ingr\'edient crucial dans la preuve de ce th\'eor\`eme est la dualit\'e de Morita (\cite{Morita 1, STMorita}, par exemple), i.e. l'isomorphisme de 
          $G$-modules topologiques
  $$\Omega^1(\Omega)\simeq ({\rm St}^{\rm an})^*, \quad \mu\in ({\rm St}^{\rm an})^*\mapsto \omega_{f}=\left(\int_{\p1(\qp)} \frac{1}{z-x}\mu(x)\right)dz.$$
  Ici $z$ est \og la\fg{} variable sur $\p1$, ${\rm St}^{\rm an}$ est la Steinberg localement analytique, quotient de l'espace 
  ${\rm LA}(\p1(\qp))$ des fonctions localement analytiques sur $\p1(\qp)$ par les fonctions constantes.
  La structure de $\O(\Omega)$-module du th\'eor\`eme pr\'ec\'edent est alors donn\'ee par 
  $$\left(\int_{\p1(\qp)} \frac{1}{z-x}\mu(x)\right)\cdot l=\int_{\p1(\qp)}(\partial-x)^{-1}(l)\mu(x),$$
  pour tout $l\in (\Pi^{\rm an}/\Pi^{\rm lisse})^*$ (bien s\^ur, il faut donner un sens \`a ces expressions!). 
 
      Un r\'esultat frappant que l'on obtient, comme corollaire du r\'esultat final, est que 
           $\Pi(\pi,0)^*$ est localement libre de rang $\dim_{L}(\rho)$ comme $\O(\Omega)$-module. Cela semble tr\`es d\'elicat \`a d\'emontrer, et m\^eme \`a deviner,
           en utilisant seulement la th\'eorie des $(\varphi,\Gamma)$-modules, qui sert \`a construire $\Pi(\pi,0)$ et l'op\'erateur $\partial$.
             
         Une fois le th\'eor\`eme \ref{Module} d\'emontr\'e, nous montrons que $\Phi: \Pi(\pi,0)^* \to \O(\Sigma_n)^{\rho}$ est surjectif. 
         Cela se fait en deux \'etapes: nous montrons d'abord que $\Phi$ est d'image dense, et ensuite qu'il est surjectif. La densit\'e de l'image de $\Phi$ vient de l'irr\'eductibilit\'e du fibr\'e $G$-\'equivariant sur 
             $\Omega$ dont les sections globales sont $ \O(\Sigma_n)^{\rho}$, r\'esultat qui se d\'emontre en utilisant les r\'esultats de Kohlhaase 
             \cite{Kohl}, permettant de \og transf\'erer le probl\`eme\fg{} sur la tour de Lubin-Tate: via ce transfert, l'irr\'eductibilit\'e se ram\`ene \`a l'irr\'eductibilit\'e 
             du fibr\'e $D^*$-\'equivariant $\rho^*\otimes \O_{\check{\mathbf{P}^1}}$, qui est nettement plus facile \`a \'etablir\footnote{Mais qui utilise de mani\`ere cruciale
             le fait que $\rho$ est irr\'eductible et {\it{lisse}}.}.
             
            Expliquons enfin rapidement l'argument pour l'injectivit\'e de $\Phi$. 
                  On note $\O(k)(\Sigma_n)$ la repr\'esentation de $G$ sur 
       $\O(\Sigma_n)$ obtenue en tordant l'action naturelle comme suit\footnote{Comme
     $\O(\Sigma_n)$ est un $\O(\Omega)\simeq \O(\Sigma_0)^{D^*}$-module (l'isomorphisme \'etant donn\'e par le plongement diagonal
     de $\O(\Omega)$ dans $\O(\Sigma_0)$), et comme $(a-cz)^{-k}\in \O(\Omega)$, la formule pr\'ec\'edente a un sens.}
     $$\left(\begin{smallmatrix} a & b \\ c & d\end{smallmatrix}\right)*_kf=(a-cz)^{-k}\cdot \left(\left(\begin{smallmatrix} a & b \\ c & d\end{smallmatrix}\right).f\right).$$
           Puisque 
       $\Sigma_n$ est \'etale sur $\Sigma_0$, on a une trivialisation 
       $\Omega^1(\Sigma_n)\simeq \O(\Sigma_n)dz$, qui induit un isomorphisme de 
       $G$-repr\'esentations 
       $$\Omega^1(\Sigma_n)\simeq \O(2)(\Sigma_n)\otimes \det.$$
       On peut faire les m\^emes constructions purement \`a partir
            des $(\varphi,\Gamma)$-modules, ce qui permet de d\'efinir une repr\'esentation $\Pi(\pi,2)^*$ 
          en faisant agir $G$ sur $\Pi(\pi,0)^*$ par 
            $$\left(\begin{smallmatrix} a & b \\ c & d\end{smallmatrix}\right)*v:=\det g\cdot (a-c\partial)^{-2} \left(\begin{smallmatrix} a & b \\ c & d\end{smallmatrix}\right).v.$$ 
            Le morphisme 
            $\Phi$ \'etant $\O(\Omega)$-lin\'eaire, $G$-\'equivariant et surjectif, il induit un morphisme 
            $G$-\'equivariant continu et surjectif
            $\Phi: \Pi(\pi,2)^*\to \Omega^1(\Sigma_n)$, et il suffit de d\'emontrer que ce morphisme est injectif. Cela se fait en plusieurs \'etapes. 
            D'abord, nous utilisons encore une fois l'uniformisation de Cerednik-Drinfeld et un argument avec la suite spectrale de Hochschild-Serre 
            comme dans \cite{Harris Inv} et \cite{Fargues these} pour montrer que $H^1_{\rm dR}(\Sigma_n)^{\rho}$ admet deux copies de
            $\pi^*$ comme quotient. Ensuite, la \og th\'eorie du mod\`ele de Kirillov\fg{} de Colmez permet \footnote{Cela fait bon usage d'anneaux de Fontaine, de l'\'equation diff\'erentielle attach\'ee par Berger \cite{Ber} \`a une repr\'esentation de de Rham, ainsi que des r\'esultats de \cite{Annalen} et \cite[chap VI]{Cbigone}.} de montrer que 
            le conoyau de $u^+: \Pi(\pi,0)^*\to \Pi(\pi,2)^*$ est canoniquement (\`a scalaire pr\`es) isomorphe \`a $M_{\rm dR}^*\otimes \pi^*$. On d\'eduit de ce qui pr\'ec\`ede que le morphisme 
            $$\Phi: \Pi(\pi,2)^*/u^+(\Pi(\pi,0))^*\to  \Omega^1(\Sigma_n)/d(\O(\Sigma_n)^{\rho})$$
            est forc\'ement un isomorphisme, ce qui fournit au passage un isomorphisme\footnote{Une m\'ethode plus naturelle serait d'utiliser la cohomologie d'Hyodo-Kato d'un mod\`ele de $\Sigma_n$, mais cela pose un certain nombre de probl\`emes...}
            $$H^1_{\rm dR}(\Sigma_n)^{\rho}\simeq M_{\rm dR}^*\otimes \pi^*.$$
           De cela on d\'eduit assez facilement l'injectivit\'e de $\Phi$, en prouvant qu'il n'existe pas de sous-espace $G$-stable dans $\cap_{n\geq 0} (u^+)^n(\Pi(\pi,2)^*)$. Ce dernier espace est en fait nul (cela d\'ecoule de \cite{Colmezpoids} ou \cite{Dthese} et utilise de mani\`ere cruciale le fait que les repr\'esentations auxquelles on travaille ne sont pas triangulines), ce qui joue un r\^ole important dans la preuve du th\'eor\`eme \ref{corollaireter} ci-dessous. 
                   
            Ce qui pr\'ec\`ede montre  
            que l'application $u^+: \Pi(\pi,0)^*\to \Pi(\pi,0)^*$ induit une suite exacte de 
             $G$-modules de Fr\'echet 
             $$0\to \Pi(\pi,0)^*\to \Pi(\pi,2)^*\to M_{\rm dR}^*\otimes_{L} \pi^*\to 0$$
           et que l'isomorphisme $\Pi(\pi,0)^*\simeq \O(\Sigma_n)^{\rho}$ induit un 
                        isomorphisme de $G$-modules topologiques 
            $$\Omega^1(\Sigma_n)^{\rho}\simeq \Pi(\pi,2)^*.$$
     Le th\'eor\`eme \ref{main2} s'en d\'eduit en suivant soigneusement ces identifications. 
        
        \begin{remarque}
          Colmez a d\'emontr\'e \cite{Colmezpoids} que la repr\'esentation $\Pi(\pi,0)$ est (topologiquement) irr\'eductible, ce qui fournit une preuve directe de l'injectivit\'e. 
          Nous avons toutefois besoin de tous les ingr\'edients ci-dessus pour la preuve du th\'eor\`eme \ref{main2}. 
        \end{remarque}
                        
\subsection{Compl\'ements}

Nombre des objets construits \`a l'aide des $(\varphi,\Gamma)$-modules mentionn\'es pr\'ec\'edemment trouvent donc une interpr\'etation g\'eom\'etrique, \`a l'exception notable du faisceau $G$-\'equivariant $U\to tN_{\rm rig} \boxtimes U$ sur $\mathbf{P}^1(\qp)$ construit par Colmez\footnote{Pour une d\'efinition, voir le \S\  \ref{points fixes}.}, dont l'espace des sections globales contient $\Pi(\pi,0)^*$ et qui joue un r\^ole capital dans la th\'eorie. Dans la derni\`ere section de cet article, nous proposons une interpr\'etation g\'eom\'etrique naturelle de ce faisceau, qui prolonge naturellement la conjecture de Breuil-Strauch. Le lecteur est renvoy\'e \`a \ref{bord} pour un \'enonc\'e pr\'ecis. Contentons-nous pour finir cette partie de citer deux autres cons\'equences de nos r\'esultats.

           Soit $D(\Gamma)$ l'alg\`ebre des distributions sur $\Gamma={\rm Gal}(\qp^{\rm cyc}/\qp)\simeq \zpet$ \`a valeurs dans $L$. Si $V$ est une repr\'esentation de de Rham de 
      $\mathcal{G}_{K}:={\rm Gal}(\overline{\qp}/K)$, avec $K$ une extension finie de $\qp$, on note $H^1_e(\mathcal{G}_K, V)$ l'image de l'exponentielle de Bloch-Kato. 
       Soit maintenant $V\in \mathcal{V}(\pi)$.
        On dispose d'applications naturelles 
      $$\int_{1+p^n\zp}: H^1(\mathcal{G}_{\qp}, D(\Gamma)\otimes_{L} V)\to H^1(\mathcal{G}_{F_n}, V),$$
      o\`u $F_n=\qp(\mu_{p^n})$.
     
     \begin{theoreme}\label{corollairebis}
        Il existe un isomorphisme de $D(\Gamma)$-modules libres de rang $2$ 
          $$(\O(\Sigma_n)^{\rho})^{\left(\begin{smallmatrix} p & 0 \\0 & 1\end{smallmatrix}\right)=1}\simeq \{\mu 
          \in H^1(\mathcal{G}_{\qp}, D(\Gamma)\otimes_{L} V) ~ | \int_{1+p^n\zp} \mu\in H^1_{e}(\mathcal{G}_{F_n}, V) \quad \forall n\geq 0\}.$$
     \end{theoreme}
     
        Enfin, la trivialisation  
        $\Omega^1(\Sigma_n)=\O(\Sigma_n)dz$ permet de d\'efinir une application 
        $\frac{d}{dz}: \O(\Sigma_n)\to \O(\Sigma_n)$. On dit qu'une fonction 
        $f\in \O(\Sigma_n)$ est infiniment primitivable si $f$ est dans l'image de 
        $(\frac{d}{dz})^{\circ k}$ pour tout $k$. En d'autres termes, $f$ est infiniment primitivable
        si $f\in \cap_{k\geq 0} (u^+)^{k}(\O(\Sigma_n)).$
     
     \begin{theoreme}\label{corollaireter}
       Soit $f\in \O(\Sigma_n)$ une fonction infiniment primitivable sur $\Sigma_n$. Alors 
       $f\in \O(\Omega)$.  
     \end{theoreme}

     \begin{remarque}  Dans un article ult\'erieur \cite{nulles}, nous discuterons le lien entre $\O(\Sigma_n)$ et la courbe de Fargues-Fontaine : soit 
      $\O(\Sigma_n)_{\infty}$
le sous-espace de $\O(\Sigma_n)$ des fonctions 
      $f$ telles que 
   $$ \lim_{v_p(b)\to-\infty} \left(\begin{smallmatrix} 1 & b \\0 & 1\end{smallmatrix}\right)f=0.$$
   
 G\'eom\'etriquement, $\O(\Sigma_n)_{\infty}$ est le sous-espace de 
   $\O(\Sigma_n)$ form\'e des fonctions qui tendent vers z\'ero quand \og on s'approche dans les directions rationnelles du point $\infty$ du bord\fg{}. Soient
   $$\tilde{\mathbf{B}}_{\rm rig}^+=\bigcap_{n\geq 0} \varphi^n(\mathbf{B}_{\rm cris}^+), \quad \mathcal{H}_{\qp}={\rm Gal}(\overline{\qp}/\qp^{\rm cyc}).$$
    On d\'emontre alors \cite{nulles} qu'il                       
                          existe un isomorphisme de repr\'esentations de $B=\left(\begin{smallmatrix} \qpet & \qp \\0 & \qpet\end{smallmatrix}\right)$
            $$\O(\Sigma_n)^{\rho}_{\infty}\simeq (\tilde{\mathbf{B}}_{\rm rig}^+\otimes_{\qp^{\rm nr}} M(\pi))^{\mathcal{H}_{\qp}}\otimes \delta,$$
            o\`u $\delta: B\to \qpet$ est le caract\`ere $\delta\left(\left(\begin{smallmatrix} a & b \\ 0 & d\end{smallmatrix}\right)\right)=\frac{a}{d}$.  
     Pr\'ecisons simplement que l'action de $\left(\begin{smallmatrix} \zpet & 0 \\0 & 1\end{smallmatrix}\right)$ sur le terme de droite se fait \`a travers l'action naturelle de $\Gamma={\rm Gal}(\qp^{\rm cyc}/\qp)$, via l'isomorphisme $\left(\begin{smallmatrix} \zpet & 0 \\0 & 1\end{smallmatrix}\right)\simeq \Gamma$ induit par le caract\`ere cyclotomique. L'action de $\left(\begin{smallmatrix} p & 0 \\0 & 1\end{smallmatrix}\right)$ correspond \`a l'action du Frobenius sur $(\tilde{\mathbf{B}}_{\rm rig}^+\otimes_{\qp^{\rm nr}} M(\pi))^{\mathcal{H}_{\qp}}$. Notons aussi que tous les objets dans l'\'enonc\'e  pr\'ec\'edent ont un sens pour ${\rm GL}_2(F)$. On peut naturellement se demander si c'est plus qu'une coïncidence... 
\\

     \end{remarque}

\subsection{Plan de l'article} L'encha\^inement des chapitres de ce texte suit essentiellement le cheminement de la preuve esquiss\'ee ci-dessus, dont nous reprenons les notations. La construction du morphisme $\Phi : (\Pi^{\rm an}/\Pi^{\rm lisse})^* \to \O(\Sigma_n)^{\rho}$ est l'objet de la section \ref{construction morphisme}. Les deux chapitres pr\'ec\'edents contiennent des r\'esultats pr\'eliminaires \`a cette construction : description de la tour de Drinfeld et propri\'et\'es de l'action des groupes $G$ et $D^*$ sur la tour (section \ref{cotedrinfeld}) ; th\'eor\`eme d'uniformisation $p$-adique (section \ref{calculdeRham}), rappels sur les formes automorphes sur les alg\`ebres de quaternions (section \ref{calculdeRham}) et enfin le calcul de la $\pi$-partie de la cohomologie de Rham \`a supports compacts des rev\^etements de Drinfeld qui sera utile plus tard. La preuve du th\'eor\`eme de compatibilit\'e local-global (d'apr\`es Emerton) est repouss\'ee en appendice. Le chapitre \ref{rappel phi Gamma} est constitu\'e de quelques rappels standard sur la th\'eorie des $(\varphi,\Gamma)$-modules, tandis que le chapitre \ref{kirillovcolmez} contient des rappels, moins standard et fondamentaux pour la suite, sur la correspondance de Langlands $p$-adique : en particulier, la description de l'action infinit\'esimale de $G$ sur les vecteurs localement analytiques et la th\'eorie du mod\`ele de Kirillov de Colmez. Ces r\'esultats sont pleinement utilis\'es dans la section \ref{independance du quotient} pour construire $\Pi(\pi,0)$, puis dans la section \ref{cotecolmez} pour munir $\Pi(\pi,0)^*$ d'un op\'erateur $\partial$ et d'une structure de $\O(\Omega)$-module. La d\'emonstration de la surjectivit\'e de $\Phi$ est alors possible et expos\'ee dans le chapitre \ref{surjectivit\'e}. La fin de la preuve des th\'eor\`emes principaux est l'objet du chapitre \ref{injectivit\'e} et fait encore appel aux r\'esultats du chapitre \ref{kirillovcolmez}. Enfin, la section \ref{complements} contient quelques corollaires et une question. 
   
\subsection{Remerciements.}
Nous tenons \`a remercier chaleureusement Christophe Breuil, qui nous a expliqu\'e sa conjecture avec Matthias Strauch et a suivi nos progr\`es avec int\'er\^et et attention. Il est \'evident que cet article n'aurait jamais vu le jour sans les articles monumentaux \cite{Cbigone, Emcomp} de Pierre Colmez et Matthew Emerton. Nous remercions 
Pierre Colmez et Laurent Fargues pour des longues et fr\'equentes discussions, ainsi que pour leurs suggestions et encouragements : en particulier la pr\'epublication \cite{Colmezpoids} et une remarque de Fargues ont jou\'e un r\^ole d\'ecisif dans l'\'elaboration de ce travail. Pour des discussions utiles, nous tenons \'egalement \`a remercier Konstantin Ardakov, Elmar Grosse-Kl\"onne, Vincent Pilloni, Peter Scholze, Matthias Strauch, Jared Weinstein, et tout sp\'ecialement Benjamin Schraen. G.D. voudrait remercier l'IHES (en particulier Ahmed Abbes et Benjamin Schraen) et le M.S.R.I. pour les excellentes conditions de travail, ainsi que l'A.N.R Percolator pour le financement.
Nous remercions enfin le rapporteur pour la lecture attentive et pour bon nombre de remarques qui ont permis de clarifier et de corriger certaines assertions.

\section{Notations et conventions} 

\begin{enumerate}

    \item On fixe une cl\^oture alg\'ebrique $\overline{\mathbf{Q}}_p$ de $\qp$ et on note 
    $\mathbf{C}_p$ son compl\'et\'e. Toutes les extensions de $\qp$ consid\'er\'ees dans la suite seront \`a l'int\'erieur de $\mathbf{C}_p$. 
         On note $\qpp$ l'unique extension non ramifi\'ee quadratique de $\qp$, et on note 
    $\breve{\mathbf{Q}}_p$ le compl\'et\'e de l'extension maximale non ramifi\'ee $\qp^{\rm nr}$ de $\qp$ dans $\overline{\mathbf{Q}}_p$. Pour $n\geq 1$, on note $F_n=\qp(\mu_{p^n})$ et $\qp^{\rm cyc}$ le compl\'et\'e de la r\'eunion des $F_n$. On pose, enfin, 
     $$\mathcal{G}_{\qp}=\mathrm{Gal}(\overline{\qp}/\qp),\quad \mathcal{H}_{\qp}= \mathrm{Gal}(\overline{\qp}/\qp^{\rm cyc}),\quad \Gamma={\rm Gal}(\qp^{\rm cyc}/\qp).$$

\item Dans tout le texte, $G={\rm GL}_2(\qp)$. On note $G_0={\rm GL}_2(\zp)$ et $G_n=1+p^n {\rm M}_2(\zp)$, pour tout $n\geq 1$. 
On note aussi $B$ le sous-groupe de Borel (sup\'erieur) de $G$ et $P=\left(\begin{smallmatrix} \qpet & \qp \\0 & 1\end{smallmatrix}\right)$
le {\it{mirabolique}}. Enfin, on pose $\mathfrak{g}=\mathfrak{gl}_2={\rm Lie}(G)$ et on consid\`ere la base de 
$\mathfrak{g}$ donn\'ee par 
$$a^+=\left(\begin{smallmatrix} 1 & 0 \\ 0 & 0\end{smallmatrix}\right), \quad a^-=\left(\begin{smallmatrix} 0 & 0 \\ 0 & 1\end{smallmatrix}\right), \quad
u^+=\left(\begin{smallmatrix} 0 & 1 \\ 0 & 0\end{smallmatrix}\right), \quad u^-=\left(\begin{smallmatrix} 0 & 0 \\ 1 & 0\end{smallmatrix}\right).$$
On note $h=a^+-a^-\in U(\mathfrak{g})$. 

\item On fait la convention importante que {\it{toutes les actions de $G$ seront \`a gauche}}. En particulier, si 
$G$ agit sur un espace rigide analytique (ou un sch\'ema) $X$, on transforme l'action naturelle 
\`a droite de $G$ sur les sections globales d'un fibr\'e $G$-\'equivariant sur $X$ en une action \`a gauche (via l'anti-involution 
$g\to g^{-1}$ de $G$). 

\item On note $D$ l'unique alg\`ebre de quaternions non d\'eploy\'ee sur $\qp$ (\`a isomorphisme pr\`es), $\O_D$ son unique ordre maximal, $D_0=\O_D^*$, et $D_n= 1+p^n\O_D$, pour tout $n \geq 1$.
Fixons, une fois pour toutes, un plongement de $\mathbf{Q}_{p^2}$ dans $D$. Soit 
  $\varpi_D$ une uniformisante de $\O_D$ telle que 
\[ \O_D=\mathbf{Z}_{p^2}[\varpi_D], \quad \varpi_D^2=p, \quad \text{et} \quad\varpi_D x= \sigma(x) \varpi_D, \quad  \forall x \in \qpp, \]
o\`u $\sigma$ est le Frobenius de $\mathbf{Q}_{p^2}$.

\item La repr\'esentation supercuspidale $\pi$, de caract\`ere central trivial et d\'efinie sur une extension finie 
$L$ de $\qp$ (qui grandira selon les besoins...) sera fix\'ee une fois pour toutes.  
La repr\'esentation irr\'eductible de $D^*$ attach\'ee \`a $\pi$ par la correspondance de Jacquet-Langlands locale sera not\'ee
$\rho={\rm JL}(\pi)$. Noter que $\rho$ est aussi \`a caract\`ere central trivial. 

\item Si 
$X$ est un espace rigide analytique sur $\qp$, on note $\O(X)=L\otimes_{\qp} H^0(X, \O_X)$ et 
$\Omega^1(X)=L\otimes_{\qp} H^0(X, \Omega^1_{X})$. M\^eme convention pour 
$H^1_{\rm dR}(X)$ (on n'utilisera la cohomologie de de Rham que pour des espaces rigides analytiques lisses et Stein).

\item Soit $X$ est une vari\'et\'e analytique $p$-adique avec action localement analytique de $G$.  
    Si $V$ est un $\zp$-module topologique, on note $\con(X, V)$ 
   (resp. ${\rm LC}(X,V)$) l'ensemble des fonctions continues (resp. localement constantes) $\phi: X\to V$. On note simplement 
   $\con(X)$ (resp. ${\rm LC}(X)$, resp. ${\rm LA}(X)$) au lieu de $\con(X,L)$ (resp. ${\rm LC}(X,L)$, resp. les fonctions localement analytiques 
   $\phi: X\to L$). Tous ces espaces de fonctions sont munis d'actions naturelles (\`a gauche) de $G$. 
    
 \item Si $H$ est un groupe de Lie $p$-adique, on \'ecrit $D(H)$ pour l'alg\`ebre des distributions sur 
    $H$ \`a valeurs dans $L$. C'est le dual topologique fort de ${\rm LA}(H)$. 
        
\item Si $\mathcal{B}$ est une $L$-repr\'esentation de Banach de $G$, on note 
$\mathcal{B}^{\rm an}$ (resp. $\mathcal{B}^{\rm alg}$, resp. $\mathcal{B}^{\rm lisse}$) le sous-espace de 
$\mathcal{B}$ form\'e des vecteurs $v$ tels que l'application $G\to \mathcal{B}$, $g\mapsto g.v$ soit localement analytique
(resp. localement polynomiale, resp. localement constante). 
Si 
 ${\rm Alg}(G)$ est l'ensemble des (classes d'isomorphisme de) repr\'esentations alg\'ebriques irr\'eductibles de $G$ d\'efinies sur $L$, alors 
$\mathcal{B}_{\rm alg}$ est l'image du morphisme naturel (injectif)
$$\bigoplus_{W\in {\rm Alg}(G)} W\otimes_L{\rm Hom}_{\mathfrak{g}}(W, \mathcal{B}^{\rm an})\to \mathcal{B}^{\rm an},$$
o\`u $\mathfrak{g}={\rm Lie}(G)$, alors que $\mathcal{B}^{\rm lisse}$ est le sous-espace 
$(\mathcal{B}^{\rm an})^ {\mathfrak{g}=0}$ de $\mathcal{B}^{\rm an}$ des vecteurs tu\'es par tout \'el\'ement de 
$\mathfrak{g}$. Enfin, 
si $K_p$ est un sous-groupe ouvert compact de $G$, on note $\mathcal{B}_{K_p-\rm alg}$ l'image du morphisme 
 $$\bigoplus_{W\in {\rm Alg}(G)} W\otimes_{L} {\rm Hom}_{K_p} (W, \mathcal{B})\to \mathcal{B}.$$
 
 \item Soit $K'/K$ une extension finie galoisienne ($K$ et $K'$ \'etant des extensions finies de $\qp$), et soit $L$ une extension finie de $\qp$ telle que $|{\rm Hom}_{\qp}(K_0', L)|=[K_0':\qp]$, o\`u $K_0'$ d\'esigne l'extension maximale non ramifi\'ee de $\qp$ dans $K'$. Soit ${\rm WD}_{K'/K}$ la cat\'egorie des repr\'esentations de Weil-Deligne $(r, N, V)$ de $W(\overline{\qp}/K)$ sur des $L$-espaces vectoriels de dimension finie, telles que $r$ soit non ramifi\'ee en restriction \`a $W(\overline{\qp}/K')$. Fontaine \cite{per adiques} (voir aussi \cite[Ch. 4]{Breuil-Schneider} pour un r\'esum\'e) a d\'efini un foncteur 
 de la cat\'egorie des $(\varphi, N, {\rm Gal}(K'/K), L)$-modules\footnote{Ce sont donc des $K_0'\otimes_{\qp} L$-modules libres de type fini avec un Frobenius bijectif semi-lin\'eaire, un op\'erateur lin\'eaire (nilpotent) $N$ satisfaisant 
 $N\varphi=p\varphi N$ et une action semi-lin\'eaire (par rapport \`a $K_0'$) de ${\rm Gal}(K'/K)$, commutant \`a $\varphi$ et $N$.} vers ${\rm WD}_{K'/K}$, qui est une \'equivalence de cat\'egories gr\^ace \`a \cite[Prop. 4.1]{Breuil-Schneider}. Si on combine cela avec la correspondance de Langlands classique, on en d\'eduit une recette attachant \`a des (classes d'isomorphisme de) repr\'esentations lisses irr\'eductibles de ${\rm GL}_n(K)$ des (classes d'isomorphisme de) $(\varphi, N, {\rm Gal}(K'/K), L)$-modules (pour certains $K'$). En partant de la supercuspidale $\pi$ de ${\rm GL}_2(\qp)$ fix\'ee et en 
  tensorisant par $\qp^{\rm nr}$ au-dessus de $K_0'$ (pour $K=\qp$ et un choix convenable de $K'$), on en d\'eduit un $(\varphi, \mathcal{G}_{\qp})$-module 
  $M(\pi)$, qui jouera un r\^ole important dans cet article. 
 
     \end{enumerate}
    
\section{Rev\^etements du demi-plan de Drinfeld et fibr\'es vectoriels}\label{cotedrinfeld}

     Nous rappelons dans ce chapitre un certain nombre de r\'esultats relativement standard sur la tour de Drinfeld et nous \'etablissons le caract\`ere localement analytique de l'action de
     $G$ sur $\O(\Sigma_n)$, ainsi que le caract\`ere lisse de la $G$-repr\'esentation $H^1_{\rm dR, c}(\Sigma_n)$. Le lecteur pourra consulter \cite{Boutot-Carayol, Dat, Drinfeld, FGL, RZ} pour plus de d\'etails concernant la tour de Drinfeld.  

\subsection{L'espace de Drinfeld et ses rev\^ etements}

   Soit $S$ un $\breve{\mathbf{Z}}_p$-sch\'ema. Un \textit{$\O_D$-module formel sp\'ecial sur $S$}
   est un groupe formel $p$-divisible $X$ sur $S$, de dimension $2$ et hauteur $4$, muni d'une action de $\O_D$ telle que 
    l'action induite de $\mathbf{Z}_{p^2}$ sur l'alg\`ebre de Lie de $X$ fait de celle-ci un $\O_S \otimes_{\zp} \mathbf{Z}_{p^2}$-module localement libre de rang $1$.
 Il existe une unique classe de $\O_D$-isog\'enie de $\O_D$-modules formels sp\'eciaux sur $\bar{\mathbf{F}}_p$. Fixons un tel $\O_D$-module formel sp\'ecial $\mathbf{X}$. Le foncteur des d\'eformations de $\mathbf{X}$ par quasi-isog\'enies $\O_D$-\'equivariantes\footnote{Il s'agit du foncteur qui \`a $S$ un $\breve{\mathbf{Z}}_p$-sch\'ema sur lequel $p$ est nilpotent associe l'ensemble des classes d'isomorphisme de couples $(X,\rho)$, avec $X$ un $\O_D$-module formel sp\'ecial sur $S$ et $\rho$ une quasi-isog\'enie $\O_D$-\'equivariante entre $\mathbf{X}_{\bar{S}}$ et $X_{\bar{S}}$; ici $\bar{S}$ est le sous-sch\'ema ferm\'e de $S$ d\'efini par $p=0$.} est repr\'esentable \cite{RZ} par un sch\'ema formel $p$-adique sur $\breve{\mathbf{Z}}_p$. On note 
 $\breve{\mathcal{M}}_0$ la fibre g\'en\'erique rigide de ce sch\'ema formel. Un th\'eor\`eme fondamental de Drinfeld \cite{Drinfeld} fournit un isomorphisme 
\[  \breve{\mathcal{M}}_0 \simeq \breve{\Omega}\times \mathbf{Z}, \]
 o\`u $\breve{\Omega}=\Omega\hat{\otimes}_{\qp} \breve{\mathbf{Q}}_p$ et $\Omega$ est le demi-plan de Drinfeld, un espace rigide sur $\qp$ 
 dont les $\cp$-points sont $$\Omega(\cp)=\mathbf{P}^1(\cp)- \mathbf{P}^1(\qp).$$ 
L'espace $ \breve{\mathcal{M}}_0$ est muni d'une action \`a gauche de $G$ donn\'ee par
\[ g.(X,\rho) = (X, \rho \circ g^{-1}), \]
qui correspond par l'isomorphisme de Drinfeld \`a l'action usuelle par homographies de $G$ sur le demi-plan et au d\'ecalage par $-v_p(\det g)$ sur $\mathbf{Z}$, ainsi que d'une donn\'ee de descente \`a la Weil\footnote{Pour m\'emoire, une donn\'ee de descente \`a la Weil sur un sch\'ema $X$ sur $\breve{\mathbf{Z}}_p$ est un isomorphisme de sch\'emas $X \to \sigma_* X=X \otimes_{\breve{\mathbf{Z}}_p,\sigma} \breve{\mathbf{Z}}_p$ sur $\breve{\mathbf{Z}}_p$, o\`u $\sigma$ est le Frobenius.}, qui correspond via l'isomorphisme de Drinfeld 
au compos\'e de la donn\'ee de descente canonique et du d\'ecalage par $1$. Elle n'est donc pas effective, mais pour tout entier 
$t>0$, cette donn\'ee de descente sur le quotient $p^{t\z} \backslash \breve{\mathcal{M}_0}$ de $\breve{\mathcal{M}}_0$ par l'action de l'\'el\'ement $p^t$ du centre de $G$ devient effective.
En prenant $t=1$, on obtient un mod\`ele $\Sigma_0$ de $p^{\z} \backslash \breve{\mathcal{M}_0}$ sur $\qp$.

Soit $X^{\rm un}$ le groupe $p$-divisible rigide universel sur $\breve{\mathcal{M}_0}$. Si $n\geq 1$, on d\'efinit $$\breve{\mathcal{M}}_{n}=X^{\rm un}[p^n]-X^{\rm un}[\varpi_D^{2n-1}].$$ C'est un rev\^ etement \'etale galoisien de $\breve{\mathcal{M}_0}$ de groupe de Galois $\O_D^*/(1+p^n \O_D)$. Une fois encore, son quotient par l'action de $p^{t\z}$ descend \`a $\qp$ pour tout $t>0$. Pour $t=1$, cela fournit un mod\`ele $\Sigma_n$ de $p^{\z} \backslash\breve{\mathcal{M}}_{n}$ sur $\qp$. Il est muni d'actions \`a gauche de $G$ et \`a droite de $D^*$, qui commutent. 

\subsection{Quelques rappels sur les espaces Stein}

 Nous renvoyons le lecteur \`a \cite{GK1, GK2, Kiehl, SS} pour les preuves des r\'esultats \'enonc\'es dans ce paragraphe. 
 
    Soit $K$ un corps de caract\'eristique $0$, complet pour une valuation discr\`ete, et soit 
$X$ un espace rigide Stein sur $K$. Rappelons que cela veut dire que $X$ admet un recouvrement croissant admissible 
$(X_n)_{n\geq 0}$ par des ouverts affino\"ides tels que $\O(X_{n+1})\to \O(X_n)$ soit d'image dense pour tout $n$. 
Dans ce cas, la fl\`eche naturelle 
$\O(X)\to \varprojlim_{n} \O(X_n)$ est un isomorphisme d'espaces vectoriels topologiques, ce qui fait que 
$\O(X)$ est naturellement un $K$-espace de Fr\'echet (c'est en fait une alg\`ebre de Fr\'echet-Stein au sens de Schneider et Teitelbaum \cite{STInv}). Le th\'eor\`eme de Kiehl \cite{Kiehl} montre que les faisceaux coh\'erents sur $X$ n'ont pas de cohomologie 
en d\'egr\'e $>0$. 

   Supposons en outre que $X$ est lisse. D'apr\`es le th\'eor\`eme de Kiehl mentionn\'e ci-dessus,
 la cohomologie de de Rham de $X$ se calcule comme la cohomologie du complexe des sections globales du complexe de de Rham de $X$. 
De plus, les diff\'erentielles $d_X^k : \Omega^k(X) \to \Omega^{k+1}(X)$ sont des morphismes stricts d'image ferm\'ee. En particulier, les groupes de cohomologie de de Rham de $X$ sont des espaces de Fr\'echet (voir \cite[cor.3.2]{GK1} pour tout ceci). Si $d=\dim X$, alors 
$H_c^k(X,\mathcal{F})=0$ pour tout fibr\'e $\mathcal{F}$ sur $X$ et tout $k<d$. On d\'efinit alors $H_{\mathrm{dR,c}}^{k+d}(X)$ comme le $k$-\`eme groupe de cohomologie du complexe
\[ \dots \to H_c^d(X,\Omega^k) \to H_c^d(X,\Omega^{k+1}) \to \dots \]
La dualit\'e de Serre pour les vari\'et\'es Stein \cite{Chiar} montre que ce complexe est dual du complexe des sections globales du complexe de de Rham de $X$, tordu par $\Omega^d(X)$. Cela permet de montrer \cite[th. 4.11]{GK2} que si $X$ est pure de dimension $d$, alors pour tout $k$ on a des isomorphismes canoniques 
$$H_{\mathrm{dR}}^k(X)\simeq H_{\rm dR,c}^{2d-k}(X)^*\quad \text{et}\quad H_{\mathrm{dR,c}}^k(X)=H_{\mathrm{dR}}^{2d-k}(X)^*,$$
les duaux \'etant topologiques (comme toujours dans cet article). 
 La preuve de \cite[cor.3.2]{GK1} montre que pour tout 
    $k$ l'espace vectoriel topologique $H^k_{\rm dR}(X)$ est isomorphe \`a la limite inverse d'une suite $(V_n)_n$ d'espaces de dimension finie sur $K$. En particulier 
    $H^k_{\rm dR}(X)$ est un Fr\'echet r\'eflexif et son dual topologique $H_{\rm dR,c}^{2d-k}(X)$ est la limite inductive des $V_n^*$. On en d\'eduit que $H_{\mathrm{dR}}^k(X)$ est aussi le dual alg\'ebrique de $H_{\rm dR,c}^{2d-k}(X)$. Puisque $\Omega$ est un espace Stein\footnote{Voir la discussion suivant la proposition \ref{Sigma est Stein} pour une explication de ce fait standard.}, il en est de m\^eme de $\Sigma_0$ et puisque 
    $\Sigma_n$ est un rev\^etement \'etale fini de $\Sigma_0$, on obtient la

\begin{proposition}\label{Sigma est Stein}
Pour tout $n\geq 0$, l'espace rigide $\Sigma_n$ est un espace Stein. 
\end{proposition}
Notons $\tau_n$ la compos\'ee du morphisme $\Sigma_n \to \Sigma_0$ et de la r\'etraction de $\Sigma_0$ sur l'arbre de Bruhat-Tits, dont on fixe l'origine en le r\'eseau standard et $B_i$ la boule centr\'ee en l'origine de rayon $i$ dans l'arbre. Alors la famille des $U_i=\tau_n^{-1}(B_i)$ ($n$ est sous-entendu dans la notation) forme un recouvrement de Stein de $\Sigma_n$. De plus, $U_i$ est stable par l'action de $G_i=1+p^iM_2(\zp)$ pour tout $i\geq 1$.

\subsection{Le caract\`ere localement analytique de $\O(\Sigma_n)^*$}

Soit $n\geq 0$ et $k \in \z$. On note $\O(k)(\Sigma_n)$ l'espace des fonctions rigides analytiques sur $\Sigma_n$, muni de l'action de $G$ 
\[ g.f= \frac{1}{(a-cz)^{k}}\cdot (g.f), \quad \text{si}\quad g=\left(\begin{smallmatrix} a & b \\ c & d\end{smallmatrix}\right)\in G\]
o\`u $g.f$ (dans le terme \`a droite) est l'action naturelle de $g\in G$ sur $\O(\Sigma_n)$ d\'eduite de l'action de $G$ sur $\Sigma_n$. On a utilis\'e la structure naturelle de 
$\O(\Omega)$-module de $\O(\Sigma_n)$ pour donner un sens \`a la multiplication par $\frac{1}{(a-cz)^{k}}\in \O(\Omega)$. 
Comme $\Sigma_n$ est Stein, $\O(k)(\Sigma_n)$ est l'espace des sections globales d'un fibr\'e $G$-\'equivariant sur $\Sigma_n$, not\'e $\O(k)$. 

Notons qu'en tout niveau le faisceau $\O(2) \otimes \det$ est simplement le faisceau des diff\'erentielles $\Omega^1$ : on le voit facilement en niveau $0$ et en niveau plus grand le faisceau $\Omega^1$ est simplement le tir\'e en arri\`ere du faisceau $\Omega^1$ en niveau $0$, puisque le rev\^ etement est \'etale.

            \begin{theoreme}\label{agit bien} 
              L'action de $G$ sur $\O(k)(\Sigma_n)^*$ est localement analytique, pour tout $k\in \z$. De mani\`ere \'equivalente, 
              $\O(k)(\Sigma_n)$ est un
              $D(G)$-module s\'epar\'ement continu pour tout $k$. 
              En particulier, $H_{\mathrm{dR}}^1(\Sigma_n)$ est un $D(G)$-module s\'epar\'ement continu.  
                        \end{theoreme}
                       
             \begin{proof} Il suffit de d\'emontrer que $\O(\Sigma_n)$ est un $D(G)$-module s\'epar\'ement continu, le reste s'en d\'eduit facilement.
             
             Soit $g_1,...,g_4$ une famille g\'en\'eratrice minimale de $G_2=1+p^2M_2(\zp)$. Notons pour $\alpha=(\alpha_1,...,\alpha_4)\in \mathbf{N}^4$
             $$b^{\alpha}=(g_1-1)^{\alpha_1}...(g_4-1)^{\alpha_4}\in \zp[G_2].$$
             Les \'el\'ements de $D(G_2)$ (l'alg\`ebre de distributions sur $G_2$ \`a valeurs dans $L$) s'\'ecrivent de mani\`ere unique 
             $$\lambda=\sum_{\alpha\in\mathbf{N}^4} a_{\alpha} b^{\alpha}$$
             avec $a_{\alpha}\in L$ et $\lim_{\alpha\to\infty} v_p(\alpha)+r|\alpha|=\infty$ pour tout 
             $r>0$, o\`u $|\alpha|=\alpha_1+...+\alpha_4$. On veut montrer que 
             $a_{\alpha} b^{\alpha} f$ tend vers $0$ dans $\O(\Sigma_n)$ pour toute telle suite 
             $(a_{\alpha})_{\alpha}$ et tout $f\in \O(\Sigma_n)$. 
             Consid\'erons le recouvrement de Stein $(U_i)_{i\geq 0}$ de $\Sigma_n$ introduit apr\`es la proposition \ref{Sigma est Stein}, chaque 
       $U_i$ \'etant stable sous l'action de $G_i$. Comme $\O(\Sigma_n)=\varprojlim_{i} \O(U_i)$, il suffit de d\'emontrer              
             que pour chaque $i$ fix\'e $a_{\alpha} b^{\alpha}f$ tend vers $0$ dans $\O(U_i)$ pour tout $f\in \O(\Sigma_n)$ (noter que 
             $b^{\alpha} f \in \O(\Sigma_n)\subset \O(U_i)$ pour tout $\alpha$). Comme 
             $$D(G_2)=\bigoplus_{g\in G_i\setminus G_2} D(G_i)\delta_{g},$$
             et $G_i\setminus G_2$ est fini 
             on se ram\`ene (en travaillant s\'epar\'ement avec chaque $\delta_g f=g.f$) \`a d\'emontrer que $\O(U_i)$ est un $D(G_i)$-module topologique (ce qui a un sens, puisque $G_i$ agit sur $U_i$).  
             
                Soit $b_i^{\alpha}$ l'analogue de $b^{\alpha}$ pour $G_i$ (i.e. $b_{i}^{\alpha}=(g_{i,1}-1)^{\alpha_1}...(g_{i,4}-1)^{\alpha_4}$, o\`u 
                $g_{i,j}$ forment une famille g\'en\'eratrice minimale de $G_i$) et soit $D_h(G_i)$ le compl\'et\'e de 
                $D(G_i)$ pour $$||\lambda:=\sum_{\alpha} a_{\alpha} b_i^{\alpha}||_h=\sup_{\alpha} |a_{\alpha}|p^{-\frac{|\alpha|}{p^h}}.$$
                Nous allons montrer que l'on peut trouver $h$ tel que l'action de $G_i$ sur $\O(U_i)$ s'\'etende en une structure de 
                $D_h(G_i)$-module topologique, ce qui suffira pour conclure. 
                
                Soit $\mathfrak{g}={\rm Lie}(G_i)$ et soit $\mathfrak{X}_1,...,\mathfrak{X_4}$ une base de 
                $\mathfrak{g}$. On note $\mathfrak{X}^{\alpha}=\mathfrak{X}_1^{\alpha_1}...\mathfrak{X}_4^{\alpha_4}\in U(\mathfrak{g})$. D'apr\`es un r\'esultat de Frommer \cite[1.4, lemma 3, corollaries 1, 2, 3]{Frommer}, 
                $D_h(G_i)$ est un module libre de type fini (\`a gauche et \`a droite) sur l'adh\'erence 
                $U_h(\mathfrak{g})$ de $U(\mathfrak{g})$ dans $D_h(G_i)$, et a une base form\'ee d'\'el\'ements de $\zp[G_i]$. De plus, 
               les \'el\'ements de $U_h(\mathfrak{g})$ s'\'ecrivent de mani\`ere unique 
               $\lambda=\sum_{\alpha} a_{\alpha} \mathfrak{X}^{\alpha}$, avec 
               $v_p(a_{\alpha})-c_h|\alpha|\to\infty$, o\`u $c_h>p^{h-1}$ ne d\'epend que de 
               $h$. 
               
               \begin{lemme}
                  L'action de $G_i$ sur $\O(U_i)$ est diff\'erentiable: pour tout $\mathfrak{X}\in \mathfrak{g}$ et 
                  $f\in \O(U_i)$ la limite 
                  $$\mathfrak{X}.f=\lim_{n\to \infty} \frac{e^{p^n\mathfrak{X}}.f-f}{p^n}$$
                  existe dans $\O(U_i)$. Cela munit $\O(U_i)$ d'une structure de $\mathfrak{g}$-module, et 
                  $f\mapsto \mathfrak{X}.f$ sont des endomorphismes continus du $L$-Banach $\O(U_i)$ (muni de la norme spectrale).
                                 \end{lemme}
               
               \begin{proof} Le morphisme \'etale fini $\Sigma_n\to \Sigma_0$ induit un morphisme \'etale fini $G_i$-\'equivariant
               $\pi: U_i\to U_{0,i}$, o\`u $U_{0,i}$ est l'analogue de $U_i$ en niveau $0$. Le lemme pr\'ec\'edent se v\'erifie sans aucun mal sur 
               $\O(U_{0,i})$, qui est donc muni d'une structure de $\mathfrak{g}$-module. Notons $\partial_{\mathfrak{X}}$ la d\'erivation continue de 
               $\O(U_{0,i})$ attach\'ee \`a $\mathfrak{X}\in \mathfrak{g}$. Comme $\pi: U_i\to U_{0,i}$ est fini \'etale, cette d\'erivation s'\'etend de mani\`ere unique en une d\'erivation continue de $\O(U_i)$, que nous notons encore $\partial_{\mathfrak{X}}$. Nous allons montrer que pour tout $f\in \O(U_i)$
               $$\lim_{n\to \infty} \frac{e^{p^n\mathfrak{X}}.f-f}{p^n}=\partial_{\mathfrak{X}}f,$$ 
     le lemme s'en d\'eduisant sans mal. Notons $g_n=e^{p^n\mathfrak{X}}\in G_i$, $f_n=g_n.f\in \O(U_i)$, et soit 
     $P=X^d+a_{d-1}X^{d-1}+...+a_0$ le polyn\^ome minimal de $f$ sur $\O(U_{0,i})$. 
     Comme $g_n.P(f_n)=g_n.(P(f))=0$, on obtient 
     $$ P(f_n)-P(f) + (g_n.P-P)(f_n) = 0$$
     On divise cette relation par $p^n$ et on fait $n\to\infty$, en utilisant le fait que 
     $\lim_{n\to\infty} \frac{g_n.P-P}{p^n}=\partial_{\mathfrak{X}}(P)$ puisque les $a_i$ sont dans $\O(U_{0,i})$ et 
     $\lim_{n\to\infty} f_n=f$ (cela d\'ecoule de la continuit\'e de l'action de $G_i$, qui se d\'eduit gr\^ace au th\'eor\`eme d'Elkik - voir \cite[lemma 2.5]{Scholze} - de l'\'enonc\'e analogue en niveau $0$ et du fait que 
     $\Sigma_n\to \Sigma_0$ est un rev\^etement fini \'etale). 
     Comme de plus 
     $$\lim_{n\to\infty} \frac{P(f_n) - P(f)}{f_n-f}=P'(f)\ne 0,$$
     on en d\'eduit que $\lim_{n\to\infty} \frac{f_n-f}{p^n}$ existe et 
     $$P'(f)\lim_{n\to\infty} \frac{f_n-f}{p^n}+\partial_{\mathfrak{X}}(P) f=0.$$ 
     Mais en appliquant $\partial_{\mathfrak{X}}$ \`a la relation $P(f)=0$ on obtient aussi 
     $$P'(f)\partial_{\mathfrak{X}}(f)+ \partial_{\mathfrak{X}}(P)(f)=0.$$
     En comparant les deux formules et en utilisant le fait que $P'(f)\ne 0$ (puisque le rev\^etement est \'etale), le r\'esultat s'en d\'eduit. 
           \end{proof}
               
             Revenons \`a la preuve du th\'eor\`eme \ref{agit bien}. on dispose de 
             quatre connexions continues $\partial_{\mathfrak{X}_j}$ sur le Banach 
             $\O(U_i)$. Elles sont toutes $C$-lipschitziennes pour un certain 
             $C>0$. Ainsi, la norme de l'op\'erateur continu $\mathfrak{X}^{\alpha}$ sur $\O(U_i)$
             (d\'eduit de la structure de 
              $U(\mathfrak{g})$-module donn\'ee par le lemme) est major\'ee par $C^{\alpha}$ pour tout 
              $\alpha\in\mathbf{N}^4$. Donc, si $p^{h/2}>C$, alors 
              $\sum_{\alpha} a_{\alpha} \mathfrak{X}^{\alpha}$ converge faiblement dans 
              $\O(U_i)$, ce qui permet de conclure que pour $h$ assez grand $D_h(G_i)$ agit continument sur 
              $\O(U_i)$. Cela permet de conclure. 
                                                     \end{proof}

\begin{remarque}
  Le m\^eme argument s'applique aux rev\^etements du demi-espace de Drinfeld en toute dimension, en utilisant \cite{boundary}, Prop. 1. 
\end{remarque}

\subsection{Num\'erologie et lissit\'e de $H_{\mathrm{dR,c}}^1(\Sigma_n)$} Rappelons que l'on utilise la base 
$$a^+=\left(\begin{smallmatrix} 1 & 0 \\ 0 & 0\end{smallmatrix}\right), \quad a^-=\left(\begin{smallmatrix} 0 & 0 \\ 0 & 1\end{smallmatrix}\right), \quad
u^+=\left(\begin{smallmatrix} 0 & 1 \\ 0 & 0\end{smallmatrix}\right), \quad u^-=\left(\begin{smallmatrix} 0 & 0 \\ 1 & 0\end{smallmatrix}\right).$$
de $\mathfrak{g}={\rm Lie}(G)$. Soit $z$ \og la\fg{} coordonn\'ee sur $\Omega$. La trivialisation $\Omega^1(\Sigma_n)=\O(\Sigma_n)dz$ induit une application 
$$\frac{d}{dz}: \O(\Sigma_n)\to \O(\Sigma_n), \quad df=\frac{df}{dz} dz\quad \forall f\in \O(\Sigma_n).$$
     L'\'enonc\'e peu app\'etissant suivant sera utilis\'e tr\`es souvent dans la suite.  
     
 \begin{lemme}\label{numerologie Lie} Soit $\partial: \O(\Sigma_n)\to \O(\Sigma_n)$ l'op\'erateur\footnote{Le lecteur trouvera certainement \'etrange d'appeler cet op\'erateur
$\partial$. Nous verrons plus loin qu'il est reli\'e \`a une connexion $\partial$ sur les $(\varphi,\Gamma)$-modules.} de multiplication par 
$z$. On a les formules suivantes pour l'action de l'alg\`ebre de Lie:

a) Sur $\O(k)(\Sigma_n)$ 
   $$ a^+=u^+\partial+1-k, \quad a^-=-\partial u^+, \quad u^-=-\partial^2 u^++k\partial.$$
   
 b) Sur  $\O(k)(\Sigma_n)^*$ 
  $$a^+=\partial u^++k-1, \quad a^-=-u^+\partial, \quad u^-=-k\partial-u^+\partial^2.$$
  
 c) Sur $\Omega^1(\Sigma_n)^*$
 $$a^+=\partial u^+, \quad a^-=-\partial u^+, \quad u^-=-\partial^2 u^+.$$
 
 De plus, on a $\partial u^+-u^+\partial=1$ sur tous ces espaces, et  
 $u^+=-\frac{d}{dz}$ sur $\O(\Sigma_n)$. 
 
\end{lemme}

\begin{proof} L'\'egalit\'e $\partial u^+-u^+\partial=1$ est imm\'ediate, car $u^+$ est une connexion telle que 
$u^+(z)=-1$. Le b) d\'ecoule directement de a) (ne pas oublier que $\langle Xl, v\rangle=-\langle l, Xv\rangle$, si 
$X\in \mathfrak{gl}_2$, $l\in \O(k)(\Sigma_n)^*$ et $v\in \O(k)(\Sigma_n)$). Le c) d\'ecoule de b), du fait que 
$\Omega^1(\Sigma_n)^*\simeq \O(2)(\Sigma_n)^*\otimes \det^{-1}$ et de l'\'egalit\'e 
$2\partial+u^+\partial^2=\partial^2 u^+$ (appliquer deux fois l'identit\'e $\partial u^+-u^+\partial=1$). 

Il reste donc \`a d\'emontrer le a), et il suffit de le faire pour $k=0$. 
Si $n=0$, il s'agit d'un exercice amusant\footnote{ou pas... Se rappeler que l'action de 
$G$ sur $\O(\Omega)$ est donn\'ee par 
$$\left(\left(\begin{smallmatrix} a & b \\c & d\end{smallmatrix}\right)f\right) (z)=f\left(\frac{dz-b}{a-cz}\right).$$} laiss\'e au lecteur.
Dans le cas g\'en\'eral, on utilise le fait que $\Sigma_n$ est un rev\^ etement \'etale de $\Sigma_0$. Toutes les \'egalit\'es ci-dessus sont des \'egalit\'es entre des d\'erivations continues
sur $\O(\Sigma_n)$, qui sont valides sur $\O(\Sigma_0)$; elles sont donc valables sur $\O(\Sigma_n)$ tout entier. L'\'egalit\'e  $u^+=-\frac{d}{dz}$ sur $\O(\Sigma_n)$ 
se d\'emontre de la m\^eme mani\`ere.
\end{proof}

   Une cons\'equence facile du lemme pr\'ec\'edent est la proposition suivante. Nous remercions Pierre Colmez pour l'argument simple et \'el\'egant ci-dessous.

\begin{proposition}\label{lissit\'e}
La $G$-repr\'esentation $H_{\mathrm{dR,c}}^1(\Sigma_n)$ est lisse.
\end{proposition}

\begin{proof}
  Puisque $u^+=-\frac{d}{dz}$ sur $\O(\Sigma_n)$ et que 
  $H_{\mathrm{dR,c}}^1(\Sigma_n)$ est le dual de $\Omega^1(\Sigma_n)/d(\O(\Sigma_n))$, on voit qu'il suffit de montrer la lissit\'e de
  $(\Omega^1(\Sigma_n)^*)^{u^+=0}$. Mais $\Omega^1(\Sigma_n)^*$ est une $G$-repr\'esentation localement analytique 
  (th\'eor\`eme \ref{agit bien}) et le point c) du lemme pr\'ec\'edent montre que tout \'el\'ement de $\Omega^1(\Sigma_n)^*$
  tu\'e par $u^+$ est en fait tu\'e par $\mathfrak{gl}_2$, et donc il est lisse. Cela permet de conclure. 
\end{proof}

\begin{remarque}  Comme nous l'a fait remarquer le rapporteur, si un groupe de Lie $p$-adique $G$ agit de mani\`ere localement analytique sur une courbe analytique lisse Stein $X$, alors $H^1_{\rm dR, c}(X)$ est un $G$-module lisse. Le point crucial est l'identit\'e 
$$(\mathfrak{X}.f)dg=(\mathfrak{X}.g)df$$
valable pour des fonctions analytiques $f,g$ sur $X$ et pour $\mathfrak{X}\in {\rm Lie}(G)$. Elle se d\'emontre en calculant les deux c\^ot\'es explicitement en termes d'une coordonn\'ee locale $z$ (les deux termes valent alors $\frac{\partial f}{\partial z}\frac{\partial g}{\partial z}  (\mathfrak{X}.z)dz$, comme le montre un calcul direct). Cela permet d'obtenir l'identit\'e 
$$\mathfrak{X}.(fdg)=d(f\cdot (\mathfrak{X}.g)),$$
qui montre que $\mathfrak{X}.\Omega^1(X)\subset d\mathcal{O}(X)$ et permet de conclure. 
\end{remarque}

\section{Uniformisation $p$-adique et cohomologie de de Rham}\label{calculdeRham}

 On fixe dans la suite un entier $n$ tel que 
 $\rho={\rm JL}(\pi)$ se factorise par $D^*/(1+p^n\O_D)$. 
  Le but de cette section est de d\'emontrer le th\'eor\`eme suivant. Sa d\'emonstration devrait s'adapter assez facilement \`a d'autres situations.
 
\begin{theoreme}\label{derham}
 On a \[ \dim_L\ho_G(H_{\rm dR,c}^1(\Sigma_n)^{\rho^{\vee}},\pi)=2.\]
 De mani\`ere \'equivalente, 
 \[\dim_L \ho_G( \pi^{*}, H^1_{\rm dR}(\Sigma_n)^{\rho})=2.\]
\end{theoreme}

 La preuve de ce r\'esultat se fait par voie globale : elle utilise le th\'eor\`eme d'uniformisation de Cerednik-Drinfeld des courbes de Shimura par les rev\^ etements de Drinfeld et est directement inspir\'ee du calcul de la cohomologie $\ell$-adique de certains espaces de Rapoport-Zink par Harris \cite{Harris Inv} et Fargues \cite{Fargues these}. 
Nous allons voir plus loin que l'on a en fait un isomorphisme\footnote{On en construira en fait un raisonnablement canonique.}
$$H^1_{\rm dR}(\Sigma_n)^{\rho}\simeq \pi^*\oplus \pi^*,$$
mais cela demande plus de travail: l'argument global utilis\'e permet de d\'emontrer facilement qu'il n'y a pas d'autre repr\'esentation de la s\'erie discr\`ete (ainsi que beaucoup de s\'eries principales) parmi les sous-quotients de $ H_{\rm dR,c}^1(\Sigma_n)^{\rho^{\vee}}$, mais il ne semble pas facile d'exclure la pr\'esence de n'importe quelle s\'erie principale avec ce genre d'argument (dans le cas 
$\ell$-adique, ce genre de difficult\'e est contourn\'e en utilisant l'isomorphisme de Faltings-Fargues \cite{faltings, FGL}; cela semble plus d\'elicat dans notre situation, mais c'est effectivement ce qui est fait dans \cite{CDN}, o\`u le r\'esultat est d\'emontr\'e avec ${\rm GL}_2(\qp)$ remplac\'e par ${\rm GL}_2(F)$, avec $F$ une extension finie quelconque de $\qp$). 

\begin{remarque} 
a) La cohomologie de de Rham de $\Omega$ vue comme repr\'esentation de $G$ (m\^eme en dimension quelconque) a \'et\'e calcul\'ee par Schneider et Stuhler, de Shalit, Orlik, par des m\'ethodes diverses et vari\'ees, cf. \cite{SS, Shalit, Orlik}. Leur preuve ne s'adapte pas en niveau sup\'erieur, mais a la vertu de ne pas faire appel \`a la th\'eorie automorphe.

b) Pour le rev\^etement de la tour de Drinfeld de niveau $1+\varpi_D \O_D$, le calcul de la cohomologie de de Rham est nettement plus simple 
gr\^ace \`a l'existence d'un mod\`ele formel explicite \cite{teitel} (voir \cite{luepan} pour les d\'etails). En particulier, pour ce rev\^etement le th\'eor\`eme pr\'ec\'edent admet une preuve purement locale, 
qui fournit d'ailleurs un isomorphisme $H^1_{\rm dR}(\Sigma_n)^{\rho}\simeq \pi^*\oplus \pi^*$. 
\end{remarque}

\subsection{Formes modulaires quaternioniques classiques et $p$-adiques}\label{formes} 

   Soit $\ob$ une alg\`ebre de quaternions sur $\mathbf{Q}$, non ramifi\'ee en $p$ et ramifi\'ee \`a l'infini. On regarde
   $\ob^*$ comme un groupe alg\'ebrique sur $\mathbf{Q}$ (donc $\ob^*(R)=(R\otimes_{\mathbf{Q}} \ob)^*$ pour toute
   $\mathbf{Q}$-alg\`ebre $R$).
      On fixe une identification
   $\ob^*(\qp)\simeq G$, ainsi qu'un sous-groupe ouvert compact 
   $K^p$ de $\ob^*(\mathbf{A}_f^p)$. On suppose que 
   $K^p=\prod_{\ell\ne p} K_{\ell}$, o\`u $K_{\ell}$ est un sous-groupe ouvert compact de 
   $\ob^*(\mathbf{Q}_{\ell})$, et on suppose qu'il existe au moins un premier 
   $\ell_0$ tel que $K_{\ell_0}$ soit sans torsion. 
   
   \begin{definition} On note 
   $$X=X(K^p)=\ob^*(\mathbf{Q})\backslash \ob^*(\mathbf{A}_f)/K^p$$
   et, si $K_p$ est un sous-groupe ouvert compact de $G$, on note 
      $$X(K_p)=X/K_p=\ob^*(\mathbf{Q})\backslash \ob^*(\mathbf{A}_f)/K^pK_p.$$
\end{definition}

    L'espace $X$ est la limite 
 projective des ensembles finis 
   $X(K_p)$. Il est muni d'une action naturelle (\`a droite) de $G$. L'espace $\con(X)$ est l'espace des formes automorphes $p$-adiques pour le groupe $\ob^*$. On va voir tout de suite que c'est un espace vectoriel topologique tout \`a fait raisonnable. Le r\'esultat suivant est standard et nous laissons sa preuve au lecteur. 
   
      \begin{lemme}\label{union disjointe}
     L'ensemble $\bar{B}^*(\q)\backslash \bar{B}^*(\mathbf{A}_f^p)/K^p$ est fini. Ecrivons 
   \[ \bar{B}^*(\mathbf{A}_f^p)/K^p= \coprod_{i=1}^r \bar{B}^*(\q) y_i, \]
   et notons 
  $\Gamma_i$ le stabilisateur de $y_i$ dans $\bar{B}^*(\q)$, vu comme un sous-groupe de $G\simeq \bar{B}^*(\qp)$. Alors 
  $\Gamma_i$ est discret cocompact dans $G$ et on a un isomorphisme de $G$-modules topologiques 
  $$X=\coprod_{i=1}^r \Gamma_i\backslash G.$$
   \end{lemme}
   
    En particulier 
   $X$ est une vari\'et\'e analytique $p$-adique compacte et l'action de $G$ y est localement analytique.
\\

Venons-en maintenant au lien avec les formes automorphes classiques (l'argument est tout \`a fait similaire \`a celui de \cite[\S\ 1]{taylor}). Rappelons que 
${\rm Alg}(G)$ d\'esigne l'ensemble des (classes d'isomorphisme de) $L$-repr\'esentations alg\'ebriques irr\'eductibles de 
$G$. Fixons $W\in {\rm Alg}(G)$ et notons $\mathcal{A}_{K_p}(W)$ l'espace des fonctions continues $f: X \to W$ telles que 
$f(xk)=k^{-1}.f(x)$ pour tous $x\in X$ et $k\in K_p$ (on peut y penser comme l'espace des formes automorphes $p$-adiques de poids $W$ et de niveau 
$K^pK_p$).  

\begin{lemme}\label{isom 1}
 Il existe un isomorphisme canonique 
   $$\ho_{K_p}(W^*,\con(X))\simeq \mathcal{A}_{K_p}(W).$$
\end{lemme}

\begin{proof}
  On a des isomorphismes canoniques
  $$\ho_{K_p}(W^*,\con(X))\simeq ((W^*)^*\otimes_{L} \con(X))^{K_p}\simeq (W\otimes_{L} \con(X))^{K_p}\simeq \con(X,W)^{K_p}.$$
  En suivant les actions de $K_p$, on voit que le dernier espace est pr\'ecis\'ement $\mathcal{A}_{K_p}(W)$, ce qui permet de conclure (notons que l'inverse de cet isomorphisme envoie simplement 
  $f\in \mathcal{A}_{K_p}(W)$ sur $\phi_f\in \ho_{K_p}(W^*,\con(X))$ d\'efini par $\phi_f(l)(x)=l(f(x))$ si $l\in W^*$ et $x\in X$).
\end{proof}

Fixons une fois pour toutes un isomorphisme $\iota: \overline{\qp}\simeq \mathbf{C}$, ce qui induit un plongement 
$L\to \mathbf{C}$ (en se rappelant que l'on a fix\'e un plongement de $L$ dans $\overline{\qp}$). 
Alors $W_{\infty}=W\otimes_{L} \mathbf{C}$ devient, via l'isomorphisme $\iota$, une $\mathbf{C}$-repr\'esentation de 
$\bar{B}^*(\mathbf{C})$, et donc de $\bar{B}^*(\mathbf{R})$ aussi. 

  On note ${\rm Aut}(K^pK_p)$ l'espace des fonctions lisses 
  \`a valeurs complexes sur $\bar{B}^*\backslash \bar{B}^*(\mathbf{A})/K_pK^p$ (ce sont les formes automorphes classiques de $\bar{B}^*(\mathbf{A})$, de niveau $K^pK_p$).
 Cet espace admet une action naturelle de $\bar{B}^*(\mathbf{R})$, via l'action \`a droite de 
ce groupe sur $\bar{B}^*\backslash \bar{B}^*(\mathbf{A})/K_pK^p$. 

\begin{lemme} \label{isom 2}
   On a un isomorphisme canonique 
   $$\mathcal{A}_{K_p}(W)\otimes_{L, \iota}\mathbf{C}\simeq \ho_{\bar{B}^*(\mathbf{R})}(W_{\infty}^*, {\rm Aut}(K^pK_p)).$$
\end{lemme}

\begin{proof}
  Nous allons nous contenter de d\'ecrire la fl\`eche en question. Si $f\in \mathcal{A}_{K_p}(W)$, on l'envoie sur 
  $\phi_f\in \ho_{\bar{B}^*(\mathbf{R})}(W_{\infty}^*, {\rm Aut}(K^pK_p))$ d\'efini par 
  $$\phi_f(l)(g)=l( g_{\infty}^{-1}.(g_p.f(g^{\infty})\otimes 1)),$$
  si $g=(g_q)_{q}=g_{\infty}\cdot g^{\infty}\in \bar{B}^*(\mathbf{A})$ et $l\in W_{\infty}^*$. Un exercice standard montre que 
  $f\mapsto \phi_f$ est un isomorphisme. 
\end{proof}

    
      La discussion pr\'ec\'edente entra\^ine alors directement le r\'esultat suivant (il suffit de remplacer $W$ par $W^*$ dans ce qui pr\'ec\`ede).

     \begin{lemme}\label{formes automorphes}  
      Soit $W\in {\rm Alg}(G)$ et $\iota: \overline{\qp}\simeq \mathbf{C}$. On a un isomorphisme canonique 
    \[ \ho_{K_p}(W,\con(X))\otimes_{L,\iota}\mathbf{C}\simeq  \bigoplus_{\pi} ~\pi_f^{K_pK^p}, \]
    la somme directe portant sur les repr\'esentations automorphes $\pi=\pi_{\infty}\otimes \pi_f$ de $\bar{B}^*(\mathbf{A})$ 
    telles que $\pi_{\infty}\simeq W_{\infty}$. Ainsi,
\[ \con(X)_{K_p-\rm alg}\otimes_{L,\iota}\mathbf{C}\simeq \bigoplus_{W\in {\rm Alg}(G)} \bigoplus_{\pi_{\infty}\simeq W_{\infty}} W_{\infty} \otimes \pi_f^{K_p}. \]

\end{lemme}

Rappelons que si $B$ est une alg\`ebre de quaternions sur $\q$, dont on note $S(B)$ l'ensemble (fini) des places de ramification, la correspondance de Jacquet-Langlands globale met en bijection naturelle $\pi \mapsto \pi'$ les repr\'esentations automorphes (de dimension infinie) sur $B^*(\mathbf{A})$ et les repr\'esentations automorphes sur ${\rm GL}_2(\mathbf{A})$ telles que $\pi'_v$ est dans la s\'erie discr\`ete pour toute place $v \in S(B)$.

\subsection{ La $\pi$-partie de la cohomologie de de Rham \`a supports de $\Sigma_n$}\label{Cer-Dr} 
Soit $B$ une alg\`ebre de quaternions sur $\q$, d\'eploy\'ee \`a l'infini et de discriminant $p\ell$, o\`u $\ell$ est un nombre premier diff\'erent de $p$ fix\'e\footnote{Ceci simplement pour fixer les id\'ees et appliquer tels quels les r\'esultats du paragraphe \ref{hecke}.}. Fixons un isomorphisme $B\otimes_{\mathbf{Q}} \qp\simeq D$, ainsi qu'un sous-groupe compact ouvert $K=K_p.K^p$ de $B^*(\mathbf{A}_f)$, avec $K_p=1+p^n \O_D$. A ces donn\'ees correspond une courbe de Shimura $\mathrm{Sh}_K$ sur $\mathbf{Q}$, classifiant des surfaces ab\'eliennes avec action de $\O_B$ et structure de niveau $K$. Les $\mathbf{C}$-points de cette courbe sont donn\'es par 
$${\rm Sh}_K(\mathbf{C})=B^*(\mathbf{Q})\backslash (X\times B^*(\mathbf{A}_f)/K),$$
o\`u $B^*(\mathbf{Q})$ agit sur $X=\mathbf{C}\setminus \mathbf{R}$ via le plongement 
$B^*(\mathbf{Q})\to B^*(\mathbf{R})\simeq {\rm GL}_2(\mathbf{R})$ et l'action naturelle de 
ce dernier groupe sur $X$.

Nous allons voir dans la suite $\mathrm{Sh}_K$ comme une courbe sur $\mathbf{Q}_p$, i.e. nous allons \'ecrire
$\mathrm{Sh}_K$ au lieu de $\mathrm{Sh}_K\otimes_{\mathbf{Q}} \qp$.
Le th\'eor\`eme d'uniformisation s'\'enonce alors ainsi\footnote{Rappelons que
$\breve{\mathcal{M}_n}$ est l'espace de Rapoport-Zink de niveau $1+p^n\O_D$ attach\'e \`a un 
$\O_D$-module formel sp\'ecial sur $\overline{\mathbf{F}_p}$, cf. l'introduction.} (pour une d\'emonstration, cf. \cite[Theorem 3.1]{BoutotZink}). 

\begin{theoreme}[Cerednik-Drinfeld]\label{uniformisation}
Si $K^p$ est suffisamment petit, il existe un isomorphisme d'espaces analytiques rigides sur $\breve{\mathbf{Q}}_p$, compatible avec les donn\'ees de descente \`a la Weil
\[  \bar{B}^{*}(\q) \backslash \left(\breve{\mathcal{M}}_n \times \bar{B}^{*}(\mathbf{A}_f^p)/K^p \right) \simeq (\mathrm{Sh}_K \otimes_{\qp} \breve{\mathbf{Q}}_p)^{\mathrm{\mathrm{an}}}, \]
o\`u $\bar{B}$ d\'esigne l'alg\`ebre de quaternions sur $\q$ isomorphe \`a $B$ hors de $\{p,\infty\}$, non ramifi\'ee en $p$ et ramifi\'ee en $\infty$.
\end{theoreme}

 On se place maintenant dans le cadre de la partie \ref{formes}, et on utilise les notations du lemme \ref{union disjointe} (avec un sous-groupe
 $K^p \subset \bar{B}^*(\mathbf{A}_f^p)$ suffisamment petit). 
 Le choix des $y_i$ induit un isomorphisme 
\[ \coprod_{i=1}^r \Gamma_i \backslash \breve{\mathcal{M}}_n \simeq \bar{B}^{*}(\q) \backslash \left(\breve{\mathcal{M}}_n \times \bar{B}^{*}(\mathbf{A}_f^p)/K^p \right),\]
  Soit $N$ tel que $p^N\in \Gamma_i$ pour tout $i$, soit 
  $ \breve{\mathcal{M}}_{n,N}=p^{N\mathbf{Z}}\backslash  \breve{\mathcal{M}}_n$ et soit 
  $\Sigma_{n,N}$ la descente \`a $\qp$ de $p^{N\mathbf{Z}}\backslash \breve{\mathcal{M}}_n$, via la donn\'ee de descente \`a la Weil. La suite exacte de 
 Hochschild-Serre \cite[par. 5]{SS} pour le rev\^etement galoisien $\breve{\mathcal{M}}_{n,N}\to \Gamma_i\backslash \breve{\mathcal{M}}_{n,N}$ et la compatibilit\'e de l'isomorphisme de Cerednik-Drinfeld avec la donn\'ee de descente \`a la Weil fournissent une suite spectrale 
 \[ E_2^{p,q}=\prod_{i=1}^r H^p(\Gamma_i,H_{\mathrm{dR}}^q(\Sigma_{n,N})) \Rightarrow H_{\mathrm{dR}}^{p+q}(\mathrm{Sh}_K). \]
  Puisque $H_{\mathrm{dR}}^k(\Sigma_{n,N})$ est le dual alg\'ebrique de $H_{\mathrm{dR,c}}^{2-k}(\Sigma_{n,N})$, la suite spectrale pr\'ec\'edente se r\'e\'ecrit 
$$E_{2}^{p,q}=\prod_{i=1}^r\ext_{\Gamma_i}^p(H_{\mathrm{dR,c}}^{2-q}(\Sigma_{n,N}),1)  \Rightarrow H_{\mathrm{dR}}^{p+q}(\mathrm{Sh}_K).$$
Le groupe $\Gamma_i$ \'etant discret et les $G$-repr\'esentations $H_{\mathrm{dR,c}}^{2-q}(\Sigma_{n,N})$ \'etant lisses (cela est trivial si $q=0$ ou $q=2$, et d\'ecoule de la proposition
\ref{lissit\'e} si $q=1$), la r\'eciprocit\'e de Frobenius (lisse) permet d'\'ecrire
$$\prod_{i=1}^r\ext_{\Gamma_i}^p(H_{\mathrm{dR,c}}^{2-q}(\Sigma_{n,N}),1)=\prod_{i=1}^r  \ext_{G}^p (H_{\mathrm{dR,c}}^{2-q}(\Sigma_{n,N}), \mathrm{Ind}_{\Gamma_i}^G 1).$$
Enfin, en remarquant que par d\'efinition 
$$\prod_{i=1}^r \mathrm{Ind}_{\Gamma_i}^G 1= {\rm LC}(X(K^p)),$$
on obtient une suite spectrale 
\[ E_2^{p,q}(K^p)=\ext_G^p ( H_{\mathrm{dR,c}}^{2-q}(\Sigma_{n,N}), {\rm LC}(X(K^p))) \Rightarrow H_{\mathrm{dR}}^{p+q}(\mathrm{Sh}_K). \]
 
On v\'erifie que cette suite spectrale ne d\'epend pas du choix des repr\'esentants $y_i$ \cite[prop. 4.3.11]{Fargues these}. Le groupe $D^*$ agit sur les groupes $E_2^{p,q}(K^p)$ ainsi que sur l'aboutissement de la suite spectrale, puisque le sous-groupe $1+ p^n \O_D$ est distingu\'e dans $D^*$, et la suite spectrale est $D^*$-\'equivariante 
(cela se v\'erifie comme dans \cite[lemme 4.3.13-4.3.14]{Fargues these}). 

Regardons maintenant la composante $\rho$-isotypique de la suite spectrale obtenue. Notons que $H_{\mathrm{dR,c}}^k(\Sigma_{n,N})^{\rho^{\vee}}=H_{\mathrm{dR,c}}^k(\Sigma_{n})^{\rho^{\vee}}$
puisque le caract\`ere central de $\rho$ est trivial. 
 Comme $\rho$ n'est pas de dimension $1$, $H_{\mathrm{dR,c}}^k(\Sigma_{n,N})^{\rho^{\vee}}$ est nul sauf pour $k=1$ (en effet, $D^*$ agit par la norme r\'eduite sur l'ensemble des composantes connexes g\'eom\'etriques de $\Sigma_{n}$ : cela se d\'eduit des r\'esultats de Strauch \cite{Strauch cc} pour la tour de Lubin-Tate et de l'isomorphisme de Faltings-Fargues \cite{faltings, FGL}, m\^eme s'il est probable qu'un argument plus simple existe...). 
 La ($\rho$-partie de la) suite spectrale d\'eg\'en\`ere donc trivialement et donne un isomorphisme
\begin{equation}  \ho_G(H_{\mathrm{dR,c}}^{1}(\Sigma_{n})^{\rho^{\vee}}, {\rm LC}(X(K^p))) \simeq H_{\mathrm{dR}}^{1}(\mathrm{Sh}_K)^{\rho} \label{egalitess} \end{equation}
  compatible au changement de niveau $K^p$, ce qui fournit un isomorphisme $\bar{B}^*(\mathbf{A}_f^{p})$-\'equivariant

\begin{equation} \underset{K^p} \varinjlim ~ \ho_G(H_{\mathrm{dR,c}}^{1}(\Sigma_n)^{\rho^{\vee}}, {\rm LC}(X(K^p))) \simeq \underset{K^p} \varinjlim ~ H_{\mathrm{dR}}^{1}(\mathrm{Sh}_K)^{\rho}. \label{egalitess} \end{equation}
Du th\'eor\`eme \ref{existeforme}, on d\'eduit\footnote{Notons que pour le calcul de la cohomologie de de Rham, on a besoin d'un \'enonc\'e bien plus faible que le th\'eor\`eme \ref{existeforme}, puisqu'on n'a besoin d'aucune condition sur la repr\'esentation r\'esiduelle. Comme $\pi$ est supercuspidale, il suffirait, pour $p>2$, de choisir pour $f$ une forme donn\'ee par l'induite automorphe d'un caract\`ere de Hecke d'un corps quadratique imaginaire.} en particulier l'existence d'une forme modulaire quaternionique $f$ de poids $2$ pour le groupe $\bar{B}^*$, telle que si $\pi(f)$ est la repr\'esentation automorphe de $\bar{B}^*(\mathbf{A})$, $\pi(f)_p=\pi$. Rappelons que $B^*(\mathbf{A}_f^p) \simeq \bar{B}^*(\mathbf{A}_f^p)$ et que l'on peut donc voir $\pi(f)^p$ alternativement comme une repr\'esentation de l'un ou l'autre groupe. On a alors d'une part
(en utilisant le lemme \ref{formes automorphes} pour $W$ triviale) 
\[ \underset{K^p} \varinjlim ~{\rm LC}(X(K^p))[\pi(f)^p] \simeq \pi \]
et (en utilisant des th\'eor\`emes de comparaison standard) 
\[ \underset{K^p} \varinjlim ~H_{\mathrm{dR}}^{1}(\mathrm{Sh}_K)^{\rho}[\pi(f)^p] \simeq E, \]
o\`u $E$ est de dimension 2. La composante $\pi(f)^p$-isotypique de l'\'egalit\'e \eqref{egalitess} s'\'ecrit donc
\[ \ho_G (H_{\mathrm{dR,c}}^{1}(\Sigma_n)^{\rho^{\vee}}, \pi) = E. \]
On en d\'eduit le th\'eor\`eme \ref{derham}.

\section{Construction d'un morphisme $G$-\'equivariant}\label{construction morphisme}

L'objectif de cette section est la preuve du th\'eor\`eme suivant:

\begin{theoreme}\label{existemorphisme}
Il existe $\Pi\in \mathcal{V}(\pi)$ et 
un morphisme $G$-\'equivariant continu non nul
\[ \Phi : (\Pi^{\rm an}/\Pi^{\rm lisse})^* \to \O(\Sigma_n)^{\rho}. \]
\end{theoreme}

  Nous verrons plus tard que le terme de gauche ne d\'epend pas du choix de $\Pi\in \mathcal{V}(\pi)$, que 
  $\Phi$ est automatiquement un isomorphisme et qu'il est unique \`a scalaire pr\`es, mais cela co\^utera nettement plus cher. 
  Tout comme le calcul de la cohomologie de de Rham dans le paragraphe pr\'ec\'edent, l'argument repose sur le th\'eor\`eme d'uniformisation $p$-adique et un ingr\'edient global, mais contrairement au chapitre pr\'ec\'edent (qui s'applique tel quel \`a toute extension finie de $\qp$), le r\'esultat de compatibilit\'e local-global \ref{locglobfort}, d\^u \`a Emerton, est pour l'instant connu uniquement pour $G$.
  
\subsection{Compatibilit\'e local-global (d'apr\`es Emerton)}\label{hecke}

    Nous rappelons dans 
ce paragraphe un r\'esultat de compatibilit\'e local-global pour l'alg\`ebre de quaternions $\ob$. Comme celle-ci est \textit{d\'eploy\'ee} en $p$, il suffit de copier l'argument d'Emerton \cite{Emcomp}. Aucune id\'ee nouvelle n'est donc requise. Cependant, comme les arguments n'ont jamais \'et\'e \'ecrits \textit{stricto sensu} pour ces alg\`ebres de quaternions, et comme ce contexte permet pas mal de simplifications, nous avons choisi de les r\'ediger, pour la commodit\'e du lecteur, dans un appendice. 
 
    Nous allons nous placer encore une fois dans le contexte de la section \ref{formes} et utiliser les notations introduites dans cette section. 
   On fixe un ensemble fini $\Sigma=\Sigma(K^p)$ de nombres premiers contenant $p$ et les places o\`u $K^p=\prod_{\ell\ne p} K_{\ell}$ est ramifi\'e. 
   Si $\ell\notin \Sigma$, on note $ \mathcal{H}(K_{\ell} \backslash \ob(\mathbf{Q}_{\ell})^*/K_{\ell},\O_L)$ l'alg\`ebre de Hecke sph\'erique correspondante. 
   Cette alg\`ebre est isomorphe \`a $\O_L[T_{\ell}, S_{\ell}^{\pm 1}]$ ($T_{\ell}$ resp. $S_{\ell}$ \'etant la fonction caract\'eristique de $K_{\ell} \left(\begin{smallmatrix} \ell & 0 \\0 & 1\end{smallmatrix}\right) K_{\ell}$ (resp. $K_{\ell} \left(\begin{smallmatrix} \ell & 0 \\0 & \ell \end{smallmatrix}\right) K_{\ell}$)) et agit par des op\'erateurs continus sur 
   $\con(X(K_p), M)$ et $\con(X, M)$ pour tout $\O_L$-module topologique $M$ et tout sous-groupe ouvert compact 
   $K_p$ de $G$. On note 
\[ \mathbf{T}_{\Sigma} := \otimes'_{\ell \notin \Sigma} \mathcal{H}(K_{\ell} \backslash \ob(\mathbf{Q}_{\ell})^*/K_{\ell},\O_L) \]
et, si $K_p$ est un sous-groupe ouvert compact de $\ob^*(\qp)$, on note $\tilde{\mathbf{T}}_{\Sigma}(K_p)$ l'image de $\mathbf{T}_{\Sigma} $ dans $\mathrm{End}_{\O_L}(\con(X(K_p),\O_L))$. Enfin, 
$$\tilde{\mathbf{T}}_{\Sigma}=\varprojlim_{K_p} \tilde{\mathbf{T}}_{\Sigma}(K_p)$$
d\'esigne 
 l'adh\'erence faible de l'image de $\mathbf{T}_{\Sigma}$ dans $\mathrm{End}_{\O_L}^{\rm cont}(\con(X,\O_L))$.
 L'alg\`ebre $\tilde{\mathbf{T}}_{\Sigma}(K_p)$ est une $\O_L$-alg\`ebre commutative, libre de type fini comme $\O_L$-module. L'alg\`ebre $\tilde{\mathbf{T}}_{\Sigma}$ est une $\O_L$-alg\`ebre compacte, r\'eduite, commutative (tous ces r\'esultats sont standard). 

\begin{definition} a) Soit $\phi: \tilde{\mathbf{T}}_{\Sigma}\to R$ un morphisme continu d'anneaux topologiques. 
Une repr\'esentation continue 
$r: G_{\q,\Sigma} \to {\rm GL}_2(R)$ est {\it{associ\'ee \`a}} $\phi$ si 
$$\det(X-r({\rm Frob}_{\ell}))=X^2-\phi(T_{\ell})X+\ell\phi(S_{\ell})\in R[X], \quad \forall  \ell\notin\Sigma.$$

b) Une repr\'esentation continue $\bar{r} : G_{\q,\Sigma} \to {\rm GL}_2(k_L)$ est dite {\it{modulaire}} (de niveau mod\'er\'e $K^p$) si elle est associ\'ee \`a la projection canonique 
$\tilde{\mathbf{T}}_{\Sigma}\to \tilde{\mathbf{T}}_{\Sigma}/\mathfrak{m}$ pour un id\'eal maximal $\mathfrak{m}$ de $\tilde{\mathbf{T}}_{\Sigma}$, de corps r\'esiduel $k_L$. 

\end{definition}

Fixons d\'esormais une repr\'esentation modulaire $\bar{r}: G_{\q,\Sigma} \to {\rm GL}_2(k_L)$ (de niveau mod\'er\'e $K^p$) et notons $\mathfrak{m}$ l'id\'eal maximal correspondant de $\tilde{\mathbf{T}}_{\Sigma}$. On fait l'hypoth\`ese suivante, qui simplifie consid\'erablement les arguments: 
\newline

\textbf{\textit{Hypoth\`ese.}} {\it{La repr\'esentation $\bar{r}_{|_{\mathcal{G}_{\qp}}}$ est absolument irr\'eductible.}} 

\begin{definition} On note $A$ le compl\'et\'e $\mathfrak{m}$-adique de $\tilde{\mathbf{T}}_{\Sigma}$. C'est une $\O_L$-alg\`ebre plate, locale, noeth\'erienne, r\'eduite de corps r\'esiduel $k_L$, facteur direct de $\tilde{\mathbf{T}}_{\Sigma}$. Si $M$ est un $\tilde{\mathbf{T}}_{\Sigma}$-module, on note 
$M_m=A\otimes_{\tilde{\mathbf{T}}_{\Sigma}} M$. Ainsi, $M_m$ est facteur direct de $M$. 
\end{definition}

Comme l'hypoth\`ese ci-dessus implique en particulier que $\bar{r}$ est irr\'eductible, il existe, d'apr\`es un th\'eor\`eme de Carayol \cite[th. 3]{Carayol2}, pour tout $K_p$ suffisamment petit
une unique repr\'esentation continue $$r^{\mathfrak{m}}(K_p) : G_{\q,\Sigma} \to {\rm GL}_2(\tilde{\mathbf{T}}_{\Sigma}(K_p)_{\mathfrak{m}})$$ associ\'ee au morphisme canonique 
$\tilde{\mathbf{T}}_{\Sigma}\to \tilde{\mathbf{T}}_{\Sigma}(K_p)_{\mathfrak{m}}$. La repr\'esentation 
$$r^{\mathfrak{m}}= \varprojlim r^{\mathfrak{m}}(K_p)$$
est alors associ\'ee au morphisme canonique $\tilde{\mathbf{T}}_{\Sigma}\to A$ et $\overline{r^{\mathfrak{m}}}=\overline{r}$.

 Notons ${\rm MaxSpec}(A[1/p])$ l'ensemble des id\'eaux maximaux de $A[1/p]$. Puisque $A[1/p]$ est un anneau de Jacobson, pour tout
 $\mathfrak{p}\in {\rm MaxSpec}(A[1/p])$ le corps r\'esiduel $k(\mathfrak{p})$ de $\mathfrak{p}$ est une extension finie de $L$ et l'image du morphisme canonique 
 $A\to k(\mathfrak{p})$ est contenue dans l'anneau des entiers de $k(\mathfrak{p})$. Soit 
$r(\mathfrak{p}) : G_{\q,\Sigma} \to {\rm GL}_2(k(\mathfrak{p}))$ la sp\'ecialisation de $r^{\mathfrak{m}}$ via $A[1/p] \to k(\mathfrak{p})$, et soit 
$\Pi(\mathfrak{p})$ la repr\'esentation de Banach unitaire de $G$ attach\'ee \`a $r(\mathfrak{p})|_{\mathcal{G}_{\qp}}$ via la correspondance de Langlands locale $p$-adique pour $G$. 
Le r\'esultat de compatibilit\'e local-global dont nous aurons besoin est alors\footnote{Nous rappelons qu'il est enti\`erement d\^u \`a Emerton.}

\begin{theoreme} \label{locglobfort}
 Pour tout id\'eal maximal $\mathfrak{p}$ de $A[1/p]$, on a un isomorphisme de repr\'esentations de $G$:
\[ \con(X)[\mathfrak{p}] \simeq \Pi(\mathfrak{p})^{\oplus r}, \]
pour un certain entier $r>0$.
\end{theoreme}

\begin{proof} Voir l'appendice.
\end{proof}

Le th\'eor\`eme suivant, dont la preuve est aussi adapt\'ee d'un argument d'Emerton \cite{Emcomp}, nous permettra d'appliquer la th\'eorie globale dans notre situation. Le lecteur pourra trouver plus de d\'etails dans la preuve du th\'eor\`eme 5.1 de \cite{BBourbaki}.

\begin{theoreme}\label{existeforme}
Il existe une forme modulaire quaternionique $f$ de poids $2$ pour le groupe $\bar{B}^*$, telle que si $\pi(f)$ est la repr\'esentation automorphe correspondante de $\bar{B}^*(\mathbf{A})$, $\pi(f)_p= \pi \otimes \xi \circ \det$, o\`u $\xi$ est un caract\`ere non ramifi\'e, et telle que $\overline{r}_f : \mathcal{G}_{\q} \to  {\rm GL}_2(\overline{\mathbf{F}_p})$ soit absolument irr\'eductible en restriction \`a $\mathcal{G}_{\qp}$.
\end{theoreme}

\begin{proof}
Soit $\sigma$ un ${\rm GL}_2(\zp)$-type minimal de $\pi$. Choisissons un r\'eseau ${\rm GL}_2(\zp)$-stable $\sigma_{0}$ dans $\sigma$, et notons $\overline{\sigma_0}=\sigma_0\otimes_{\O_L} k_L$. Fixons un poids de Serre $W\in {\rm soc}_{{\rm GL}_2(\zp)} (\overline{\sigma_0})$.

\begin{lemme}
Il existe une forme modulaire quaternionique $g$ pour $\bar{B}^*$ telle que si 
$r: \mathcal{G}_{\q}\to {\rm GL}_2(L)$ est la repr\'esentation galoisienne associ\'ee \`a $g$, alors: 

a) La restriction $\overline{r}_p:=\overline{r}|_{\mathcal{G}_{\qp}}$ de la 
r\'eduction (modulo $p$) 
 $\overline{r} : \mathcal{G}_{\q} \to {\rm GL}_2(\overline{\mathbf{F}}_p)$ est absolument irr\'eductible, et
 
 b) Si $\bar{\pi}$ est la repr\'esentation lisse de $G$ correspondant \`a $\overline{r}_p$ par la correspondance de Langlands locale modulo $p$, alors $W\in {\rm soc}_{{\rm GL}_2(\zp)} (\bar{\pi})$.\end{lemme} 

\begin{proof}
On peut supposer $W=\mathrm{Sym}^r \mathbf{F}_p^2$ pour un certain $0\leq r\leq p-1$. La description explicite de la correspondance de Langlands locale modulo $p$ montre qu'il suffit d'imposer que $\overline{r}_p={\rm Ind}_{\mathbf{Q}_{p^2}}^{\qp} \omega_2^{r+1}$, \`a twist par un caract\`ere non ramifi\'e pr\`es.
Donc il suffit de trouver $\overline{r}$ modulaire telle que $\overline{r}_p$ soit cette induite. Soit 
$F$ un corps quadratique imaginaire avec $p$ et $\ell$ inertes dans $F$. Soit $\chi$ un caract\`ere de Hecke $\chi$ de $F$ tel que
$\overline{\eta_p}=\omega_2^{r+1}$, $\chi_{\infty}(z)=z^{-1}$ et $\chi_{\ell}$ choisi de sorte que l'induite automorphe locale (pour le groupe $ {\rm GL}_2$) ${\rm Ind}_{F_{\ell}}^{\mathbf{Q}_{\ell}} \chi_{\ell}$ soit supercuspidale. 
 La th\'eorie de Hecke dit alors que l'induite automorphe globale de $\chi$ est la repr\'esentation automorphe associ\'ee \`a une forme modulaire de poids $2$, qui se transf\`ere (gr\^ace \`a la correspondance de Jacquet-Langlands globale) en une forme quaternionique pour notre alg\`ebre de quaternions $\bar{B}$. Cette forme satisfait aux conditions impos\'ees. 
  \end{proof}

  Soit $g$ une forme comme dans le lemme pr\'ec\'edent, choisissons un ensemble 
  $\Sigma$ suffisamment grand et notons 
$\mathfrak{m}$ l'id\'eal de l'alg\`ebre de Hecke $\tilde{\mathbf{T}}_{\Sigma}$  correspondant \`a $\overline{r}$. 
Soit 
$\mathfrak{p}$ l'id\'eal maximal de $A[1/p]$ associ\'e \`a $r$. Le th\'eor\`eme \ref{locglobfort}\footnote{Pour cet argument, l'\'enonc\'e plus faible \ref{locglobfaible} de l'appendice serait en fait suffisant.} fournit un plongement 
$\Pi(\mathfrak{p})\to \con(X(K^p))[\mathfrak{p}]$, et donc en r\'eduisant mod $p$, on en d\'eduit que $\overline{\pi}=\overline{\Pi(\mathfrak{p})}$ se plonge dans $\con(X(K^p),k_L)[\mathfrak{m}]$. En particulier,
\[ \ho_{ {\rm GL}_2(\zp)}(W, \con(X(K^p),k_L)_{\mathfrak{m}}) \neq 0, \]
De plus, le foncteur $\ho_{ {\rm GL}_2(\zp)}(\cdot, \con(X(K^p),\oo_L/\pi_L^n)_{\mathfrak{m}})$ est exact pour tout $n\geq 1$ (lemme \ref{injective 1} dans l'appendice). On en d\'eduit dans un premier temps (en prenant $n=1$) que 
\[ \ho_{ {\rm GL}_2(\zp)}(\overline{\sigma_0}, \con(X(K^p),k_L)_{\mathfrak{m}}) \neq 0 \]
puis, par r\'ecurrence sur $n$ et en passant \`a la limite projective que
\[ \ho_{ {\rm GL}_2(\zp)}(\sigma_0,\con(X(K^p),\oo_L)_{\mathfrak{m}}) \neq 0. \]
Ensuite, comme $\sigma$ est lisse, on a $$\ho_{ {\rm GL}_2(\zp)}(\sigma_0,\con(X(K^p),\O_L)_{\mathfrak{m}})[1/p]=\ho_{ {\rm GL}_2(\zp)}(\sigma_0, {\rm LC}(X(K^p),\O_L)_{\mathfrak{m}})[1/p]$$ et ce $L$-espace vectoriel est de dimension finie. Il existe donc un id\'eal maximal 
$\mathfrak{p}_1$ de $(\mathbf{T}_{\Sigma})_\mathfrak{m}[1/p]$ tel que 
\[ \ho_{ {\rm GL}_2(\zp)}(\sigma_0, {\rm LC}(X(K^p),\oo_L)_{\mathfrak{m}}[\mathfrak{p}_1]) \neq 0. \]
L'id\'eal $\mathfrak{p}_1$ correspond \`a une forme modulaire $f$ pour $\bar{B}^*$ et la non annulation obtenue nous dit donc que $\sigma$ se plonge dans $\pi(f)_p$ (utiliser le lemme \ref{formes automorphes}) et donc \cite{henniart} que $\pi(f)_p \simeq  \pi\otimes \xi \circ \det$, avec $\xi$ un caract\`ere non ramifi\'e. La propri\'et\'e de $r_f$ d\'ecoule du choix de $\mathfrak{p}_1$.
\end{proof}

\subsection{Nouvelle application du th\'eor\`eme d'uniformisation $p$-adique}
 
   Le but de cette partie est d'expliquer la preuve du r\'esultat suivant, qui est une r\'eformulation tr\`es simple, modulo quelques pr\'ecautions topologiques, mais bien pratique, du th\'eor\`eme de Cerednik-Drinfeld. On \'ecrira 
   ${\rm Hom}_G$ au lieu de ${\rm Hom}_G^{\rm cont}$ dans la suite. Nous renvoyons le lecteur aux sections \ref{formes} et \ref{Cer-Dr} pour 
   les objets apparaissant dans l'\'enonc\'e suivant. 
   
\begin{theoreme}\label{omega}
On a un isomorphisme de modules de Hecke
\[  \ho_G({\rm LA}(X(K^p))^*,\Omega^1(\Sigma_n)^{\rho})\simeq \Omega^1(\mathrm{Sh}_K)^{\rho}. \]
\end{theoreme}

\begin{proof}
     Reprenons les notations du lemme \ref{union disjointe}. En calculant les sections globales de 
     $\Omega^1$ des deux c\^ot\'es de l'isomorphisme de Cerednik-Drinfeld, puis en prenant les parties $\rho$-isotypiques, on obtient un isomorphisme 
    $$\bigoplus_{i=1}^r ((\Omega^1(\Sigma_{n}))^{\rho})^{\Gamma_i} \simeq \Omega^1(\mathrm{Sh}_K)^{\rho}.$$
     Pour simplifier les formules, nous allons poser 
     $$W=((\Omega^1(\Sigma_{n}))^{\rho})^*\quad \text{et} \quad X=X(K^p)=\coprod \Gamma_i \backslash G$$
     dans la suite de cette preuve. Le th\'eor\`eme \ref{agit bien} montre que 
     $W$ est une repr\'esentation localement analytique de $G$ sur un espace de type compact. 
     Puisque ${\rm LA}(X)$ et $W$ sont r\'eflexifs, le th\'eor\`eme \ref{omega} est une cons\'equence du 
     r\'esultat suivant, qui est une application simple de la r\'eciprocit\'e de Frobenius, modulo quelques probl\`emes topologiques. 

\begin{lemme}\label{Frob}
On a un isomorphisme de modules de Hecke 
\[\bigoplus_{i=1}^r ((\Omega^1(\Sigma_{n}))^{\rho})^{\Gamma_i}   \simeq \ho_G(W, {\rm LA}(X)). \]
\end{lemme}

\begin{proof}
Soit $\omega=(\omega_i)_{i=1,\dots,r} \in \bigoplus_{i=1}^r ((\Omega^1(\Sigma_{n}))^{\rho})^{\Gamma_i}$. On lui associe $\phi_{\omega} : W \to {\rm LA}(X)$, qui envoie $\ell \in W$ sur la fonction $\phi_{\omega}(\ell) : \Gamma_i g \mapsto \ell(g^{-1}.\omega_i)$ sur $X=\coprod \Gamma_i \backslash G$. Montrons tout d'abord que cette d\'efinition a un sens : si $g'=\gamma g$, avec $\gamma \in \Gamma_i$, $1 \leq i \leq r$, on a bien
\[ \ell((g')^{-1} .\omega_i)=\ell(g^{-1}\gamma_i^{-1}.\omega_i )=\ell(g^{-1}.\omega_i), \]
puisque $\omega_i$ est invariante par $\Gamma_i$. Pour voir que $\phi_{\omega}(l)$ est bien localement analytique, il suffit de noter que $\ell(g^{-1}.\omega_i)= (g.\ell)(\omega_i)$ et d'utiliser le fait que $G \to \Omega^1(\m_{n})^*, g \mapsto g.\ell$ est localement analytique (cf. \ref{agit bien}). De plus, $\phi_{\omega}$ est $G$-\'equivariante, puisque
\[ \phi_{\omega} (g.\ell)(\Gamma_i g')=(g.\ell)((g')^{-1}.\omega_i)=\ell(g^{-1}(g')^{-1}.\omega_i)=\ell((g' g)^{-1}.\omega_i) \]
\[     =\phi_{\omega}(\ell)(\Gamma_i g'g)=(g.\phi_{\omega}(\ell))(\Gamma_i g'). \]
Il reste \`a voir que $\phi_{\omega}$ est un morphisme continu. Cela revient \'evidemment \`a montrer que 
\[W \to {\rm LA}(G),  ~ \ell \mapsto (g.\ell)(\omega_i) \]
est continue, pour chaque $i=1, \dots,r$, ce qui d\'ecoule du r\'esultat g\'en\'eral suivant (en prenant $\lambda=ev_{\omega_i}$, $w=\ell$) :

\begin{lemme}
Soit $W$ une repr\'esentation localement analytique d'un groupe de Lie $p$-adique $G$ sur un espace de type compact. Fixons $\lambda \in W^*$. Le morphisme
\[ F: W \to {\rm LA}(G),  ~ w \mapsto (g \mapsto \lambda(g.w)) \]
est continu.
\end{lemme}

\begin{proof} En d\'ecomposant 
$G$ selon les classes \`a gauche modulo un sous-groupe ouvert compact de $G$, on peut supposer que 
$G$ est compact.
Comme $W$ et ${\rm LA}(G)$ sont des espaces r\'eflexifs, il suffit de montrer que $F^*: D(G) \to W^*$ est continue. Puisque $D(G)$ et $W$ sont tous deux des espaces de Fr\'echet, on peut tester la continuit\'e de $F^*$ avec des suites. Soit $(\mu_n)_n$ une suite d'\'el\'ements de $D(G)$ convergeant vers $0$. Par d\'efinition, $F^*$ envoie $\mu \in D(G)$ sur l'\'el\'ement $w \mapsto \int_G \lambda(g.w) \mu$. Or cette derni\`ere quantit\'e est exactement $\lambda(I(\rho_w)(\mu))=\lambda(\mu * w)$ dans les notations de \cite{ST}. D'apr\`es \cite[prop. 3.2]{ST}, $\mu \mapsto \mu * w$ est continue. En particulier, pour chaque $w \in W$, la suite $(\mu_n * w)_n$ tend vers $0$. On conclut alors avec le tr\`es utile lemme suivant, appliqu\'e \`a $V=W^*$.

\begin{lemme}\label{faible}
 Soit $V$ un espace vectoriel localement convexe sur un corps sph\'eriquement complet. Une suite 
 $(v_n)_n$ de $V$ converge vers $0$ si et seulement si elle converge faiblement vers $0$. 
\end{lemme}
\end{proof}

   Pour conclure la preuve du lemme \ref{Frob}, il suffit d'exhiber un inverse de l'application d\'ej\`a construite: cet inverse 
   envoie $\psi \in \ho_G(W, {\rm LA}(X))$ sur le $r$-uplet des \'el\'ements de $\Omega^1(\Sigma_{n})^{\rho}=W^*$ correspondant \`a $\ell \in W \mapsto \psi(\ell)(\gamma_i)$, pour un $\gamma_i \in \Gamma_i$ quelconque. On v\'erifie que l'isomorphisme construit commute \`a l'action de l'alg\`ebre de Hecke hors $p$. 
\end{proof}
   Cela finit la preuve du th\'eor\`eme \ref{omega}.
\end{proof}

\begin{remarque}\label{schneider}
La fl\`eche naturelle d\'eduite de la surjection $\Omega^1(\Sigma_n)^{\rho} \to H_{\mathrm{dR}}^1(\Sigma_n)^{\rho}$ :
\[ \ho_G((\Omega^1(\Sigma_n)^{\rho}))^*,{\rm LA}(X(K^p))) \to \ho_G(H_{\mathrm{dR,c}}^{1}(\Sigma_{n})^{\rho}, {\rm LC}(X(K^p))) \]
s'identifie via les isomorphismes de la proposition \ref{omega} et du th\'eor\`eme \ref{derham} \`a la filtration de Hodge
\[ \Omega^1(\mathrm{Sh}_K)^{\rho} \to H_{\mathrm{dR}}^1(\mathrm{Sh}_K)^{\rho}. \]
Pour le voir, on se ram\`ene imm\'ediatement au cas du rev\^ etement \'etale de la vari\'et\'e propre et lisse $X_{\Gamma}=\Gamma \backslash \breve{\mathcal{M}_{n}}$, avec $\Gamma$ discret cocompact dans $G$, par la vari\'et\'e Stein $\breve{\mathcal{M}}_{n}$ ; dans ce cas, la filtration de Hodge sur $H_{\mathrm{dR}}^1(X_{\Gamma})=H^1(\Gamma,\Omega^{\cdot}(\breve{\mathcal{M}}_{n}))$ est donn\'ee par
\[ \mathrm{Im} \left( H^1(\Gamma,\Omega^1(\breve{\mathcal{M}}_{n})[1]) \to H^1(\Gamma,\Omega^{\cdot}(\breve{\mathcal{M}}_{n})) \right). \]
\end{remarque}

\subsection{Preuve du th\'eor\`eme \ref{existemorphisme}}\label{construction}

On peut maintenant appliquer le th\'eor\`eme \ref{existeforme} : il existe une forme modulaire quaternionique $f$ de poids $2$ pour le groupe $\bar{B}^*$, telle que si $\pi(f)$ est la repr\'esentation automorphe correspondante de $\bar{B}^*(\mathbf{A})$, $\pi(f)_p= \pi \otimes \xi \circ \det$, o\`u $\xi$ est un caract\`ere non ramifi\'e, et telle que $\overline{r}_f : \mathcal{G}_{\q} \to {\rm GL}_2(\overline{\mathbf{F}_p})$ soit absolument irr\'eductible en restriction \`a $\mathcal{G}_{\qp}$. Soit $\Sigma$ un ensemble fini de premiers contenant $p$ et suffisamment grand pour que $\overline{r}_f$ d\'efinisse un id\'eal maximal de l'alg\`ebre de Hecke $\tilde{\mathbf{T}}_{\Sigma}$\footnote{Pour les notations relatives aux alg\`ebres de Hecke, voir le paragraphe \ref{hecke}.} ; soit $A$ le compl\'et\'e $\mathfrak{m}$-adique de $\tilde{\mathbf{T}}_{\Sigma}$ et $\mathfrak{p}$ l'id\'eal maximal de $A[1/p]$ associ\'e \`a la forme $f$. 

On a d'apr\`es le th\'eor\`eme \ref{omega}
\[ \ho_G(({\rm LA}(X(K^p))[\mathfrak{p}])^*,\Omega^1(\Sigma_n)^{\rho}) \simeq \Omega^1(\mathrm{Sh}_K)^{\rho}[\mathfrak{p}]\ne 0 \]
 Le th\'eor\`eme \ref{locglobfort} permet d'obtenir ainsi l'existence d'un 
morphisme $G$-\'equivariant continu non nul
\[ \Phi_0 : (\Pi^{\mathrm{an}})^* \to \Omega^1(\Sigma_n)^{\rho}, \]
o\`u $\Pi=\Pi(\mathfrak{p})\in\mathcal{V}(\pi)$ (modulo un twist non ramifi\'ee que l'on va ignorer dans la suite).
Le morphisme $\Phi_0$ induit, par restriction, un morphisme continu $G$-\'equivariant 
\[ \Phi : (\Pi^{\rm an}/\Pi^{\rm lisse})^* \to \Omega^1(\Sigma_n)^{\rho}. \]

Montrons que ce morphisme est non nul. S'il \'etait nul, $\Phi_0$ se factoriserait par le quotient
$(\Pi^{\rm lisse})^*$ de $(\Pi^{\rm an})^*$. 
Comme $\Phi_0$ n'est pas nul et $\Pi^{\rm lisse}\simeq \pi$ est irr\'eductible, $\pi^*$ se plongerait dans $\Omega^1(\Sigma_n)^{\rho}$, ce qui est absurde puisque $u^+$ n'annule aucun \'el\'ement de $\Omega^1(\Sigma_n)^{\rho}$ (se rappeler que $u^+=-\frac{d}{dz}$ sur $\mathcal{O}(\Sigma_n)$, cf. lemme \ref{numerologie Lie}). On en d\'eduit que $\Phi$ est effectivement non nul. 

En outre, la compos\'ee de $\Phi$ avec la surjection canonique $\Omega^1(\Sigma_n)^{\rho} \to H_{\mathrm{dR}}^1(\Sigma_n)^{\rho}$ est nulle : en dualisant cette fl\`eche on obtient en effet un morphisme $G$-\'equivariant $ H_{\mathrm{dR,c}}^1(\Sigma_n)^{\rho}\to \Pi^{\rm an}/\Pi^{\rm lisse}$, qui est forc\'ement nul car $\Pi^{\rm an}/\Pi^{\rm lisse}$ n'a pas de vecteurs lisses non nuls (exercice), alors que $ H_{\mathrm{dR,c}}^1(\Sigma_n)^{\rho}$ est lisse (proposition \ref{lissit\'e}). Par cons\'equent, $\Phi$ se factorise par $\O(\Sigma_n)^{\rho}$, ce qui finit la preuve du th\'eor\`eme \ref{existemorphisme}.

\section{$(\varphi,\Gamma)$-modules sur l'anneau de Robba et \'equations diff\'erentielles $p$-adiques}\label{rappel phi Gamma}

   Ce chapitre ne contient aucun r\'esultat original et sert uniquement comme r\'ef\'erence pour la suite.
  Le lecteur pourra consulter \cite{Ber}, \cite{BerAst}, \cite{FonAst}, \cite{Marmora} pour plus de d\'etails concernant l'\'equation diff\'erentielle $p$-adique attach\'ee \`a une repr\'esentation de de Rham de $\mathcal{G}_{\qp}$. Nous nous contenterons d'\'enoncer les principaux r\'esultats concernant cette construction. 

    Rappelons que l'on a fix\'e une suite $(\zeta_{p^n})_{n\geq 1}$, o\`u $\zeta_{p^n}$ est une racine primitive de l'unit\'e d'ordre $p^n$, et 
  $\zeta_{p^{n+1}}^p=\zeta_{p^n}$. Soit $\mathcal{E}^{]0,r_n]}$ l'anneau des fonctions analytiques (d\'efinies sur $L$) sur la couronne 
  $|\zeta_{p^n}-1|\leq |T|<1$, et soit 
  $$\mathcal{R}=\varinjlim_{n} \mathcal{E}^{]0,r_n]}\subset L[[T, T^{-1}]]$$
  {\it l'anneau de Robba}. L'anneau $\mathcal{R}$ est muni d'actions continues de 
$\Gamma={\rm Gal}(\qp^{\rm cyc}/\qp)$ et d'un Frobenius $\varphi$, commutant entre elles,
en posant $$\varphi(T)=(1+T)^p-1 \quad \text{et} \quad \sigma_a(T)=(1+T)^a-1,\quad a\in\zpet,$$
o\`u $\sigma_{a}\in \Gamma$ est tel que $\chi_{\rm cyc}(\sigma_a)=a$ (en d'autres termes 
$\sigma_a(\zeta)=\zeta^a$ pour tout $\zeta\in\mu_{p^{\infty}}$). 

 \begin{definition}       Un {\it $(\varphi,\Gamma)$-module $\Delta$ sur $\mathcal{R}$} 
  est un $\mathcal{R}$-module libre de type fini muni d'actions semi-lin\'eaires continues de 
  $\varphi$ et $\Gamma$, commutant entre elles et telles que l'application naturelle 
  $\mathcal{R}\otimes_{\varphi(\mathcal{R})}\varphi(\Delta)\to \Delta$ 
  soit un isomorphisme. 
  \end{definition}

     Soit $\Delta$ un $(\varphi,\Gamma)$-module sur $\mathcal{R}$. D'apr\`es un 
   r\'esultat standard de Berger (voir \cite[lemme 4.1]{Ber}), l'action de $\Gamma$ sur $\Delta$ se d\'erive, d'o\`u une connexion 
     $$\nabla: \Delta\to \Delta, \quad \nabla(z)=\lim_{a\to 1} \frac{\sigma_a(z)-z}{a-1}.$$
     Par exemple, si $\Delta=\mathcal{R}$ est le $(\varphi,\Gamma)$-module trivial, alors 
     \[ \nabla = t \partial, \quad \text{o\`u} \quad \partial= (1+T)\frac{d}{dT} \quad \text{et}\^e\quad t=\log(1+T)\in \mathcal{R}.\]
Notons que $\varphi(t)=pt$ et $\gamma(t)=\chi_{\rm cyc}(\gamma) t$ pour 
$\gamma\in \Gamma$, ce qui fait que si $\Delta$ est un $(\varphi,\Gamma)$-module sur $\mathcal{R}$, alors 
$t^k\Delta$ l'est ausssi, pour tout entier $k$.      

 \begin{example} a) La th\'eorie de Fontaine \cite{FoGrot} combin\'ee au th\'eor\`eme de surconvergence de Cherbonnier-Colmez \cite{CCsurconv} permettent d'attacher 
  \`a toute $L$-repr\'esentation $V$ de $\mathcal{G}_{\qp}$ un $(\varphi,\Gamma)$-module (\'etale) $D_{\rm rig}(V)$ sur 
  $\mathcal{R}$, de rang $\dim_L(V)$. 
  Berger a montr\'e \cite{Ber} que le foncteur $V\to D_{\rm rig}(V)$ est pleinement fid\`ele. 
  
  b) Soit $V$ une $L$-repr\'esentation de de Rham de $\mathcal{G}_{\qp}$. Un th\'eor\`eme fondamental de Berger \cite{Ber} 
  montre l'existence d'un unique sous-$(\varphi,\Gamma)$-module $N_{\rm rig}(V)\subset D_{\rm rig}(V)[1/t]$ tel que 
  $$N_{\rm rig}(V)[1/t]=D_{\rm rig}(V)[1/t] \quad \text{et}\quad \nabla(N_{\rm rig}(V))\subset t\cdot N_{\rm rig}(V).$$
  Le $(\varphi,\Gamma)$-module $N_{\rm rig}(V)$ devient une \'equation diff\'erentielle $p$-adique avec structure de Frobenius sur $\mathcal{R}$, gr\^ace \`a la connexion 
  $\partial=\frac{1}{t} \nabla$. Ce $(\varphi,\Gamma)$-module jouera un r\^ole fondamental dans cet article (c'est pr\'ecis\'ement cette construction qui a permis \`a Berger
    de d\'emontrer le th\'eor\`eme de monodromie $p$-adique \cite{Ber}). Contrairement \`a $V\to D_{\rm rig}(V)$, le foncteur 
    $V\to N_{\rm rig}(V)$ n'est pas pleinement fid\`ele\footnote{On perd la filtration de Hodge ; voir la fin de ce chapitre pour un \'enonc\'e plus pr\'ecis.}.
  \end{example}
  
  \begin{lemme}\label{injective} Soit $P\in L[X]$ un polyn\^ome non nul et soit $V$ une $L$-repr\'esentation absolument irr\'eductible 
de dimension $2$ de $\mathcal{G}_{\qp}$, non trianguline (au sens de \cite{Ctriang}). Alors $P(\nabla)$ est injectif sur $D_{\rm rig}(V)$. 
\end{lemme}

\begin{proof} D'apr\`es \cite[prop 2.1]{DJacquet} le noyau $X$ de $P(\nabla)$ sur $D_{\rm rig}(V)$ est de dimension finie sur 
$L$. Comme il est stable par $\varphi$, on en d\'eduit que si $X\ne 0$, alors $\varphi$ a des vecteurs propres sur $D_{\rm rig}(V)$ (\'eventuellement apr\`es avoir remplac\'e 
$L$ par une extension finie). Cela contredit \cite[lemme 3.2]{Ctriang}. 
\end{proof}

       Le r\'esultat suivant est une observation importante de Colmez, qui joue un r\^ole cl\'e dans \cite{Colmezpoids}. Nous nous en servirons aussi dans le chapitre suivant.

           \begin{proposition}\label{bijective} Soit $V$ une $L$-repr\'esentation de de Rham, absolument irr\'eductible 
de dimension $2$ de $\mathcal{G}_{\qp}$. Si $V$ n'est pas trianguline, alors 
           $\partial=\frac{1}{t}\nabla$ est bijectif sur $N_{\rm rig}(V)$.
         \end{proposition}
         
         \begin{proof} Soient $N_{\rm rig}=N_{\rm rig}(V)$ et $D_{\rm rig}=D_{\rm rig}(V)$, et soit $h$ tel que 
                         $t^hN_{\rm rig}\subset D_{\rm rig}$. Puisque $(\nabla-h)(t^hx)=t^{h+1}\partial x$ et 
         $D_{\rm rig}^{\nabla=h}=0$ (lemme \ref{injective}), $\partial$ est injectif sur $N_{\rm rig}$. 
         La surjectivit\'e est plus subtile, et utilise le th\'eor\`eme de monodromie $p$-adique.
         Plus pr\'ecis\'ement, 
         \cite[prop. 20.4.2]{Kedbook} fournit un accouplement parfait $$N_{\rm rig}/\partial(N_{\rm rig})\otimes \check{N}_{\rm rig}^{\partial=0}\to L$$ o\`u 
           $\check{N}_{\rm rig}=N_{\rm rig}(\check{V})$ est l'\'equation diff\'erentielle attach\'ee au dual de Cartier $\check{V}=V^*\otimes \chi_{\rm cyc}$ de $V$. 
         Puisque $\check{V}$ n'est pas trianguline, la premi\`ere partie de la preuve montre que 
         $\check{N}_{\rm rig}^{\partial=0}=0$, ce qui permet de conclure. 

                        \end{proof}

       Soit $\Delta$ un $(\varphi,\Gamma)$-module sur $\mathcal{R}$. Berger montre \cite[th. I.3.3]{BerAst} 
       que pour tout $n$ assez grand (d\'ependant de $\Delta$) il existe un unique 
        sous $\mathcal{E}^{]0,r_n]}$-module $\Delta^{]0,r_n]}$ de $\Delta$\footnote{Si $\Delta=D_{\rm rig}(V)$, on notera $D^{]0,r_n]}$ au lieu de $(D_{\rm rig})^{]0,r_n]}$, et ainsi de suite, pour all\'eger les notations.} tel que 
        $\mathcal{R}\otimes_{\mathcal{E}^{]0,r_n]}} \Delta^{]0,r_n]}\to \Delta$ soit un isomorphisme 
        et tel que $\mathcal{E}^{]0,r_{n+1}]}\otimes_{\mathcal{E}^{]0,r_n]}} \Delta^{]0,r_n]}$ admette une base contenue dans
        $\varphi(\Delta^{]0,r_n]})$ (pour l'existence, il suffit de prendre pour 
        $\Delta^{]0,r_n]}$ le sous $\mathcal{E}^{]0,r_n]}$-module de $\Delta$
        engendr\'e par une base fix\'ee de $\Delta$). De plus, $\Delta^{]0,r_n]}$ est stable sous l'action de 
        $\Gamma$ et $\nabla$, et $\varphi(\Delta^{]0,r_n]})\subset \Delta^{]0,r_{n+1}]}$ pour tout $n$ assez grand.

Notons $L_n=L\otimes_{\qp} \qp(\zeta_{p^n})$. On
dispose pour tout $n\geq 1$ d'une injection $\Gamma$-\'{e}quivariante\footnote{Comme $f$ converge en $\zeta_{p^n}-1$, $f(\zeta_{p^n}e^{t/p^n}-1)$
   est bien d\'{e}fini en tant qu'\'{e}l\'{e}ment de $L_n[[t]]$.} d'anneaux $\varphi^{-n}: \mathcal{E}^{]0,r_n]}\to L_n[[t]]$, qui
 envoie $f$ sur $f(\zeta_{p^n}e^{t/p^n}-1)$. Si 
 $\Delta$ est un $(\varphi,\Gamma)$-module sur $\mathcal{R}$, on note (pour $n$ assez grand)
 $$\Delta_{\rm dif,n}^+=L_n[[t]]\otimes_{\mathcal{E}^{]0,r_n]}} \Delta^{]0,r_n]},\quad \Delta_{\rm dif}^+=\varinjlim_{n} \Delta_{\rm dif,n}^+,$$ le morphisme de transition $ \Delta_{\rm dif,n}^+\to  \Delta_{\rm dif, n+1}^+$ \'etant donn\'e par $u\otimes z\mapsto u\otimes \varphi(z)$ pour $u\in L_n[[t]]\subset L_{n+1}[[t]]$ et $z \in \Delta^{]0,r_n]}$ (noter que $\varphi(z)\in \Delta^{]0, r_{n+1}]}$). 
 
  Ainsi, $\Delta_{\rm dif}^+$ est un $L_{\infty}[[t]]:=\varinjlim_{n} L_n[[t]]$-module libre de m\^eme rang que $\Delta$ et il est muni d'une action de 
  $\Gamma$, qui respecte $\Delta_{\rm dif,n}^+$ pour $n$ assez grand. La d\'eriv\'ee de cette action fournit une connexion $\nabla$ sur $\Delta_{\rm dif}^+$, qui respecte $\Delta_{\rm dif,n}^+$  
          pour $n$ assez grand, et qui satisfait 
          $$\nabla(fz)=t\frac{df}{dt}\cdot z+f\cdot \nabla z,\quad \forall \, f\in L_{\infty}[[t]],\, z\in \Delta_{\rm dif}^+.$$
          
     \begin{example}\label{exemple fondamental} L'exemple suivant sera syst\'ematiquement utilis\'e dans la suite. 
    Soit $V$ une $L$-repr\'esentation de de Rham de $\mathcal{G}_{\qp}$. Si $n$ est assez grand, on dispose d'un {\it{morphisme de localisation}}\footnote{Rappelons que 
    $\mathcal{H}_{\qp}={\rm Gal}(\overline{\qp}/\qp^{\rm cyc})$.} 
     $$\varphi^{-n}: D^{]0,r_n]}(V)\to (\mathbf{B}_{\rm dR}^+\otimes_{\qp} V)^{\mathcal{H}_{\qp}},$$
     qui est $\Gamma$-\'equivariant et induit un morphisme injectif $$\varphi^{-n}: D_{\rm dif,n}^+(V)\to (\mathbf{B}_{\rm dR}^+\otimes_{\qp} V)^{\mathcal{H}_{\qp}}.$$ 
     Fontaine a montr\'e \cite{FonAst} que ce morphisme induit un isomorphisme canonique de $L_n[[t]]$-modules avec action semi-lin\'eaire de $\Gamma$
     $$D_{\rm dif,n}^+(V)={\rm Fil}^0(L_n((t))\otimes_{L} D_{\rm dR}(V)),$$
     o\`u l'on consid\`ere la filtration $t$-adique sur $L_n((t))$ et la filtration de Hodge sur $D_{\rm dR}(V)$. Berger a montr\'e \cite{Ber, BerAst} que via cet isomorphisme on peut d\'ecrire
     $N_{\rm dif,n}^+(V)\subset D_{\rm dif,n}^+(V)[1/t]$ par
     $$N_{\rm dif, n}^+(V)=L_n[[t]]\otimes_{L} D_{\rm dR}(V).$$
  On peut aussi d\'ecrire simplement $N_{\rm rig}(V)$ \`a partir de $D_{\rm rig}(V)$ via 
     $$N_{\rm rig}(V)=\{z\in D_{\rm rig}(V)[1/t]| \,\ \varphi^{-n}(z)\in N_{\rm dif,n}^+(V)\quad \forall n\gg0\}.$$
     De plus, pour $n$ assez grand 
           $$N^{]0,r_n]}(V)=\{z\in D_{\rm rig}(V)[1/t]| \,\ \varphi^{-k}(z)\in N_{\rm dif,k}^+(V)\quad \forall k\geq n\}.$$
           
    Enfin, on peut reconstruire $D_{\rm rig}(V)$ \`a partir de $N_{\rm rig}(V)$ et de la filtration de Hodge, via 
     $$D_{\rm rig}(V)=\{z\in N_{\rm rig}(V)[1/t]| \,\ \varphi^{-n}(z)\in {\rm Fil}^0(L_n((t))\otimes_{L} D_{\rm dR}(V))\quad \forall n\gg0\}.$$
  On a une description similaire de $D^{]0,r_n]}(V)$ pour $n$ assez grand.
     \end{example}
     
     Nous aurons aussi besoin d'une description plus pr\'ecise de $N_{\rm rig}(V)$. Supposons que 
     $V$ est une repr\'esentation potentiellement cristalline (pour simplifier) de $\mathcal{G}_{\qp}$, 
     et prenons une extension finie galoisienne $K$ de $\qp$ telle que $V$ soit cristalline en tant que 
     repr\'esentation de ${\rm Gal}(\overline{\qp}/K)$. Soit $K_0$ l'extension maximale non ramifi\'ee de $\qp$ dans $K$. 
     Une construction standard utilisant la th\'eorie du corps des normes de Fontaine-Wintenberger (voir le chapitre 1 de \cite{BerAst} pour les d\'etails\footnote{Nous allons noter $\mathcal{R}_K$ ce que Berger note 
     $B^{\dagger}_{\rm rig, K}$.}) permet d'associer \`a l'extension de corps perfecto\"ides $K^{\rm cyc}/\qp^{\rm cyc}$ une extension finie \'etale 
     $\mathcal{R}_K$ de l'anneau de Robba $\mathcal{R}$. De plus, $\mathcal{R}_K$ est muni d'un Frobenius $\varphi$ et d'une action 
     de ${\rm Gal}(K^{\rm cyc}/\qp)$, qui sont compatibles avec les actions correspondantes sur $\mathcal{R}$, et telles que 
     $\mathcal{R}_K^{{\rm Gal}(K^{\rm cyc}/\qp^{\rm cyc})}=\mathcal{R}$ et $\mathcal{R}_K^{{\rm Gal}(K^{\rm cyc}/K)}=K_0$. 
     Berger \cite{BerAst, Marmora} montre l'existence d'un isomorphisme 
          de comparaison, compatible avec toutes les structures suppl\'ementaires 
          $$\mathcal{R}_K\otimes_{\mathcal{R}} N_{\rm rig}(V)= \mathcal{R}_K\otimes_{K_0} D_{\rm cris, K}(V),$$
          o\`u $$D_{\rm cris,K}(V)=(\mathbf{B}_{\rm cris}\otimes_{\qp} V)^{{\rm Gal}(\overline{\qp}/K)}=D_{\rm pst}(V)^{{\rm Gal}(\overline{\qp}/K)}.$$
          En prenant les invariants par ${\rm Gal}(K^{\rm cyc}/\qp^{\rm cyc})$, on obtient la description suivante de $N_{\rm rig}(V)$
          \begin{eqnarray} N_{\rm rig}(V)=\left(\mathcal{R}_K\otimes_{K_0} D_{\rm pst}(V)^{{\rm Gal}(\overline{\qp}/K)}\right)^{{\rm Gal}(K^{\rm cyc}/\qp^{\rm cyc})}. \label{nrig} \end{eqnarray}
              Cette description montre que $N_{\rm rig}(V)$ ne d\'epend que du $(\varphi,\mathcal{G}_{\qp})$-module $D_{\rm pst}(V)$, \textit{sans sa filtration de Hodge}.\footnote{Elle montre aussi que le terme de droite de l'\'egalit\'e (4) ne d\'epend pas de $K$; cela est aussi une cons\'equence \'el\'ementaire de la th\'eorie de Galois.}
                            
     \begin{example} \label{plongement et scalaire}      Revenons maintenant \`a notre contexte usuel et consid\'erons le $(\varphi,\mathcal{G}_{\qp})$-module $M(\pi)$ attach\'e \`a 
            $\pi$ (cf. le dernier paragraphe des notations et conventions). Au vu de la discussion pr\'ec\'edente, la d\'efinition suivante n'est pas bien surprenante: 
            
                     \begin{definition}
                                 Soit $K$ une extension finie galoisienne de $\qp$, qui contient $L$ et telle que l'inertie $I_K$ de ${\rm Gal}(\overline{\qp}/K)$ agisse trivialement sur $M(\pi)$. 
          On pose 
        $$N_{\rm rig}(\pi)=\left(\mathcal{R}_K\otimes_{K_0} M(\pi)^{{\rm Gal}(\overline{\qp}/K)}\right)^{{\rm Gal}(K^{\rm cyc}/\qp^{\rm cyc})},$$
        o\`u $K_0$ est l'extension maximale non ramifi\'ee de $\qp$ dans $K$.  
      \end{definition}
      
      Le $(\varphi,\Gamma)$-module $N_{\rm rig}(\pi)$ sur $\mathcal{R}$ qui s'en d\'eduit 
        est ind\'ependant du choix de $K$, et il est libre de rang $2$ sur $\mathcal{R}$. Soit 
         $$M_{\rm dR}(\pi)=(\overline{\qp}\otimes_{\qp^{\rm nr}} M(\pi))^{\mathcal{G}_{\qp}}\simeq (K\otimes_{K_0} M(\pi)^{{\rm Gal}(\overline{\qp}/K)})^{{\rm Gal}(K/\qp)},$$
        un $L$-espace vectoriel de dimension $2$. 
        On a un isomorphisme canonique pour $n$ assez grand 
        $$(N_{\rm rig}(\pi))_{\rm dif,n}^+\simeq L_n[[t]]\otimes_{L} M_{\rm dR}(\pi). $$ 

     Toute $L$-droite $\mathcal{L}$ de  
     $M_{\rm dR}(\pi)$ d\'efinit
     une filtration exhaustive d\'ecroissante ${\rm Fil}_{\mathcal{L}}$ sur $M_{\rm dR}(\pi)$, en posant 
     $${\rm Fil}^{-1}(M_{\rm dR}(\pi))=M_{\rm dR}(\pi),\quad {\rm Fil}^0(M_{\rm dR}(\pi))=\mathcal{L}, \quad   
    {\rm Fil}^1(M_{\rm dR}(\pi))=0.$$
   
     Rappelons que $\mathcal{V}(\pi)$ est l'ensemble des (classes d'isomorphisme des) $\Pi\in {\rm Ban}^{\rm adm}(G)$ absolument irr\'eductibles telles que 
     $\Pi^{\rm lisse}\simeq \pi$. On voit dans la suite $\mathcal{V}(\pi)$ comme sous-ensemble de l'ensemble des (classes d'isomorphismes de) $L$-repr\'esentations
     absolument irr\'eductibles de $\mathcal{G}_{\qp}$, de dimension $2$, gr\^ace au foncteur de Colmez \cite[chap. II, IV]{Cbigone} et au r\'esultat principal de 
     \cite{PCD}.  
     
   Soit $V\in \mathcal{V}(\pi)$ et fixons une identification $D_{\rm pst}(V)\simeq M(\pi)$ en tant que $(\varphi, \mathcal{G}_{\qp})$-modules.
   Cet isomorphisme est unique \`a scalaire pr\`es et il induit un isomorphisme de $L$-espaces vectoriels de dimension $2$
   $$D_{\rm dR}(V)\simeq (\overline{\qp}\otimes_{\qp^{\rm nr}} D_{\rm pst}(V))^{\mathcal{G}_{\qp}}\simeq M_{\rm dR}(\pi).$$ 
   La filtration de Hodge ${\rm Fil}^0(D_{\rm dR}(V))$ d\'efinit ainsi une $L$-droite $\mathcal{L}(V)\subset M_{\rm dR}(\pi)$, qui ne d\'epend pas du choix
   de l'isomorphisme $D_{\rm pst}(V)\simeq M(\pi)$. R\'eciproquement, \'etant donn\'ee une $L$-droite $\mathcal{L}$ de $M_{\rm dR}(\pi)$, la filtration 
     ${\rm Fil}_{\mathcal{L}}$ sur $M_{\rm dR}(\pi)$ est faiblement admissible (cela d\'ecoule facilement du fait que la repr\'esentation du groupe de Weil attach\'ee \`a $M(\pi)$ est irr\'eductible, cf. la preuve du th\'eor\`eme 5.2 de \cite{Breuil-Schneider}). Le th\'eor\`eme de Colmez-Fontaine 
      \cite{CF} permet donc de construire une unique (\`a isomorphisme pr\`es) $L$-repr\'esentation $V_{\mathcal{L}}\in \mathcal{V}(\pi)$
      telle que $\mathcal{L}(V_{\mathcal{L}})=\mathcal{L}$. On d\'eduit alors de \cite[th1.3]{PCD} (cela utilise \cite{Emcomp}) que $V\to \mathcal{L}(V)$ 
      induit une bijection $\mathcal{V}(\pi)\to {\rm Proj}(M_{\rm dR}(\pi))$. 
          
          Pour r\'esumer, si $V\in \mathcal{V}(\pi)$, alors le choix d'un isomorphisme $D_{\rm pst}(V)\simeq M(\pi)$ induit:
          
          $\bullet$ un isomorphisme de $(\varphi,\Gamma)$-modules sur $\mathcal{R}$
  \begin{eqnarray} N_{\rm rig}(V)\simeq N_{\rm rig}(\pi), \label{iso nrig} \end{eqnarray}
  qui induit une identification de $L_n[[t]]$-modules avec action semi-lin\'eaire de $\Gamma$ (pour $n$ assez grand)
        $$N_{\rm dif,n}^+(V)\simeq L_n[[t]]\otimes_{L} M_{\rm dR}(\pi).$$
  
   $\bullet$ un isomorphisme de $L$-espaces vectoriels filtr\'es $D_{\rm dR}(V)\simeq (M_{\rm dR}(\pi), {\rm Fil}_{\mathcal{L}(V)})$. 
  
  $\bullet$ des identifications de $L_n[[t]]$-modules avec action semi-lin\'eaire de $\Gamma$ (pour $n$ assez grand)
     $$D_{\rm dif,n}^+(V)\simeq {\rm Fil}^0(L_n((t))\otimes M_{\rm dR}(\pi))\simeq t N_{\rm dif,n}^+(V)+L_n[[t]]\otimes_{L} \mathcal{L}(V),$$
     
       $\bullet$ des inclusions de $(\varphi,\Gamma)$-modules
        $$tN_{\rm rig}(\pi)\subset D_{\rm rig}(V)\subset N_{\rm rig}(\pi).$$
              
               Tout ceci est une cons\'equence de la discussion ci-dessus et du fait que l'on travaille en poids $0,1$. Toutes ces identifications et inclusions 
               {\it{d\'ependent du choix de l'isomorphisme}} $D_{\rm pst}(V)\simeq M(\pi)$, mais uniquement \`a scalaire pr\`es. Par cons\'equent, {\it{la $L$-droite 
               $\mathcal{L}(V)\subset M_{\rm dR}(\pi)$ est parfaitement bien d\'efinie}}, ainsi que les images dans $N_{\rm rig}(\pi)$, $L_n[[t]]\otimes_{L} M_{\rm dR}(\pi)$, de toutes ces identifications.
\end{example}

\section{Repr\'esentations localement analytiques de $G$ et mod\`ele de Kirillov-Colmez} \label{kirillovcolmez}

   Nous rappelons dans ce chapitre un certain nombre de constructions et r\'esultats concernant 
   la correspondance de Langlands locale $p$-adique, en particulier la th\'eorie du mod\`ele de Kirillov de Colmez 
   \cite[chap. VI]{Cbigone}, qui sera indispensable dans les chapitres suivants. 
   Notre discussion est assez rapide et nous renvoyons le lecteur au paragraphe 5 du chapitre VI de \cite{Cbigone} ou aux chapitres 4 et 5 de \cite{DComp}, o\`u tout ceci est d\'ecrit en d\'etail.

\subsection{$(\varphi,\Gamma)$-modules et repr\'esentations de $G$}

   Pour toute $L$-repr\'esentation $V$ de dimension $2$ de $\mathcal{G}_{\qp}$, Colmez \cite[ch.II,IV,V]{Cbigone}
   construit un faisceau $G$-\'equivariant\footnote{Le groupe $G$ agit sur $\p1(\qp)$ par $\left(\begin{smallmatrix} a & b \\ c & d \end{smallmatrix}\right).x=\frac{ax+b}{cx+d}$.} $U\to D_{\rm rig}(V)\boxtimes U$ sur $\p1(\qp)$, dont les sections sur $\zp$ sont donn\'ees par 
   $D_{\rm rig}(V)\boxtimes\zp=D_{\rm rig}(V)$. L'action du monoïde $P^+=\left(\begin{smallmatrix} \zp - \{0\} & \zp \\ 0 & 1 \end{smallmatrix}\right)$ (qui stabilise $\zp$) sur $D_{\rm rig}(V)$ est
    donn\'ee par\footnote{Rappelons que $a\to \sigma_a$ d\'esigne l'inverse de l'isomorphisme
    $\Gamma\simeq \zpet$ fourni par le caract\`ere cyclotomique.}
    \[  \left(\begin{smallmatrix} p^k a & b \\ 0 & 1 \end{smallmatrix}\right)z = (1+T)^b. \varphi^k(\sigma_a(z)),   \quad \forall z \in D_{\rm rig}(V), k \geq 0, a \in \zp^*, b\in \zp,\]
    alors que l'action de $\left(\begin{smallmatrix} 1 & 0 \\ p\zp & 1 \end{smallmatrix}\right)$ est tr\`es compliqu\'ee. 

    L'espace $D_{\rm rig}(V)\boxtimes \p1$ des sections globales du faisceau attach\'e \`a $V$ est un espace LF avec action continue de $G$. Par construction, le caract\`ere central de $D_{\rm rig}(V)\boxtimes \p1$ est\footnote{Vu comme caract\`ere de $\qpet$ par la th\'eorie du corps de classe local.}
   $$\delta_V=\chi_{\rm cyc}^{-1}\det V.$$ Dans toutes les applications que nous avons en vue, on aura $\det V=\chi_{\rm cyc}$ et donc $\delta_V=1$.
    
    Pour tout ouvert compact 
    $U$ de $\p1(\qp)$ on dispose d'une application de prolongement par z\'ero $D_{\rm rig}(V)\boxtimes U\to D_{\rm rig}(V)\boxtimes\p1$, qui permet d'identifier $D_{\rm rig}(V)\boxtimes U$ \`a un sous-espace de $D_{\rm rig}(V)\boxtimes\p1$. Soit $$w=\left(\begin{smallmatrix} 0 & 1\\ 1 & 0 \end{smallmatrix}\right)\in G,$$ et notons aussi $w$ la restriction \`a $D_{\rm rig}(V)\boxtimes \zpet$ de l'action de l'involution $w$ de $D_{\rm rig}(V)\boxtimes\p1$. Alors $D_{\rm rig}(V)\boxtimes\p1=D_{\rm rig}(V)+wD_{\rm rig}(V)$ et
 l'application $z\to ({\rm Res}_{\zp}(z), {\rm Res}_{\zp}(wz))$ induit une identification
   \[ D_{\rm rig}(V) \boxtimes \p1= \{ (z_1,z_2) \in D_{\rm rig}(V) \times D_{\rm rig}(V)|\,\, \mathrm{Res}_{\zpet}(z_2)=w(\mathrm{Res}_{\zpet}(z_1)) \}. \]

Nous allons utiliser syst\'ematiquement le r\'esultat suivant de Colmez \cite[ch V]{Cbigone}.    

  \begin{theoreme}\label{correspondance}
   Soit $V$ une $L$-repr\'{e}sentation de dimension~$2$ de $\mathcal{G}_{\qp}$.
   
  a) L'action de $G$ sur
   $D_{\rm rig}(V)\boxtimes \p1$ s'\'etend en une structure\footnote{Rappelons que $D(G)$ est l'alg\`ebre des distributions sur $G$.}
de $D(G)$-module topologique. 

b) Si $\check{V}=V^*\otimes \chi_{\rm cyc}\simeq V\otimes \delta_V^{-1}$ est le dual de Cartier de $V$, il existe un accouplement parfait de
$D(G)$-modules topologiques 
 $$\{\,\,\}_{\p1}: (D_{\rm rig}(\check{V})\boxtimes \p1)\times (D_{\rm rig}(V)\boxtimes\p1)\to L,$$
 
 c) Il existe une suite exacte canonique de 
  $D(G)$-modules topologiques 
          $$0\to (\Pi(\check{V})^{\rm an})^*\to D_{\rm rig}(V)\boxtimes \p1\to \Pi(V)^{\rm an}\to 0,$$
      et $(\Pi(\check{V})^{\rm an})^*$ s'identifie ainsi \`a l'orthogonal de $(\Pi(V)^{\rm an})^*\subset D_{\rm rig}(\check{V})\boxtimes\p1$ (ou de $\Pi(V)^*$) dans 
      $D_{\rm rig}(V)\boxtimes\p1$.
  \end{theoreme}
  
  \begin{remarque}
   a) Dans toutes nos applications on aura $\det V=\chi_{\rm cyc}$ et donc 
   $\delta_{V}=1$ et $\check{V}\simeq V$ canoniquement. La suite exacte pr\'ec\'edente devient dans ce cas 
   $$0\to (\Pi(V)^{\rm an})^*\to D_{\rm rig}(V)\boxtimes \p1\to \Pi(V)^{\rm an}\to 0$$
   et $(\Pi(V)^{\rm an})^*$ s'identifie \`a son propre orthogonal dans $D_{\rm rig}(V)\boxtimes \p1$ via l'accouplement 
   $\{\,\,\}_{\p1}: (D_{\rm rig}(V)\boxtimes \p1)\times (D_{\rm rig}(V)\boxtimes \p1)\to L$.
   
   b)  Si $U$ est un ouvert compact de $\p1(\qp)$ et si
$H$ est un sous-groupe ouvert compact de $G$ qui stabilise $U$, l'espace  
$D_{\rm rig}\boxtimes U\subset D_{\rm rig}\boxtimes \p1$ est
stable par $D(H)\subset D(G)$ et ${\rm Res}_U(\lambda\cdot z)=\lambda\cdot {\rm Res}_U(z)$
 pour tout $z\in D_{\rm rig}\boxtimes \p1$ et tout $\lambda\in D(H)$.
  \end{remarque}

\subsection{L'action infinit\'esimale de $G$}

 Soit $U(\mathfrak{gl_2})\subset D(G)$ l'alg\`{e}bre enveloppante de l'alg\`{e}bre de Lie de $G$ (tensoris\'ee avec $L$). 
  La discussion pr\'{e}c\'{e}dente montre l'existence d'une action de $\mathfrak{gl_2}$ sur $D_{\rm rig}\boxtimes \zp=D_{\rm rig}$, qui satisfait ${\rm Res}_{U}(X\cdot z)=X\cdot {\rm Res}_{U}(z)$ pour $z\in D_{\rm rig}\boxtimes \p1$, $X\in U(\mathfrak{gl_2})$ et $U\subset \zp$ ouvert compact. Bien que l'action du Borel soit relativement explicite, l'action de l'involution $w$ est tr\`es compliqu\'ee. Le r\'esultat suivant permet de contourner ce probl\`eme. 
  
Consid\'erons la base $$\quad a^+=\left(\begin{smallmatrix} 1 & 0 \\0 & 0\end{smallmatrix}\right),\quad
a^-=\left(\begin{smallmatrix} 0 & 0 \\0 & 1\end{smallmatrix}\right), \quad 
u^+=\left(\begin{smallmatrix} 0 & 1 \\0 & 0\end{smallmatrix}\right), \quad u^-=\left(\begin{smallmatrix} 0 & 0 \\1 & 0\end{smallmatrix}\right)$$
   de $\mathfrak{gl_2}$, ainsi que l'\'el\'ement de Casimir 
   $$C=u^+u^-+ u^-u^++\frac{1}{2}h^2\in U(\mathfrak{gl}_2),\quad \text{o\`u}\^e\quad h=a^+-a^-=\left(\begin{smallmatrix} 1 & 0 \\0 & -1\end{smallmatrix}\right).$$
     L'\'el\'ement $C$ engendre le centre de $U(\mathfrak{sl}_2)$. On identifie un \'el\'ement $f$ de $\mathcal{R}$ avec l'op\'erateur \og multiplication par $f$\fg{} sur 
  $D_{\rm rig}(V)$ dans l'\'enonc\'e suivant, qui est le r\'esultat principal de \cite{Annalen}.

  \begin{theoreme}\label{Casimirinf} Soient $a$ et $b$ les poids de Hodge-Tate g\'en\'eralis\'es de $V$, et soit $k=a+b$. 
  En tant qu'op\'erateurs sur $D_{\rm rig}(V)$ nous avons
  $$a^+=\nabla, \quad a^-=k-1-\nabla, \quad u^+=t, \quad u^-=-\frac{(\nabla-a)(\nabla-b)}{t}, \quad C=\frac{(a-b)^2-1}{2}.$$
     En particulier, si $a=0$ et $b=1$ on a 
   $$a^+=\nabla=-a^-, \quad C=0, \quad u^+=t, \quad u^-=-\frac{\nabla(\nabla-1)}{t}=-t\partial^{2},$$
   o\`u $\partial=\frac{1}{t}\nabla$, une connexion sur $D_{\rm rig}(V)[1/t]$.
  \end{theoreme}
  
\subsection{Vecteurs $P$-finis et mod\`ele de Kirillov} Rappelons que $P=\left(\begin{smallmatrix} \qpet & \qp \\0 & 1\end{smallmatrix}\right)$. 
On fixe une $L$-repr\'esentation $V$ de dimension $2$ de $\mathcal{G}_{\qp}$ et on suppose que $V$ est {\it{absolument irr\'eductible}}. 

\begin{definition} a) On dit qu'un vecteur $v\in \Pi(V)$ est {\it{$P$-fini}} s'il existe $n,k\geq 1$ tels que 
  $$\left( \left(\begin{smallmatrix} 1 & p^n \\0 & 1\end{smallmatrix}\right)-1\right)^k v=0$$
  et si $L\left [ \left(\begin{smallmatrix} \zpet & 0 \\0 & 1\end{smallmatrix}\right)\right] v$ est de dimension finie sur $L$. 
   On note $\Pi(V)^{P-\rm fini}$ l'espace des vecteurs $P$-finis de $\Pi(V)$.
   
   b) Un vecteur $v\in \Pi(V)^{P-\rm fini}$ est dit {\it{de pente infinie}} s'il existe $n,k\geq 1$ et $m\in \mathbf{Z}$ tels que  
   $$\left(\sum_{i=0}^{p^n-1}   \left(\begin{smallmatrix} 1 & i \\0 & 1\end{smallmatrix}\right) \right)^k\circ \left(\begin{smallmatrix} p^m & 0 \\0 & 1\end{smallmatrix}\right)v=0.$$
   On note $\Pi(V)^{P-\rm fini}_c\subset \Pi(V)^{P-\rm fini}$ l'espace des vecteurs $P$-finis de pente infinie.
   \end{definition}
   
  \begin{remarque}\label{Jacquet compact}  Soit $v\in \Pi(V)^{P-\rm fini}$ et soient $n,k$ comme dans la d\'efinition ci-dessus. 
  Alors pour tout $u\in  \left(\begin{smallmatrix} 1 & \qp \\0 & 1\end{smallmatrix}\right)$ on a 
  $(1-u)^kv\in  \Pi(V)^{P-\rm fini}_c$. Plus pr\'ecis\'ement, n'importe quels
  $m\geq -v_p(x)$ (avec $u=\left(\begin{smallmatrix} 1 & x \\0 & 1\end{smallmatrix}\right)$) et $N\geq m+n$ satisfont 
    $$\left(\sum_{i=0}^{p^N-1}   \left(\begin{smallmatrix} 1 & i \\0 & 1\end{smallmatrix}\right) \right)^k\circ \left(\begin{smallmatrix} p^m & 0 \\0 & 1\end{smallmatrix}\right) (1-u)^k v=0.$$
   La preuve de ce r\'esultat est un exercice amusant laiss\'e au lecteur. Nous utiliserons \`a plusieurs reprises cette observation (avec $k=1$). 
  \end{remarque}

    Soit $X$ un $L_{\infty}[[t]]$-module muni d'une action semi-lin\'eaire de $\Gamma$ (par rapport \`a l'action 
   naturelle de $\Gamma$ sur $L_{\infty}[[t]]=\varinjlim_{n} L_n[[t]]$). 
    On note ${\rm LP}(\qpet, X)^{\Gamma}$ l'espace des fonctions $\phi: \qpet\to X$ \`a support compact
  dans $\qp$ et satisfaisant $\phi(ax)=\sigma_{a}(\phi(x))$ pour tous $x\in\qpet$ et $a\in\zpet$. On note ${\rm LP}_c(\qpet, X)^{\Gamma}$
  le sous-espace de ${\rm LP}(\qpet, X)^{\Gamma}$ form\'e des fonctions nulles au voisinage de $0$. 
      L'application $\phi\to (\phi(p^i))_{i\in\mathbf{Z}}$ induit un isomorphisme $L$-lin\'eaire 
     $${\rm LP}_c(\qpet, X)^{\Gamma}\simeq \bigoplus_{i\in\mathbf{Z}} X.$$

     Rappelons que l'on a fix\'e un syst\`eme compatible $(\zeta_{p^n})_{n\geq 1}$ de racines de l'unit\'e, ce qui permet de d\'efinir un caract\`ere additif localement constant 
     $$\varepsilon: \qp\to \mu_{p^{\infty}}, \quad \varepsilon(b)=\zeta_{p^n}^{p^nb}, \,\ \forall n\geq -v_p(b).$$
    On munit les espaces ${\rm LP}(\qpet, X)^{\Gamma}$ et ${\rm LP}_c(\qpet, X)^{\Gamma}$
 d'une action de $P$, d\'efinie par 
      $$\left(\left(\begin{smallmatrix} a & b \\0 & 1\end{smallmatrix}\right)\phi \right) (x)=\varepsilon(bx)e^{tbx}\phi(ax).$$

    \begin{proposition}\label{inclusion Kir}
      Les sous-espaces 
     $\Pi(V)^{P-\rm fini}_c$ et $\Pi(V)^{P-\rm fini}$ sont stables sous l'action de $P$ et il existe une injection
     $P$-\'equivariante canonique $v\mapsto \phi_v$
     $$ \Pi(V)^{P-\rm fini}\to {\rm LP}(\qpet, D_{\rm dif}^-(V))^{\Gamma}, \quad \text{o\`u} \quad D_{\rm dif}^-(V)=\varinjlim_{n} D_{\rm dif,n}^+(V)[1/t]/D_{\rm dif,n}^+(V),$$
     qui induit un isomorphisme 
     $$\Pi(V)^{P-\rm fini}_c\simeq {\rm LP}_c(\qpet, D_{\rm dif}^-(V))^{\Gamma}.$$
     En particulier $v\to (\phi_v(p^i))_{i\in\mathbf{Z}}$ induit une bijection
    $$\Pi(V)^{P-\rm fini}_c\to \bigoplus_{i\in\mathbf{Z}} D_{\rm dif}^-(V).$$
    \end{proposition}
    
    \begin{proof} Ces constructions ont \'et\'e introduites par Colmez dans le paragraphe VI.5 de \cite{Cbigone}. Leur extension sous la forme de la proposition \ref{inclusion Kir} (qui ne demande aucune id\'ee suppl\'ementaire) se trouve dans le chapitre 4 de \cite{DComp}, plus pr\'ecis\'ement les propositions 4.6, 4.8, 4.11, le lemme 4.12 et le corollaire 4.13 de loc.cit.  
        \end{proof}
  
   \subsection{Dualit\'e et mod\`ele de Kirillov} 
   
     Le r\'esultat suivant (th\'eor\`eme \ref{dualKir}) de Colmez est crucial, mais demande pas mal de pr\'eliminaires. 
     L'accouplement naturel $\check{V}\times V\to L(1)$ induit par fonctorialit\'e un accouplement 
     $$\langle \,\,\rangle: D_{\rm dif}^+(\check{V})[1/t]\times D_{\rm dif}^+(V)[1/t]\to L_{\infty}((t))dt,$$
     o\`u l'on note $dt$ la base canonique de $\qp(1)$. On d\'efinit un accouplement 
     $$\{\,\,\}_{\rm dif}: D_{\rm dif}^+(\check{V})[1/t]\times D_{\rm dif}^+(V)[1/t]\to L$$ en posant $$\quad \{\check{z}, z\}_{\rm dif}=\lim_{n\to\infty} \frac{1}{p^n} {\rm res}_{0}\left({\rm Tr}_{L_n((t))/L((t))}(\langle \sigma_{-1}(\check{z}), z\rangle)\right),$$
    o\`u ${\rm res}_0((\sum_{n\gg-\infty} a_n t^n)dt)=a_{-1}$.
    
     \begin{proposition} \label{accouplement}
         a) $\{\,\,\}_{\rm dif}$ est un accouplement $\Gamma$-\'equivariant parfait entre $ D_{\rm dif}^+(\check{V})[1/t]$ et $D_{\rm dif}^+(V)[1/t]$.
                  Pour tout $n$ assez grand l'orthogonal de $D_{\rm dif,n}^+(\check{V})$ est $D_{\rm dif,n}^+(V)$. Ainsi, 
         $\{\,\,\}_{\rm dif}$ induit un accouplement parfait entre $D_{\rm dif}^+(\check{V})$ et $D_{\rm dif}^{-}(V)$. 
         
         b) Si $V$ est de de Rham \`a poids de Hodge-Tate $0$ et $k\geq 1$, alors pour tout $n$ assez grand l'orthogonal de 
         $N_{\rm dif,n}^+(\check{V})$ est $t^k N_{\rm dif,n}^+(V)$.  
     \end{proposition}
     
     \begin{proof}
      Ce sont des traductions \'el\'ementaires, voir par exemple la discussion qui pr\'ec\`ede le lemme VI.3.3 de \cite{Cbigone}, ainsi que le lemme 
      VI.4.16 de loc.cit. 
           \end{proof}
     
      D'apr\`es le corollaire VI.13 de \cite{CD}, il existe $m(V)$ assez grand tel que l'inclusion 
      $(\Pi(V)^{\rm an})^*\subset D_{\rm rig}(\check{V})\boxtimes\p1$ se factorise \`a travers 
      $$D^{]0,r_{m(V)}]}(\check{V})\boxtimes \p1:=\{z\in D_{\rm rig}(\check{V})\boxtimes \p1| \, {\rm Res}_{\zp}(z), {\rm Res}_{\zp}(wz)\in D^{]0,r_{m(V)}]}(\check{V})\},$$
      ce qui nous permet de d\'efinir pour $n\geq m(V)$ et $j\in\mathbf{Z}$ 
      $$i_{j,n}: (\Pi(V)^{\rm an})^*\to D_{\rm dif, n}^+(\check{V}), \quad i_{j,n}=\varphi^{-n}\circ {\rm Res}_{\zp}\circ  \left(\begin{smallmatrix} p^{n-j} & 0\\0 & 1\end{smallmatrix}\right).$$
      
      Rappelons qu'on dispose d'un accouplement canonique $G$-\'equivariant parfait $\{\,\,\}_{\p1}$ entre
      $D_{\rm rig}(\check{V})\boxtimes\p1$ et $D_{\rm rig}(V)\boxtimes \p1$, qui induit l'accouplement naturel entre 
      $(\Pi(V)^{\rm an})^*\subset D_{\rm rig}(\check{V})\boxtimes\p1$ et $\Pi(V)^{\rm an}=(D_{\rm rig}(V)\boxtimes\p1)/(\Pi(\check{V})^{\rm an})^*$. 
      Enfin, la proposition \ref{inclusion Kir} fournit un isomorphisme $v\to \phi_v$ entre
      $\Pi(V)^{P-\rm fini}_c$ et ${\rm LP}_c(\qpet, D_{\rm dif}^{-}(V))^{\Gamma}$, ce qui permet de donner un sens \`a l'\'egalit\'e ci-dessous:

 \begin{theoreme}\label{dualKir}
 On a $\Pi(V)^{P-\rm fini}_c\subset \Pi(V)^{\rm an}$ et pour tous $n\geq m(V)$, $v\in \Pi(V)^{P-\rm fini}_c$
 et $l\in (\Pi(V)^{\rm an})^*$ on a 
   $$\{l,v\}_{\p1}=\sum_{j\in\mathbf{Z}} \{i_{j,n}(l), \phi_v(p^{-j})\}_{\rm dif}.$$
 \end{theoreme}
 
 \begin{proof}
  Cette g\'en\'eralisation de la proposition VI.5.12 de \cite{Cbigone} est d\'emontr\'ee (de mani\`ere diff\'erente) dans 
  \cite[th. 5.3]{DComp}. 
 \end{proof}
 
 \subsection{Une description utile de $\Pi^{\rm lisse}$} Les deux r\'esultats techniques suivants seront utilis\'es constamment dans le chapitre suivant. 

\begin{proposition} \label{P-lisse} Soit $V\in \mathcal{V}(\pi)$ et notons $$\Pi(V)^{P-\rm lisse}_c=\{v\in \Pi(V)^{P-\rm fini}_c| \, u^+v=a^+v=0\}.$$
Alors $\Pi(V)^{P-\rm lisse}_c\subset \Pi(V)^{\rm lisse}$ et l'isomorphisme $\Pi(V)^{P-\rm fini}_c\simeq {\rm LP}_c(\qpet, D_{\rm dif}^{-}(V))^{\Gamma}$ induit un isomorphisme 
 de $P$-modules 
 $$\Pi(V)^{P-\rm lisse}_c\simeq {\rm LP}_c(\qpet, N_{\rm dif}^+(V)/D_{\rm dif}^+(V))^{\Gamma}.$$
\end{proposition}

\begin{proof} L'espace $\Pi(V)^{P-\rm lisse}_c$ est not\'e $\Pi_{c}^{P-\rm alg}$ dans \cite{Annalen} (en prenant $k=1$ dans loc.cit.). 
L'inclusion $\Pi(V)^{P-\rm lisse}_c\subset \Pi(V)^{\rm lisse}$ d\'ecoule alors du th\'eor\`eme 5.6 de loc.cit (c'est une cons\'equence facile des th\'eor\`emes
\ref{Casimirinf} et \ref{dualKir}, combin\'es avec la proposition \ref{accouplement}). La deuxi\`eme partie d\'ecoule de la proposition 
5.4 de loc.cit (et se d\'eduit aussi de la proposition \ref{inclusion Kir} et de l'\'egalit\'e 
 $$ (t^{-1} D_{\rm dif}^+(V)/D_{\rm dif}^+(V))^{\nabla=0}=N_{\rm dif}^+(V)/D_{\rm dif}^+(V),$$
 qui se d\'emontre sans aucun probl\`eme). 
\end{proof}
   
   \begin{remarque}
    En se rappelant que $$N_{\rm dif}^+(V)=L_{\infty}[[t]]\otimes_{L} D_{\rm dR}(V) \quad \text{et}\quad D_{\rm dif}^+(V)=tN_{\rm dif}^+(V)+L_{\infty}[[t]]\otimes_{L} {\rm Fil}^0(D_{\rm dR}(V)),$$
    on obtient un isomorphisme canonique $\Gamma$-\'equivariant 
    $$N_{\rm dif}^+(V)/D_{\rm dif}^+(V)\simeq L_{\infty}\otimes_{L} D_{\rm dR}(V)/{\rm Fil}^0(D_{\rm dR}(V)).$$
    L'action de $\Gamma$ sur le terme de droite \'etant lisse, on en d\'eduit que ${\rm LP}_c(\qpet, N_{\rm dif}^+(V)/D_{\rm dif}^+(V))^{\Gamma}$ n'est rien d'autre que 
    l'espace des fonctions localement constantes $\phi: \qpet\to  L_{\infty}\otimes_{L} D_{\rm dR}(V)/{\rm Fil}^0(D_{\rm dR}(V))$ \`a support compact dans $\qpet$ 
    et telles que $\phi(ax)=\sigma_{a}(\phi(x))$ pour $x\in\qpet$ et $a\in \zpet$. On a un isomorphisme naturel (utiliser le th\'eor\`eme de Hilbert 90) 
    $${\rm LP}_c(\qpet, L_{\infty}\otimes_{L} D_{\rm dR}(V)/{\rm Fil}^0(D_{\rm dR}(V)))^{\Gamma} \otimes_L L_{\infty} \simeq {\rm LC}_c(\qpet, L_{\infty}\otimes_{L} D_{\rm dR}(V)/{\rm Fil}^0(D_{\rm dR}(V))),$$ et le terme de droite (avec son action naturelle de $P$, d\'efinie en utilisant le caract\`ere additif $\varepsilon$) est le mod\`ele de Kirillov usuel de $\pi\otimes_{L} L_{\infty}$, ce qui explique le nom de ce chapitre. 
           \end{remarque}
  
  \begin{theoreme}\label{Ann} 
   Si $V\in \mathcal{V}(\pi)$, alors 
   $$\Pi(V)^{\rm lisse}=\{v\in \Pi(V)^{\rm an} ~ | ~ u^+v=a^+v=0\}=\Pi(V)^{P-\rm lisse}_c$$
   et on a un isomorphisme canonique de $P$-modules 
    $$\Pi(V)^{\rm lisse}\simeq {\rm LC}_c(\qpet, N_{\rm dif}^+(V)/D_{\rm dif}^+(V))^{\Gamma}.$$
\end{theoreme}

\begin{proof} Notons $\Pi=\Pi(V)$. Commençons par montrer la premi\`ere \'egalit\'e. Une inclusion \'etant \'evidente, supposons que
$v\in \Pi^{\rm an}$ est tu\'e par $u^+$ et $a^+$, et montrons que $v$ est tu\'e par 
$u^-$. Soit $x\in \qp$ et soit $v_x=\left(\begin{smallmatrix} 1 & x \\0 & 1\end{smallmatrix}\right)v-v$, de telle sorte que 
$u^+v_x=0$ et 
$$a^+v_x=a^+\left(\begin{smallmatrix} 1 & x \\0 & 1\end{smallmatrix}\right)v=\left(\begin{smallmatrix} 1 & x \\0 & 1\end{smallmatrix}\right)(a^+v+xu^+v)=0.$$
On en d\'eduit\footnote{Noter que $v_x\in \Pi(V)_{c}^{P-\rm fini}$ gr\^ace \`a la remarque \ref{Jacquet compact}.} que $v_x\in \Pi^{P-\rm lisse}_c$ et donc, gr\^ace \`a la proposition \ref{P-lisse}, 
$u^{-}v_x=0$. L'identit\'e 
$$u^-\left(\begin{smallmatrix} 1 & x \\0 & 1\end{smallmatrix}\right)=\left(\begin{smallmatrix} 1 & x \\0 & 1\end{smallmatrix}\right)(-x^2u^++u^--xh)$$
et le fait que $v$ est tu\'e par $u^+$ et $h=2a^+$ montrent que $u^-v_x=(\left(\begin{smallmatrix} 1 & x \\0 & 1\end{smallmatrix}\right)-1)u^-v$. Ainsi, 
$u^-v\in \Pi^{\left(\begin{smallmatrix} 1 & \qp \\0 & 1\end{smallmatrix}\right)}=0$, la derni\`ere \'egalit\'e \'etant une cons\'equence de
\cite[lemme 7.1]{DComp} (c'est un r\'esultat \'el\'ementaire). Cela montre la premi\`ere \'egalit\'e. 

  Pour conclure la premi\`ere partie, il reste \`a prouver l'\'egalit\'e $\Pi(V)^{P-\rm lisse}_c=\Pi(V)^{\rm lisse}$. Une inclusion est fournie par la proposition \ref{P-lisse}. 
  L'autre inclusion vient du fait que $\Pi(V)^{\rm lisse}\simeq \pi$ est supercuspidale, donc tout vecteur de $\Pi(V)^{\rm lisse}$ est combinaison lin\'eaire de vecteurs du type
  $(1-u)v$ avec $u\in \left(\begin{smallmatrix} 1 & \qp \\0 & 1\end{smallmatrix}\right)$ et $v\in \Pi(V)^{\rm lisse}$. On conclut en utilisant l'inclusion
  $(1-u)\Pi(V)^{P-\rm lisse}\subset \Pi(V)^{P-\rm lisse}_c$ (remarque \ref{Jacquet compact}). 
  
   La deuxi\`eme partie s'obtient en combinant ce qu'on vient de d\'emontrer avec la proposition \ref{P-lisse}.
  \end{proof}

    \section{Le $G$-module $\Pi(\pi,0)$}\label{independance du quotient}

   Le but de ce chapitre est d'expliquer la preuve du th\'eor\`eme suivant, d\^u \`a Colmez \cite{Cbigone, Colmezpoids}.
   
   \begin{theoreme}
     Soient $V_1,V_2\in \mathcal{V}(\pi)$. Le choix d'isomorphismes 
     $D_{\rm pst}(V_1)\simeq M(\pi)$ et $D_{\rm pst}(V_2)\simeq M(\pi)$ induit un isomorphisme de 
     $G$-modules topologiques :
        $$\Pi(V_1)^{\rm an}/\Pi(V_1)^{\rm lisse}\simeq \Pi(V_2)^{\rm an}/\Pi(V_2)^{\rm lisse}.$$
   \end{theoreme}
  
     Nous allons en fait d\'emontrer un r\'esultat nettement plus pr\'ecis, et construire 
      un isomorphisme explicite, ce qui est indispensable pour nos besoins. Comme la construction demande un certain nombre de pr\'eliminaires techniques, nous renvoyons le lecteur au th\'eor\`eme \ref{almost canonique}. Nous avons cherch\'e \`a distinguer le plus soigneusement possible les identifications parfaitement canoniques de celles qui ne le sont qu'\`a scalaire pr\`es, ce qui alourdit un peu la r\'edaction, mais \'evite tout risque de confusion. 

\subsection{Points fixes de $\psi$ et le $G$-module $tN_{\rm rig}(V)\boxtimes \p1$}\label{points fixes}
Soit $V\in \mathcal{V}(\pi)$. Notons 
$$tN_{\rm rig}(V)\boxtimes \p1=\{z\in D_{\rm rig}(V)\boxtimes \p1| \,\ {\rm Res}_{\zp}(z), {\rm Res}_{\zp}(wz)\in tN_{\rm rig}(V)\}$$
et d\'efinissons d'une mani\`ere similaire $tN^{]0,r_n]}(V)\boxtimes \p1$ pour $n$ assez grand. Il n'est pas clair {\it{\`a priori}} que 
$tN_{\rm rig}(V)\boxtimes \p1$ soit stable sous l'action de $G$, mais nous allons voir que c'est en effet le cas. Ce r\'esultat a d\'ej\`a \'et\'e d\'emontr\'e de mani\`ere tr\`es d\'etourn\'ee dans le chapitre VI de 
\cite{Cbigone}, puis d'une mani\`ere compl\`etement diff\'erente dans \cite{Colmezpoids}. Nous en donnons une preuve diff\'erente.

\begin{proposition}\label{inclusion rayon}
  Pour tout $V\in \mathcal{V}(\pi)$ il existe $n$ tel que 
  l'inclusion\footnote{On identifie implicitement ici $\check{V}$ et $V$.} $(\Pi(V)^{\rm an})^*\subset D_{\rm rig}(V)\boxtimes \p1$ induise une inclusion 
  $$(\Pi(V)^{\rm an}/\Pi(V)^{\rm lisse})^*\subset tN^{]0,r_n]}(V)\boxtimes \p1.$$
\end{proposition}

\begin{proof} Soit $m(V)$ comme dans la discussion qui suit la proposition \ref{accouplement} et soit 
$l\in (\Pi(V)^{\rm an}/\Pi(V)^{\rm lisse})^*$, vu comme \'el\'ement de $(\Pi(V)^{\rm an})^*$ nul sur 
$\Pi(V)^{\rm lisse}$. En combinant l'isomorphisme 
$\Pi(V)^{\rm lisse}\simeq {\rm LC}_c(\qpet, N_{\rm dif}^+(V)/D_{\rm dif}^+(V))^{\Gamma}$ 
(th\'eor\`eme \ref{Ann}) avec le th\'eor\`eme 
\ref{dualKir} et la proposition \ref{accouplement}, on obtient $i_{j,n}(l)\in tN_{\rm dif,n}^+(V)$  
pour $n\geq m(V)$ et $j\in \mathbf{Z}$. En particulier (en prenant $j=n$) $\varphi^{-n}({\rm Res}_{\zp}(l))\in tN_{\rm dif,n}^+(V)$ pour 
$n\geq m(V)$ et donc ${\rm Res}_{\zp}(l)\subset tN^{]0,r_{m(V)}]}(V)$ (cf. l'exemple \ref{exemple fondamental}), ce qui permet de conclure (en remplaçant $l$ par
$wl$). 
\end{proof}

   Rappelons que si $\Delta$ est un $(\varphi,\Gamma)$-module sur $\mathcal{R}$, il existe un unique op\'erateur
   $\psi$ sur $\Delta$ qui commute avec $\Gamma$, s'annule sur $\sum_{i=1}^{p-1} (1+T)^i\varphi(\Delta)$ et satisfait $\psi\circ \varphi={\rm id}$.
   Soit maintenant $V\in \mathcal{V}(\pi)$ et identifions comme toujours $\check{V}$ avec $V$.
   Par construction, on a 
   $${\rm Res}_{\zp}\circ \left(\begin{smallmatrix} p^{-1} & 0 \\0 & 1\end{smallmatrix}\right)=\psi\circ {\rm Res}_{\zp}\quad \text{sur}\quad D_{\rm rig}(V)\boxtimes\p1.$$
    Ainsi, l'inclusion
$(\Pi(V)^{\rm an})^*\subset D_{\rm rig}(V)\boxtimes \p1$ compos\'ee avec ${\rm Res}_{\zp}$ induit une inclusion 
$$[(\Pi(V)^{\rm an})^*]^{\left(\begin{smallmatrix} p & 0 \\0 & 1\end{smallmatrix}\right)=1}\subset D_{\rm rig}(V)^{\psi=1}.$$

\begin{theoreme} \label{psi egal un} L'inclusion pr\'ec\'edente est un isomorphisme de $D(\Gamma)$-modules et induit un isomorphisme de $D(\Gamma)$-modules
$$[(\Pi(V)^{\rm an}/\Pi(V)^{\rm lisse})^*]^{\left(\begin{smallmatrix} p & 0 \\0 & 1\end{smallmatrix}\right)=1}\simeq (tN_{\rm rig}(V))^{\psi=1}.$$
 En composant avec ${\rm Res}_{\zpet}$ on obtient un isomorphisme de $D(\Gamma)$-modules
 $$[(\Pi(V)^{\rm an}/\Pi(V)^{\rm lisse})^*]^{\left(\begin{smallmatrix} p & 0 \\0 & 1\end{smallmatrix}\right)=1}\simeq (1-\varphi)(tN_{\rm rig}(V))^{\psi=1}.$$
\end{theoreme}

\begin{proof} La seconde partie est une cons\'equence de la premi\`ere, de l'\'egalit\'e ${\rm Res}_{\zpet}=1-\varphi$ sur $(tN_{\rm rig}(V))^{\psi=1}$
et du fait que $D_{\rm rig}(V)^{\varphi=1}=0$, car $V$ n'est pas trianguline. Commençons par montrer que 
$[(\Pi(V)^{\rm an})^*]^{\left(\begin{smallmatrix} p & 0 \\0 & 1\end{smallmatrix}\right)=1}\subset D_{\rm rig}(V)^{\psi=1}$ est un isomorphisme. 
Soit $D(V)$ le $(\varphi,\Gamma)$-module \'etale sur l'anneau de Fontaine $\mathcal{E}$ attach\'e \`a $V$ par l'\'equivalence de cat\'egories de Fontaine 
\cite{FoGrot}. 
D'apr\`es \cite[prop. V.1.18]{Cbigone} l'application naturelle $D(\Gamma)\otimes_{\Lambda(\Gamma)} D(V)^{\psi=1}\to D_{\rm rig}(V)^{\psi=1}$ est un isomorphisme de
$D(\Gamma)$-modules, o\`u $\Lambda(\Gamma)$ est l'alg\`ebre des mesures sur $\Gamma$ \`a valeurs dans $L$. Or \cite[remarque V.14]{CD} l'application 
${\rm Res}_{\zp}$ induit un isomorphisme de $\Lambda(\Gamma)$-modules $$(\Pi(V)^*)^{\left(\begin{smallmatrix} p & 0 \\0 & 1\end{smallmatrix}\right)=1}\simeq D(V)^{\psi=1}.$$
Soit alors $z\in D_{\rm rig}(V)^{\psi=1}$ et \'ecrivons 
$$z=\sum_{i=1}^n \int_{\Gamma} \gamma(x_i)\mu_i(\gamma)$$
avec $x_i\in D(V)^{\psi=1}$ et $\mu_i\in D(\Gamma)$. Si $l_i\in (\Pi^*)^{\left(\begin{smallmatrix} p & 0 \\0 & 1\end{smallmatrix}\right)=1}$ satisfont 
${\rm Res}_{\zp}(l_i)=x_i$, alors 
$$z={\rm Res}_{\zp}(l), \quad \text{o\`u} \quad l=\sum_{i=1}^n \int_{\Gamma} \left(\begin{smallmatrix} \chi_{\rm cyc}(\gamma) & 0 \\0 & 1\end{smallmatrix}\right)l_i\mu_i(\gamma)\in 
[(\Pi(V)^{\rm an})^*]^{\left(\begin{smallmatrix} p & 0 \\0 & 1\end{smallmatrix}\right)=1},$$ ce qui permet de conclure. 

  Ensuite, la proposition \ref{inclusion rayon} montre que ${\rm Res}_{\zp}$ induit une inclusion 
$$[(\Pi(V)^{\rm an}/\Pi(V)^{\rm lisse})^*]^{\left(\begin{smallmatrix} p & 0 \\0 & 1\end{smallmatrix}\right)=1}\subset (tN_{\rm rig}(V))^{\psi=1}.$$
Soit $z\in (tN_{\rm rig}(V))^{\psi=1}$. D'apr\`es ce que l'on vient de faire, il existe $l\in [(\Pi(V)^{\rm an})^*]^{\left(\begin{smallmatrix} p & 0 \\0 & 1\end{smallmatrix}\right)=1}$
tel que ${\rm Res}_{\zp}(l)=z$. Nous allons montrer que $l$ s'annule sur $\Pi(V)^{\rm lisse}=\Pi(V)^{P-\rm lisse}_c$ (l'\'egalit\'e d\'ecoule du th\'eor\`eme \ref{Ann}), ce qui permettra de conclure. Soit 
$a\geq m(V)$ tel que $z\in D^{]0,r_a]}(V)$. 
Comme $\left(\begin{smallmatrix} p & 0 \\0 & 1\end{smallmatrix}\right)l=l$, on a $i_{j,n}(l)=\varphi^{-n}(z)\in tN_{\rm dif, n}^+(V)$ pour
$n\geq a$ et $j\in \mathbf{Z}$. Le r\'esultat s'obtient alors en combinant les propositions \ref{accouplement} et \ref{P-lisse} avec le th\'eor\`eme
\ref{dualKir}. 
\end{proof}
  
\begin{proposition}\label{stable par w} Pour tout $V\in \mathcal{V}(\pi)$ 

  a) Le sous-espace $tN_{\rm rig}(V)\boxtimes \zpet:=(tN_{\rm rig}(V))^{\psi=0}$ est stable sous l'involution
   $w=\left(\begin{smallmatrix} 0 & 1 \\1 & 0\end{smallmatrix}\right)$ de $D_{\rm rig}(V)\boxtimes \zpet=D_{\rm rig}(V)^{\psi=0}$. 
   
   b) Le sous-espace $tN_{\rm rig}(V)\boxtimes \p1$ de $D_{\rm rig}(V)\boxtimes \p1$ est stable sous l'action de $G$. 
\end{proposition}

\begin{proof} a) Le caract\`ere central de $\Pi(V)^{\rm an}$ \'etant trivial, l'involution $w$ de $(\Pi(V)^{\rm an}/\Pi(V)^{\rm lisse})^*$ laisse stable
$[(\Pi(V)^{\rm an}/\Pi(V)^{\rm lisse})^*]^{\left(\begin{smallmatrix} p & 0 \\0 & 1\end{smallmatrix}\right)=1}$. 
La derni\`ere partie du th\'eor\`eme \ref{psi egal un} entra\^ine la stabilit\'e de $(1-\varphi)(tN_{\rm rig}(V))^{\psi=1}$ par 
$w$. Ensuite, $tN_{\rm rig}(V)\boxtimes \zpet=(tN_{\rm rig})^{\psi=0}$ est engendr\'e comme 
$D(\Gamma)$-module par  $(1-\varphi)(tN_{\rm rig}(V))^{\psi=1}$ (cela suit par exemple de \cite[prop. 4.3.8]{KPL}), ce qui permet de conclure, en utilisant le fait que 
$w\circ \sigma_{a}=\sigma_{a^{-1}}\circ w$ sur $D_{\rm rig}(V)\boxtimes \zpet$ (qui suit de l'\'egalit\'e 
$w\left(\begin{smallmatrix} a & 0 \\0 & 1\end{smallmatrix}\right)=\left(\begin{smallmatrix} a & 0 \\0 & a\end{smallmatrix}\right)\left(\begin{smallmatrix} a^{-1} & 0 \\0 & 1\end{smallmatrix}\right)w$).

b) C'est une cons\'equence formelle de la stabilit\'e de $tN_{\rm rig}(V)$ par $P^+=\left(\begin{smallmatrix} \zp\setminus\{0\} & \zp \\0 & 1\end{smallmatrix}\right)$ et du point a).
\end{proof}

\subsection{La repr\'esentation $\Pi(\pi,0)$} Nous avons besoin du r\'esultat suivant de Colmez, qui permet de se d\'ebarrasser de la d\'ependance en la filtration de Hodge dans les constructions pr\'ec\'edentes.  

   \begin{theoreme} \label{independence day}    Soit $V\in \mathcal{V}(\pi)$. Choisissons un isomorphisme $D_{\rm pst}(V)\simeq M(\pi)$, induisant un isomorphisme 
     $N_{\rm rig}(V)\simeq N_{\rm rig}(\pi)$. L'involution de $(tN_{\rm rig}(\pi))^{\psi=0}$ induite (proposition \ref{stable par w}) par l'involution 
     $w$ de $(tN_{\rm rig}(V))^{\psi=0}$ ne d\'epend ni du choix de $V\in \mathcal{V}(\pi)$ ni du choix de l'isomorphisme 
     $D_{\rm pst}(V)\simeq M(\pi)$.
  \end{theoreme}
  
  \begin{proof}
    L'ind\'ependance par rapport au choix de l'isomorphisme $D_{\rm pst}(V)\simeq M(\pi)$ est claire, car deux tels isomorphismes diff\`erent par un scalaire. 
    L'ind\'ependance par rapport au choix de $V\in \mathcal{V}(\pi)$ a \'et\'e d\'emontr\'ee (par voie tr\`es d\'etourn\'ee) dans le paragraphe 9 du chapitre VI de \cite{Cbigone} (se rappeler que les \'el\'ements de $\mathcal{V}(\pi)$ sont classifi\'es par la filtration de Hodge
    sur $M_{\rm dR}(\pi)$). Le lecteur trouvera une preuve nettement plus simple dans \cite{Colmezpoids}, qui exploite la description explicite de l'action infinit\'esimale de 
  $G$ sur $tN_{\rm rig}(V)\boxtimes \p1$.
  \end{proof}

   Notons encore $w$ l'involution de $(tN_{\rm rig}(\pi))^{\psi=0}$ obtenue dans le th\'eor\`eme \ref{independence day}. L'existence de cette involution 
   combin\'ee avec le fait que $tN_{\rm rig}(\pi)$ est un 
    $(\varphi,\Gamma)$-module sur $\mathcal{R}$ permettent de copier les constructions usuelles de Colmez (voir le paragraphe 2 du chapitre II de \cite{Cbigone}) 
    et d'obtenir un faisceau $G$-\'equivariant $U\to tN_{\rm rig}(\pi)\boxtimes U$ sur $\p1(\qp)$ dont les sections sur $\zp$ sont 
    $tN_{\rm rig}(\pi)$, muni de l'action canonique de $P^+$ d\'efinie par 
    $$\left(\begin{smallmatrix} p^k a & b \\0 & 1\end{smallmatrix}\right)z=(1+T)^b\sigma_a(\varphi^k(z)).$$
    Notons qu'en copiant ces constructions on obtient {\it{\`a priori}} 
  seulement un faisceau \'equivariant sous l'action du groupe libre engendr\'e par $P^+$ et $w$, mais toutes les relations qui ont lieu dans $G$ sont satisfaites aussi par les sections de ce faisceau, car c'est le cas pour le faisceau attach\'e \`a $D_{\rm rig}(V)$ (pour un $V\in \mathcal{V}(\pi)$ quelconque) et car $tN_{\rm rig}(V)\subset D_{\rm rig}(V)$. 
  
     Soit $V\in \mathcal{V}(\pi)$. Fixons 
  un isomorphisme $D_{\rm pst}(V)\simeq M(\pi)$, qui induit un isomorphisme $tN_{\rm rig}(V)\simeq tN_{\rm rig}(\pi)$. Cet isomorphisme s'\'etend par construction en un isomorphisme 
  $G$-\'equivariant $$tN_{\rm rig}(V)\boxtimes \p1\simeq tN_{\rm rig}(\pi)\boxtimes \p1,$$  {\it{qui d\'epend du choix de l'isomorphisme 
   $D_{\rm pst}(V)\simeq M(\pi)$, mais uniquement \`a scalaire pr\`es}}. Ensuite, par construction de $tN_{\rm rig}(V)\boxtimes \p1$, on dispose d'une inclusion $G$-\'equivariante canonique
   $$ tN_{\rm rig}(V)\boxtimes \p1\subset D_{\rm rig}(V)\boxtimes \p1.$$
   On en d\'eduit une inclusion $G$-\'equivariante
   $$tN_{\rm rig}(\pi)\boxtimes \p1\subset D_{\rm rig}(V)\boxtimes \p1,$$
   {\it{canonique \`a scalaire pr\`es}}; son image est donc un sous-$G$-module canonique de
   $D_{\rm rig}(V)\boxtimes \p1$. Enfin, la proposition 
   \ref{inclusion rayon} fournit une inclusion canonique 
    $(\Pi(V)^{\rm an}/\Pi(V)^{\rm lisse})^*\subset tN_{\rm rig}(V)\boxtimes \p1$, donc une inclusion $G$-\'equivariante 
    $$(\Pi(V)^{\rm an}/\Pi(V)^{\rm lisse})^*\subset tN_{\rm rig}(\pi)\boxtimes \p1,$$
    {\it{canonique \`a scalaire pr\`es}}.
   
     Apr\`es ces pr\'eliminaires un peu p\'edants mais malheureusement n\'ecessaires, on peut d\'emontrer le r\'esultat suivant, d\^u \`a Colmez \cite{Cbigone, Colmezpoids}. Notre preuve est diff\'erente de celles trouv\'ees dans loc.cit., mais elle s'inspire fortement d'une astuce que l'on peut trouver dans \cite{Colmezpoids}. 
     
     \begin{theoreme}\label{almost canonique}
       Soient $V_1, V_2\in \mathcal{V}(\pi)$. On a 
       $$(\Pi(V_1)^{\rm an}/\Pi(V_1)^{\rm lisse})^*=(\Pi(V_2)^{\rm an}/\Pi(V_2)^{\rm lisse})^*$$ \`a l'int\'erieur de $tN_{\rm rig}(\pi)\boxtimes \p1$. 
       Ainsi, il existe un isomorphisme canonique \`a scalaire pr\`es $\Pi(V_1)^{\rm an}/\Pi(V_1)^{\rm lisse}\simeq \Pi(V_2)^{\rm an}/\Pi(V_2)^{\rm lisse}$. 
     \end{theoreme}
     
         \begin{proof} La seconde assertion est une cons\'equence de la premi\`ere et de la discussion qui pr\'ec\`ede le th\'eor\`eme \ref{almost canonique}. 
         Notons pour simplifier $\Delta_j=D_{\rm rig}(V_j)$ et $N_{\rm rig}=N_{\rm rig}(\pi)$, ainsi que $\Pi_j=\Pi(V_j)$. Fixons des identifications $N_{\rm rig}(V_j)=N_{\rm rig}$
         et regardons $(\Pi_1^{\rm an}/\Pi_1^{\rm lisse})^*$ comme un sous-$D(G)$-module de $\Delta_2\boxtimes \p1$, via la compos\'ee d'inclusions
$(\Pi_1^{\rm an}/\Pi_1^{\rm lisse})^*\subset tN_{\rm rig}\boxtimes \p1\subset \Delta_2\boxtimes \p1$ construites ci-dessus.  

   Montrons d'abord que $(\Pi_1^{\rm an}/\Pi_1^{\rm lisse})^*$ est contenu dans 
$(\Pi_2^{\rm an})^*\subset \Delta_2\boxtimes \p1$. On d\'eduit du th\'eor\`eme de Hahn-Banach et du caract\`ere r\'eflexif de $\Pi_1^{\rm an}/\Pi_1^{\rm lisse}$
que $\mathfrak{g}(\Pi_1^{\rm an})^*$ est dense dans $(\Pi_1^{\rm an}/\Pi_1^{\rm lisse})^*$, et puisque $(\Pi_2^{\rm an})^*$ est ferm\'e dans $\Delta_2\boxtimes \p1$, il suffit de montrer que 
$\mathfrak{g}(\Pi_1^{\rm an})^*\subset (\Pi_2^{\rm an})^*$, et m\^eme 
$u^+ (\Pi_1^{\rm an})^*\subset  (\Pi_2^{\rm an})^*$ (si cela est vrai, on d\'eduit le r\'esultat pour $u^-$ en conjuguant par $w$, et pour 
$a^+$ et $a^-$ en utilisant les relations $h=u^+u^- -u^-u^+=2a^+=-2a^-$).
 
Autrement dit, il s'agit de montrer que 
$u^+(\Pi_1^{\rm an})^*$ et $(\Pi_2^{\rm an})^*$ sont orthogonaux dans 
$\Delta_2\boxtimes\p1$. Soient $\Pi_1^{0}$ et $\Pi_2^{0}$ les boules unit\'es pour des normes $G$-invariantes sur 
$\Pi_1$, respectivement $\Pi_2$. Alors $(\Pi_1^{0})^*\otimes_{\O_L} L=\Pi_1^*$ et $(\Pi_2^{0})^*\otimes_{\O_L} L=\Pi_2^*$ sont denses dans $(\Pi_1^{\rm an})^*$ et
$(\Pi_2^{\rm an})^*$, respectivement; il suffit donc de montrer que $u^+ (\Pi_1^{(0)})^*$ et $(\Pi_2^{(0)})^*$ sont orthogonaux dans $\Delta_2\boxtimes\p1$.
Mais pour tous 
$x\in (\Pi_1^{(0)})^*$, $y\in (\Pi_2^{(0)})^*$ et $n\in \mathbf{Z}$ on a 
$$\{u^+x, y\}_{\p1}=\{\left(\begin{smallmatrix} p^n & 0 \\0 & 1\end{smallmatrix}\right)u^+x, \left(\begin{smallmatrix}  p^n & 0 \\0 & 1\end{smallmatrix}\right)y\}_{\p1}=$$
$$p^n\{u^+ \left(\begin{smallmatrix} p^n & 0 \\0 & 1\end{smallmatrix}\right)x, \left(\begin{smallmatrix} p^n & 0 \\0 & 1\end{smallmatrix}\right)y\}_{\p1}\in p^n \{u^+ (\Pi_1^{(0)})^*, (\Pi_2^{(0)})^*\}_{\p1}.$$
La compacit\'e de $ (\Pi_1^{(0)})^*$ et $(\Pi_2^{(0)})^*$ et la continuit\'e de $\{\,\,\}_{\p1}$ permettent de conclure que $\{u^+x,y\}_{\p1}=0$ et donc $\mathfrak{g}(\Pi_1^{\rm an})^*\subset (\Pi_2^{\rm an})^*$, comme voulu (cette derni\`ere partie de l'argument est fortement inspir\'ee de \cite{Colmezpoids}). 

   Il nous reste \`a montrer que tout \'el\'ement $l\in (\Pi_1^{\rm an}/\Pi_1^{\rm lisse})^*$, vu comme \'el\'ement de 
   $ (\Pi_2^{\rm an})^*$, s'annule sur $\Pi_2^{\rm lisse}$. D'apr\`es la proposition \ref{inclusion rayon} il existe 
   $a\geq m(V_1)$ tel que $ (\Pi_1^{\rm an}/\Pi_1^{\rm lisse})^*\subset tN^{]0,r_a]}\boxtimes \p1$, ce qui fait que 
   $i_{j,n}(l)\in tN_{\rm dif,n}^+(V_2)$ pour tous $j\in\mathbf{Z}$ et $n\geq a$. On conclut en utilisant les th\'eor\`eme 
   \ref{dualKir} et $\ref{Ann}$, ainsi que les propositions \ref{inclusion Kir} et \ref{accouplement}. 
         \end{proof}
               
            \begin{definition}
             On note $\Pi(\pi,0)^*$ le sous-$G$-module de $tN_{\rm rig}(\pi)\boxtimes \p1$ image de 
             $(\Pi(V)^{\rm an}/\Pi(V)^{\rm lisse})^*$ pour n'importe quel $V\in \mathcal{V}(\pi)$ et n'importe quel isomorphisme
             $D_{\rm pst}(V)\simeq M(\pi)$. On note $\Pi(\pi,0)$ le dual topologique de $\Pi(\pi,0)^*$ muni de l'action duale de
              $G$. 
            \end{definition}
                     
             Ainsi, $\Pi(\pi, 0)$ est une repr\'esentation localement analytique de $G$ sur un espace de type compact et pour tout 
             $V\in \mathcal{V}(\pi)$ on dispose d'un isomorphisme $\Pi(V)^{\rm an}/\Pi(V)^{\rm lisse}\simeq \Pi(\pi,0)$, canonique \`a scalaire pr\`es. On peut reformuler comme suit le th\'eor\`eme \ref{existemorphisme}.
             
             \begin{theoreme}\label{existemorphisme2}
             Il existe un morphisme $G$-\'equivariant continu non nul
\[ \Phi : \Pi(\pi,0)^* \to \O(\Sigma_n)^{\rho}. \]
\end{theoreme}

\section{Structure de $\O(\Omega)$-module sur $\Pi(\pi,0)^*$}\label{cotecolmez}
    
   Ce chapitre assez technique est un des points centraux de ce texte. On y explique dans un premier temps la construction d'un op\'erateur $\partial$, d\'eduit de la connexion de l'\'equation diff\'erentielle $p$-adique de Berger, sur $\Pi(\pi,0)^*$. Cet op\'erateur joue un r\^ole tout \`a fait semblable \`a l'op\'erateur de \og multiplication par $z$\fg{} sur 
    $\O(\Sigma_n)$ ($z$ est la coordonn\'ee sur $\Omega$) et cette heuristique guide sa construction. L'observation de base est que cet op\'erateur sur $\O(\Sigma_n)$
    peut se d\'ecrire en comparant les actions infinit\'esimales 
    $a^+$ et $u^+$ des groupes $\left(\begin{smallmatrix} \zp^* & 0 \\0 & 1\end{smallmatrix}\right)$
    et $\left(\begin{smallmatrix} 1 & \zp \\0 & 1\end{smallmatrix}\right)$. En effet, on v\'erifie sans mal que
    $$a^+-1=u^+\circ \partial.$$
    Nous verrons que l'on peut d\'efinir un unique automorphisme $\partial$ du 
    $L$-espace vectoriel topologique $\Pi(\pi,0)^*$ qui satisfait la relation pr\'ec\'edente. Ce r\'esultat est d\^u \`a Colmez \cite{Colmezpoids}, mais nous avons d\'ecid\'e de le reprendre de mani\`ere assez d\'etaill\'ee, avec un argument diff\'erent et plus direct.

    Puis nous montrons que l'analogie pr\'ec\'edente n'est pas anodine et am\'eliorons le r\'esultat en munissant $\Pi(\pi,0)^*$ d'une structure de $\O(\Omega)$-module telle que la variable $z$ agisse via $\partial$. Cet \'enonc\'e, dont la preuve exploite la dualit\'e de Morita \cite{Morita 1, STMorita}, jouera un r\^ole capital dans la suite. 
\\

  {\it{Nous fixons $\Pi\in \mathcal{V}(\pi)$ et notons $V=V(\Pi)$, $D_{\rm rig}=D_{\rm rig}(V)$. 
  On fixe un isomorphisme $D_{\rm pst}(V)\simeq M(\pi)$, ce qui induit un isomorphisme 
   $N_{\rm rig}=N_{\rm rig}(V)\simeq N_{\rm rig}(\pi)$
et (th\'eor\`eme \ref{almost canonique} et ce qui suit) une identification $\Pi(\pi,0)=\Pi^{\rm an}/\Pi^{\rm lisse}$.}} La construction qui suit ne d\'epend pas des choix faits, mais il convient de les introduire pour certains arguments techniques.  

\subsection{Construction de l'op\'erateur $\partial$ sur $\Pi(\pi,0)^*$}\label{construction partiel}

     Le r\'esultat suivant, qui sera utilis\'e constamment dans la suite, repose sur la proposition \ref{bijective}.
            
      \begin{proposition}\label{imagefermee}
      L'op\'erateur $u^+$ est d'image ferm\'ee sur $(\Pi^{\rm an})^*$ et induit un hom\'eomorphisme entre
      $(\Pi^{\rm an})^*$ et $u^+((\Pi^{\rm an})^*)$.
   \end{proposition}
   
   \begin{proof} Consid\'erons une suite $z_n=(x_n,y_n)\in (\Pi^{\rm an})^*\subset D_{\rm rig}\boxtimes\p1$
  telle que $u^+z_n$ converge dans $(\Pi^{\rm an})^*$ vers $z=(x,y)\in (\Pi^{\rm an})^*$. Puisque 
  $u^+z_n=(tx_n, -t\partial^2(y_n))$ (utiliser le th\'eor\`eme \ref{Casimirinf}) converge dans $D_{\rm rig}\boxtimes \p1$ vers 
  $(x,y)$, il existe $a>0$ tel que 
   $\lim_{n\to\infty} tx_n=x$, $\lim_{n\to \infty} (-t\partial^2(y_n))=y$ dans 
    $D^{]0,r_a]}$. Comme
   $t D^{]0,r_a]}$ est ferm\'e dans $D^{]0,r_a]}$ et 
      la multiplication par $t$ est un hom\'eomorphisme 
       $D^{]0,r_a]}$ dans $tD^{]0,r_a]}$ (par le th\'eor\`eme de l'image ouverte), on conclut que $x=tx'$ avec $x'\in D^{]0,r_a]}$ et $x_n$ tend vers $x'$ dans $D^{]0,r_a]}$.
      
 Pour les $y_n$ l'argument est plus d\'elicat. D'abord,      
     le m\^eme raisonnement
      avec $D^{]0,r_a]}$ remplac\'e par $N^{]0,r_a]}$ (noter que $\partial^2(y_n)\in N^{]0,r_a]}$ pour tout $n$ assez grand, en grandissant \'eventuellement $a$) montre que $y=-tu$ avec 
      $u\in N^{]0,r_a]}$, et $\partial^2(y_n)$ tend vers $u$ dans $N^{]0,r_a]}$. En particulier $\partial^2(y_n)$ converge vers 
      $u$ dans $N_{\rm rig}$. Mais $\partial: N_{\rm rig}\to N_{\rm rig}$
  \'etant une bijection lin\'eaire continue entre des espaces LF (proposition \ref{bijective}), c'est un hom\'eomorphisme (th\'eor\`eme de l'image ouverte), 
  donc $y_n$ converge dans $N_{\rm rig}$ vers $y':=\partial^{-2}(u)$.
  
     On a donc montr\'e l'existence de $x'\in D_{\rm rig}$ et $y'\in N_{\rm rig}$ tels que 
     $x=tx'$, $y=-t\partial^2(y')$ et $x_n$ tend vers $x'$, alors que $y_n$ tend vers $y'$ dans $N_{\rm rig}$.
     Il existe donc $b>a$ tel que $y_n$ tend vers $y'$ dans $N^{]0,r_b]}$. Alors $y_n\in D^{]0,r_b]}$ converge dans $N^{]0,r_b]}$ vers $y'$ et comme 
     $D^{]0,r_b]}$ est ferm\'e dans $N^{]0,r_b]}$, on obtient $y'\in D^{]0,r_b]}$ et donc 
      $z':=(x',y')\in D_{\rm rig}\times D_{\rm rig}$. 
     
       Pour montrer que l'image de $u^+$ est ferm\'ee, il reste \`a v\'erifier que 
     $z':=(x',y')\in (\Pi^{\rm an})^*$, car si ce r\'esultat est \'etabli, l'\'egalit\'e $z=u^+z'$ est claire par construction. 
     Le fait que ${\rm Res}_{\zpet}(y')=w({\rm Res}_{\zpet}(x'))$
   suit en passant \`a la limite dans 
   ${\rm Res}_{\zpet}(y_n)=w({\rm Res}_{\zpet}(x_n))$.
   Donc $z'\in D_{\rm rig}\boxtimes \p1$. Il reste \`a v\'erifier que 
   $z'$ est orthogonal \`a $(\Pi^{\rm an})^*$, ce qui suit du fait que les 
   $(x_n,y_n)$ le sont.  
           Enfin, la derni\`ere assertion est une cons\'equence de ce que l'on a d\'ej\`a d\'emontr\'e et du th\'eor\`eme de l'image ouverte pour les Fr\'echets, ce qui permet de conclure.
   \end{proof}
   
   \begin{remarque} On peut se demander si $aM$ est ferm\'e dans $M$ pour tout $a\in A$, toute 
   alg\`ebre de Fr\'echet-Stein $A$ et tout $A$-module coadmissible $M$ (au sens de Schneider et Teitelbaum \cite{STInv}). Cela d\'ecoule directement des r\'esultats de \cite{STInv} quand $A$ est commutative, mais tombe malheureusement en d\'efaut si $A$ ne l'est plus (m\^eme si 
   $M$ est un $A$-module simple). Le lecteur pourra s'amuser \`a construire des contre-exemples en utilisant 
  (par exemple) des induites comme dans le chapitre $5$ de \cite{ST} (prendre, avec les notations de loc. cit., un caract\`ere $\chi$ tel que $c(\chi)$ soit un nombre de Liouville $p$-adique et montrer en utilisant les formules explicites du lemme 5.2 de loc. cit, que $u^-$ n'est pas d'image ferm\'ee sur $M_{\chi}^-$). 
   \end{remarque}

   \begin{corollaire}\label{closed image}
      L'op\'erateur $u^+$ est d'image ferm\'ee sur $\Pi(\pi,0)^*$ et 
      $$\Pi(\pi,0)^*/u^+(\Pi(\pi,0)^*) \simeq (\Pi(\pi,0)^{u^+=0})^*. $$
   \end{corollaire}
   
   \begin{proof} Le premier point d\'ecoule directement de la proposition pr\'ec\'edente. 
  Ensuite,  
   $\Pi(\pi,0)^*/u^+\Pi(\pi,0)^*$ est un Fr\'echet nucl\'eaire d'apr\`es ce que l'on vient de d\'emontrer. Il est donc r\'eflexif 
 et comme son dual est trivialement $\Pi(\pi,0)^{u^+=0}$, cela permet de conclure.
   \end{proof}
  
\begin{theoreme}
   Il existe une unique application lin\'eaire continue $\partial:
   \Pi(\pi,0)^*\to \Pi(\pi,0)^*$ telle que l'on ait une \'egalit\'e d'op\'erateurs sur $\Pi(\pi,0)^*$
   $$a^+-1=u^+\circ \partial.$$
\end{theoreme}

\begin{proof} L'unicit\'e d\'ecoule simplement du fait que $u^+$ est injectif sur 
$\Pi(\pi,0)^*$, le point d\'elicat est l'existence. Soit $l\in \Pi(\pi,0)^*$, on veut d\'emontrer que 
$l_1:=a^+l-l$ est dans $u^+ \Pi(\pi,0)^*$. D'apr\`es le corollaire pr\'ec\'edent, il suffit de voir que 
$l_1$ s'annule sur $\Pi(\pi,0)^{u^+=0}=(\Pi^{\rm an}/\Pi^{\rm lisse})^{u^+=0}$. En consid\'erant $l_1$ comme une forme lin\'eaire sur 
$\Pi^{\rm an}$ s'annulant sur $\Pi^{\rm lisse}$, il s'agit de montrer que $l_1(a^+v+v)=0$
pour tout $v\in \Pi^{\rm an}$ tel que $u^+v\in \Pi^{\rm lisse}$. Il suffit donc de d\'emontrer le 

\begin{lemme}
  Soit $v\in \Pi^{\rm an}$ tel que $u^+v\in \Pi^{\rm lisse}$. Alors 
  $v_1:=a^+v+v\in\Pi^{\rm lisse}$.
\end{lemme}

\begin{proof} La relation $u^+(a^++1)=a^+u^+$ et le fait que $a^+u^+v=0$ montrent que 
$u^+v_1=0$. Comme le Casimir agit par $0$ sur $\Pi^{\rm an}$ (th\'eor\`eme \ref{Casimirinf}), on obtient aussi 
$$u^-u^+v+a^+v+(a^+)^2v=0,$$
et donc $a^+v_1=0$ (car $u^-u^+v=0$ par hypoth\`ese). Le th\'eor\`eme \ref{Ann} permet de conclure. 
\end{proof}

  Ce qui pr\'ec\`ede montre l'existence d'une unique application $\partial: \Pi(\pi,0)^*\to \Pi(\pi,0)^*$ telle que 
  $a^+-1=u^+\partial$. Par unicit\'e, $\partial$ est lin\'eaire. La continuit\'e suit de la continuit\'e de l'application 
  $a^+-1: \Pi(\pi,0)^*\to \Pi(\pi,0)^*$ et du fait que $u^+$ r\'ealise un hom\'eomorphisme de 
  $\Pi(\pi,0)^*$ sur le ferm\'e $u^+(\Pi(\pi,0)^*)$ de $\Pi(\pi,0)^*$.
\end{proof}

   Le r\'esultat suivant montre que la connaissance de $u^+$ et $\partial$ sur 
   $\Pi(\pi,0)^*$ \'equivaut \`a la connaissance de toute l'action de $U(\mathfrak{sl}_2)$:
   
   \begin{proposition}\label{numer}
    a) En tant qu'op\'erateurs sur $\Pi(\pi,0)^*$ $$a^+=\partial\circ u^+,\quad \partial u^+-u^+\partial=1, \quad u^-=-\partial a^+=-\partial^2 u^+.$$
    
    b) Pour tout $\Pi\in \mathcal{V}(\pi)$, le choix d'un isomorphisme $\Pi(\pi,0)\simeq \Pi^{\rm an}/\Pi^{\rm lisse}$ induit une inclusion $\mathfrak{g}(\Pi^{\rm an})^*\subset \Pi(\pi, 0)^*$ et 
    on a une \'egalit\'e d'op\'erateurs sur $(\Pi^{\rm an})^*$ 
    $$a^+=\partial\circ u^+, \quad u^-=-\partial a^+=-\partial^2 u^+.$$
   \end{proposition}
   
   \begin{proof}
 a) On a $$u^+a^+=a^+u^+-u^+=(a^+-1)u^+=u^+\partial u^+,$$
 ce qui permet de conclure pour la premi\`ere \'egalit\'e, car $u^+$ est injective sur $\Pi(\pi,0)^*$.
La seconde relation s'en d\'eduit. Ensuite,  comme le Casimir agit trivialement sur $\Pi(\pi,0)^*$ et que 
 $u^+u^--u^-u^+=h=2a^+$, on a 
 $$u^+u^-=a^+-(a^+)^2=-u^+\partial a^+,$$
 donc $u^-=-\partial a^+=-\partial^2 u^+$, comme voulu.
 
 b) L'inclusion est claire. Soit $l\in (\Pi^{\rm an})^*$, alors $l_1=u^+l\in \Pi(\pi,0)^*$ et donc 
$a^+l_1-l_1=u^+\partial l_1$, autrement dit 
$a^+u^+l-u^+l=u^+\partial u^+l$. En utilisant la relation 
$u^+a^+=a^+u^+-u^+$ sur $(\Pi^{\rm an})^*$ on obtient bien 
$u^+(a^+l-\partial u^+l)=0$, ce qui permet de conclure pour la premi\`ere relation, car 
$u^+$ est injectif sur $(\Pi^{\rm an})^*$. Puisque $\Pi^{\rm an}$ a un caract\`ere infinit\'esimal nul (th\'eor\`eme \ref{Casimirinf})
la relation $u^-=-\partial a^+=-\partial^2 u^+$ s'en d\'eduit comme dans la preuve de la partie a). 
   \end{proof}

 Enfin, il sera utile de comprendre le lien entre l'action de $G$ et les op\'erateurs $\partial$ et $u^+$, lien fourni par le :     
          
    \begin{theoreme}\label{adjoint}
       Pour tout $g=\left(\begin{smallmatrix} a & b \\ c & d\end{smallmatrix}\right)\in G$ l'op\'erateur $a-c\partial$ est inversible sur $\Pi(\pi,0)^*$ et 
    $$g\circ \partial \circ g^{-1}=(d\partial-b)(a-c\partial)^{-1}, \quad gu^+g^{-1}=\frac{1}{\det g} (a-c\partial)^2u^+.$$
       \end{theoreme}
       
       \begin{proof}
       
       
       
       
              
              
              En conjuguant la relation $a^+-1=u^+\circ \partial$ avec 
              $g$, on obtient 
              $$ga^+g^{-1}-1=gu^+g^{-1}\circ g\partial g^{-1}.$$
              Un calcul imm\'ediat montre les identit\'es suivantes dans $D(G)$, donc aussi dans $\Pi(\pi,0)^*$
              $$ga^{+}g^{-1}=\frac{1}{\det g}(ad a^+-abu^++cdu^--bca^-) \quad \text{et}\quad gu^+g^{-1}=\frac{1}{\det g} (-ach+a^2u^+-c^2u^-).$$
           En utilisant la proposition \ref{numer} (ainsi qu'un calcul direct laiss\'e au lecteur), ces identit\'es se r\'e\'ecrivent 
               \begin{eqnarray}gu^+g^{-1}=\frac{1}{\det g} (a-c\partial)^2u^+ \quad \text{et}\quad ga^+g^{-1}=\frac{1}{\det g} (a-c\partial)(d\partial-b)u^+ \label{formules} \end{eqnarray}
          Nous avons besoin du 
          
          \begin{lemme}\label{p\'enible}
           On a une \'egalit\'e d'op\'erateurs sur $\Pi(\pi,0)^*$
           $$g\partial g^{-1}(a-c\partial)=d\partial-b.$$
          \end{lemme}
          
          \begin{proof} 
          
             Pour simplifier les formules, posons
                $x=a-c\partial$ et $y=d\partial-b$.
              La relation $\partial u^+-u^+\partial=1$ 
             fournit alors 
             $$yu^+x-xu^+y=\det g,$$
             qui, combin\'ee avec 
         la relation $\eqref{formules}$, donne
             $$\left(ga^+g^{-1}-1\right)x=\left(\frac{1}{\det g} xyu^+-1\right)x=\frac{1}{\det g} xyu^+x-x=$$
             $$\frac{1}{\det g} x(xu^+y+\det g)-x=\frac{1}{\det g} x^2u^+y=gu^+ g^{-1}y.$$
            En combinant ceci avec la relation
             $\left(ga^+g^{-1}-1\right)x=gu^{+}g^{-1} \circ g\partial g^{-1} x$ et avec l'injectivit\'e de 
             $gu^{+}g^{-1}$ sur $\Pi(\pi,0)^*$, on obtient enfin 
             $g\partial g^{-1} x=y$, ce qui permet de conclure.
          \end{proof}
           
           Le th\'eor\`eme \ref{adjoint} est ainsi r\'eduit \`a la preuve de l'inversibilit\'e de $a-c\partial$. 
                Or, on d\'eduit du lemme \ref{p\'enible} la relation
                $$(c\cdot g\partial g^{-1}+d)(a-c\partial)=\det g,$$
                ce qui permet de conclure. 
                     \end{proof}

\subsection{Construction de la structure de $\O(\Omega)$-module}\label{structure O(Omega)}
        
   Afin de prouver dans la section suivante que le morphisme $\Phi$ du th\'eor\`eme \ref{existemorphisme2} est surjectif, il sera vital de savoir que l'on peut munir $\Pi(\pi,0)^*$ d'une structure de $\O(\Omega)$-module, et c'est \`a cette t\^ache qu'est consacr\'ee ce paragraphe. Vu la construction de l'op\'erateur $\partial$ dans le paragraphe \ref{construction partiel}, le lecteur ne sera pas surpris par l'\'enonc\'e du
        
   \begin{theoreme} \label{module} Il existe une unique structure de $\O(\Omega)$-module sur $\Pi(\pi,0)^{*}$ 
   qui soit compatible avec sa structure de $L$-espace vectoriel et telle que 
   $z.l=\partial(l)$ pour tout $l\in \Pi(\pi,0)^*$. 
       \end{theoreme}

     La preuve du th\'eor\`eme \ref{module} occupe le reste de ce chapitre. L'argument suit un chemin un peu d\'etourn\'e, car l'op\'erateur 
     $\partial$ sur $\Pi(\pi,0)$ (ou son dual) ne pr\'eserve pas de sous-espaces de Banach \og \'evidents\fg{} de 
     $\Pi(\pi,0)$ (en particulier, il ne pr\'eserve pas les vecteurs localement analytiques de rayon fix\'e). La raison en est que m\^eme si 
     $\partial$ est 
      construit \`a partir de la connexion 
     $\partial$ sur $N_{\rm rig}$, sa d\'efinition fait aussi appara\^itre $\partial^{-1}$, qui est difficilement contr\^olable. 
      Pour contourner ces difficult\'es, nous utilisons la dualit\'e de Morita, 
     des arguments d'analyse fonctionnelle et le r\'esultat technique suivant. 
     On note $\langle\,\,,\,\rangle$ l'accouplement canonique entre $\Pi(\pi,0)^*$ et $\Pi(\pi,0)$.

      \begin{proposition}\label{loc an}
       Soient $v\in \Pi(\pi,0)$ et $l\in \Pi(\pi,0)^*$. L'application 
       $$\phi_{l,v}: \qp\to L, \quad \phi_{l,v}(x)=\langle (\partial-x)^{-1}l, v\rangle$$
       s'\'etend en une fonction localement analytique sur $\p1(\qp)$, nulle \`a l'infini.
     \end{proposition}
     
     \begin{proof} Posons $\phi_{l,v}(\infty)=0$. Pour montrer que $\phi_{l,v}$ est localement analytique, nous aurons besoin du r\'esultat suivant: 
          
        \begin{lemme}
       Pour tous $g\in G$, $l\in \Pi(\pi,0)^*$, $v\in \Pi(\pi,0)$ et $x\in\p1(\qp)$ on a 
       $$\phi_{l,v}(gx)=\frac{cx+d}{\det g} \phi_{(c\partial+d)g^{-1}l, g^{-1}v}(x).$$
          \end{lemme}   
     
     \begin{proof}
        C'est un calcul un peu fastidieux dont l'ingr\'edient cl\'e est le th\'eor\`eme \ref{adjoint}. Explicitement, en utilisant deux fois ce th\'eor\`eme
        on obtient, en posant $g=\left(\begin{smallmatrix} a & b \\c & d\end{smallmatrix}\right)$:
                $$\phi_{l,v}(gx)=\langle (\partial-gx)^{-1}l, v\rangle=(cx+d)\langle (d\partial-b+x(c\partial-a))^{-1}l, v\rangle$$
                $$=-(cx+d)\langle (a-c\partial)^{-1} (x-(d\partial-b)(a-c\partial)^{-1})^{-1}l, v\rangle=-(cx+d)\langle (a-c\partial)^{-1} (x-g\partial g^{-1})^{-1}l, v\rangle$$
                $$=(cx+d)\langle (a-c\partial)^{-1} g(\partial-x)^{-1} g^{-1}l, v\rangle=(cx+d)\langle (a-c g^{-1}\partial g)^{-1} (\partial-x)^{-1} g^{-1}l, g^{-1}v\rangle$$
                $$=\frac{cx+d}{\det g} \langle (c\partial+d)(\partial-x)^{-1} g^{-1}l, g^{-1}v\rangle=\frac{cx+d}{\det g} \phi_{(c\partial+d) g^{-1}l, g^{-1}v}.$$
   \end{proof}
     
        Il suffit donc de voir que $\phi_{l,v}$ est analytique au voisinage de $0$. Notons que gr\^ace au th\'eor\`eme 
        \ref{adjoint} on a\footnote{On d\'efinit les op\'erateurs $\partial$, $\partial^{-1}$ sur $\Pi(\pi,0)$ par dualit\'e.} 
        $$\phi_{l,v}(x)=\langle \left(\begin{smallmatrix} 1 & x \\0 & 1\end{smallmatrix}\right)\partial^{-1} \left(\begin{smallmatrix} 1 & -x \\0 & 1\end{smallmatrix}\right)l, v\rangle=
        \langle l, \left(\begin{smallmatrix} 1 & x \\0 & 1\end{smallmatrix}\right)\partial^{-1} \left(\begin{smallmatrix} 1 & -x \\0 & 1\end{smallmatrix}\right)v\rangle.$$

        En \'ecrivant 
        $\Pi(\pi,0)=\Pi^{\rm an}/\Pi^{\rm lisse}$, pour un choix de $\Pi \in \mathcal{V}(\pi)$, et en filtrant $\Pi$ \og par rayon d'analyticit\'e\fg{} (voir \cite[chap. IV]{CD}
        pour les d\'etails), on peut \'ecrire $\Pi(\pi,0)$ comme une r\'eunion croissante d'espaces de Banach 
        $\Pi(\pi,0)^{(h)}$ ($h\in \mathbf{N}^*$), stables par $\left(\begin{smallmatrix} 1 & \zp \\0 & 1\end{smallmatrix}\right)$, mais 
        \textit{pas} par $\partial$ ou $\partial^{-1}$. Soit $h$ tel que 
        $v\in \Pi(\pi,0)^{(h)}$. La fonction 
        $x\to \left(\begin{smallmatrix} 1 & x \\0 & 1\end{smallmatrix}\right)v$ \'etant analytique au voisinage de $0$, \`a valeurs dans le Banach 
        $\Pi(\pi,0)^{(h)}$, on peut \'ecrire pour $x\in p^N\zp$ ($N$ assez grand, ne d\'ependant que de $v$)
        $$ \left(\begin{smallmatrix} 1 & -x \\0 & 1\end{smallmatrix}\right)v=\sum_{n\geq 0} x^n v_n$$
        avec $v_n\in \Pi(\pi,0)^{(h)}$. On a $p^{nN}v_n\to 0$ dans 
        $\Pi(\pi,0)^{(h)}$, donc aussi dans $\Pi(\pi,0)$. Puisque $\partial$ est un hom\'eomorphisme de 
        $\Pi(\pi,0)$ (car elle est lin\'eaire bijective et que $\Pi(\pi,0)$ est un espace de type compact), 
        on a $p^{nN}\partial^{-1}(v_n)\to 0$ dans $\Pi(\pi,0)$. Il existe donc 
        $h'\geq h$ tel que $p^{nN} \partial^{-1}(v_n)\to 0$ dans $\Pi(\pi,0)^{(h')}$. Posons 
        $$v'_n=p^{nN} \partial^{-1}(v_n).$$
         Ainsi, pour tout 
        $x\in p^N \zp$ on a (par continuit\'e de $\partial^{-1}$ et de la restriction de $l$ \`a $\Pi(\pi,0)^{(h')}$)
        $$ \phi_{l,v}(x)=
        \langle l, \left(\begin{smallmatrix} 1 & x \\0 & 1\end{smallmatrix}\right)\partial^{-1} \left(\begin{smallmatrix} 1 & -x \\0 & 1\end{smallmatrix}\right)v\rangle=\sum_{n\geq 0} 
        \left(\frac{x}{p^N}\right)^n l\left( \left(\begin{smallmatrix} 1 & x \\0 & 1\end{smallmatrix}\right)v'_n\right).$$
        
         Puisque $v'_n$ tendent vers $0$ dans $\Pi(\pi,0)^{(h')}$, il existe $M$, d\'ependant de $h'$, tel que les fonctions 
         $$f_n: \zp\to L, \quad f_n(x)=l\left( \left(\begin{smallmatrix} 1 & x \\0 & 1\end{smallmatrix}\right)v'_n\right)$$
         tendent vers $0$ dans l'espace des fonctions analytiques sur $a+p^M\zp$ pour tout $a\in\zp$ (cela d\'ecoule du th\'eor\`eme  
         IV.6 et de la remarque IV.15 de \cite{CD}). On en d\'eduit que la fonction 
         $$x\mapsto \phi_{l,v}(x)=\sum_{n\geq 0}   \left(\frac{x}{p^N}\right)^n f_n(x)$$ est analytique au voisinage de $0$, ce qui permet de conclure. 
             \end{proof}
     
          Passons maintenant \`a la preuve du th\'eor\`eme \ref{module}. 
                        Soit ${\rm St}^{\rm an}$ la {\it Steinberg analytique}, quotient de l'espace ${\rm LA}(\p1(\qp))$ des fonctions localement analytiques sur $\p1(\qp)$, \`a valeurs dans $L$, par les fonctions constantes. Nous ferons un usage constant du r\'esultat classique et fondamental suivant, connu sous le nom de dualit\'e de Morita \cite{Morita 1}.

             \begin{proposition} \label{Morita}
             
             a) Soit $\lambda\in \O(\Omega)^*$. La fonction $f_{\lambda}$ d\'efinie (pour $x\in\qp$) par
             $$f_{\lambda}(x)=\lambda\left(\frac{1}{z-x}\right)$$
s'\'etend en une fonction localement analytique sur $\p1(\qp)$, nulle \`a l'infini. De plus, 
l'application $\lambda\mapsto f_{\lambda}$ induit un isomorphisme de $L$-espaces vectoriels topologiques\footnote{C'est m\^eme un isomorphisme de $G$-repr\'esentations, si l'on remplace $\O(\Omega)$ par $\Omega^1(\Omega)$.}
$$\O(\Omega)^*\simeq {\rm St}^{\rm an}.$$

b) La transpos\'ee de l'isomorphisme pr\'ec\'edent est (via l'identification $\O(\Omega)^{**}=\O(\Omega)$) l'application $\mu\mapsto f_{\mu}\in \O(\Omega)$, o\`u
$$f_{\mu}(z)=\int_{\p1(\qp)} \frac{1}{z-x} \mu(x).$$
          \end{proposition}

 Soit $$\mathcal{S}={\rm Vect}_{x\in \qp} \frac{1}{z-x}\subset \O(\Omega)$$ le sous $L$-espace vectoriel de $\O(\Omega)$ engendr\'e par les fonctions
$\frac{1}{z-x}$ pour $x\in \qp$. Si $f\in \mathcal{S}$, on note $\mu_f$ l'\'el\'ement de $({\rm St^{\rm an}})^*$ qui correspond \`a $f$ via l'isomorphisme 
 $\O(\Omega)\simeq ({\rm St^{\rm an}})^*$. Explicitement, 
    $$ \mu_f=\sum_{x\in \qp} a_x \delta_x-\left(\sum_{x\in \qp} a_x\right)\delta_{\infty}\quad \text{si}\quad  f=\sum_{x\in \qp} \frac{a_x}{z-x}\in \mathcal{S}.$$
Notons que $\mathcal{S}$ est dense dans $\O(\Omega)$ (cela d\'ecoule directement de la 
     proposition \ref{Morita} et du th\'eor\`eme de Hahn-Banach).

        Le th\'eor\`eme \ref{adjoint} donne un sens \`a la d\'efinition suivante: 
        
       \begin{definition} Si $f=\sum_{x\in\qp} \frac{a_x}{z-x}\in \mathcal{S}$, on d\'efinit 
       un op\'erateur lin\'eaire continu 
       $$T_f: \Pi(\pi,0)^*\to \Pi(\pi,0)^*, \quad T_f(l)=\sum_{x\in \qp} a_x(\partial-x)^{-1}(l)\in \Pi(\pi,0)^*.$$
     \end{definition}

     Notons que par construction on a pour tous $l\in \Pi(\pi,0)^*, v\in \Pi(\pi,0), f\in \mathcal{S}$     
     \begin{eqnarray} \langle T_f(l), v\rangle=\sum_{x\in\qp} a_x \phi_{l,v}(x)=\int_{\p1(\qp)} \phi_{l,v}\mu_f,   \label{integrale} \end{eqnarray}
    puisque $\phi_{l,v}$ s'annule \`a l'infini.

\begin{proposition}
  Soit $f\in \O(\Omega)$ et soit $(f_n)_{n}$ une suite d'\'el\'ements de $\mathcal{S}$ qui converge vers $f$ dans $\O(\Omega)$. Alors la suite
 d'op\'erateurs $T_{f_n}$ converge faiblement vers un op\'erateur continu $T_f: \Pi(\pi,0)^*\to \Pi(\pi,0)^*$.
\end{proposition}

\begin{proof} Posons $g_n=f_n-f_{n-1}$, de telle sorte que $g_n\in \mathcal{S}$ tend vers $0$ dans 
$\O(\Omega)$. Si $v\in \Pi(\pi,0)$ et $l\in \Pi(\pi,0)^*$ on a 
$$\langle T_{g_n}(l), v\rangle=\int_{\p1(\qp)} \phi_{l,v} \mu_{g_n}.$$
Comme $g_n$ tend vers $0$ dans $\O(\Omega)$, $\mu_{g_n}$ tend vers $0$ dans $({\rm St^{\rm an}})^*$ (par dualit\'e de Morita)
et donc $\mu_{g_n}$ tend faiblement vers $0$ dans $({\rm St^{\rm an}})^*$, ce qui permet de conclure que 
$\lim_{n\to\infty} \langle T_{g_n}(l), v\rangle=0$. Combin\'e avec le lemme \ref{faible} et la r\'eflexivit\'e de $\Pi(\pi,0)$, 
cela montre que $\lim_{n\to\infty} T_{g_n}(l)=0$ dans $\Pi(\pi,0)^*$. Ainsi, 
$\lim_{n\to\infty} (T_{f_n}(l)-T_{f_{n-1}}(l))=0$ et comme $\Pi(\pi,0)^*$ est complet, la suite d'op\'erateurs $(T_{f_n})_n$ converge faiblement. La continuit\'e de la limite d\'ecoule du th\'eor\`eme de Banach-Steinhaus.  
\end{proof}

\begin{proposition}\label{continuit\'e}
a) Pour tous $f\in \O(\Omega), l\in \Pi(\pi,0)^*, v\in \Pi(\pi,0)$ on a 
  $$\langle T_f(l), v\rangle=\int_{\p1(\qp)} \phi_{l,v}\mu_f.$$
  
  b)  Si $f_n\in \O(\Omega)$ tend vers $f\in \O(\Omega)$, alors pour tout 
  $l\in \Pi(\pi,0)^*$ on a $$\lim_{n\to\infty} T_{f_n}(l)=T_f(l)\quad \text{dans}\quad \Pi(\pi,0)^*.$$
\end{proposition}

\begin{proof} a) Soit $(f_n)_n$ comme dans la proposition pr\'ec\'edente, alors
$$\langle T_{f}(l), v\rangle=\lim_{n\to\infty} \langle T_{f_n}(l), v\rangle=\lim_{n\to\infty} \int_{\p1(\qp)}\phi_{l,v}\mu_{f_n}=\int_{\p1(\qp)} \phi_{l,v}\mu_f,$$
la derni\`ere \'egalit\'e \'etant une cons\'equence du fait que $(\mu_{f_n})$ converge faiblement vers $\mu_f$ (encore une fois, par dualit\'e de Morita).

b) Le lemme \ref{faible} montre qu'il suffit de v\'erifier que 
$\lim_{n\to\infty} \langle T_{f_n}(l), v\rangle=\langle T_f(l),v\rangle$ pour tout 
$v\in \Pi(\pi,0)^*$. Cela d\'ecoule directement du a) et de la dualit\'e de Morita.
\end{proof}

    Le th\'eor\`eme \ref{module} se d\'eduit maintenant facilement de ce qui pr\'ec\`ede. En effet, la structure de $\O(\Omega)$-module sur 
    $\Pi(\pi,0)^*$ s'obtient en posant $f.l=T_f(l)$ pour $f\in \O(\Omega)$ et $l\in \Pi(\pi,0)^*$. Le seul point qu'il nous reste \`a v\'erifier est que 
    $(fg).l=f.(g.l)$ pour tous $f,g\in \O(\Omega)$ et $l\in \Pi(\pi,0)^*$. Par densit\'e de $\mathcal{S}$ dans $\O(\Omega)$ et la continuit\'e des applications $l\mapsto f.l$ et $f\mapsto f.l$ (qui d\'ecoule des r\'esultats pr\'ec\'edents), on se ram\`ene au cas o\`u $f,g\in \mathcal{S}$, et ensuite au cas $f=\frac{1}{z-x}$ et $g=\frac{1}{z-y}$, avec
    $x,y\in \qp$. De plus (encore par continuit\'e), on peut supposer que $x\ne y$. Mais alors 
    $$fg=\frac{1}{x-y}\cdot \left(\frac{1}{z-x}-\frac{1}{z-y}\right)\in \mathcal{S},$$
    et le r\'esultat voulu d\'ecoule de l'identit\'e 
    $$(\partial-x)^{-1}\circ (\partial-y)^{-1}=\frac{1}{x-y}((\partial-x)^{-1}-(\partial-y)^{-1})$$
    valable sur $\Pi(\pi,0)^*$.
    
    Le r\'esultat suivant est une cons\'equence imm\'ediate du th\'eor\`eme \ref{module}, mais il nous sera tr\`es utile par la suite.

\begin{proposition}\label{commute}
  Tout morphisme $G$-\'equivariant continu 
$\Phi: \Pi(\pi,0)^*\to \O(\Sigma_n)^{\rho}$ est $\O(\Omega)$-lin\'eaire. 
\end{proposition}

\begin{proof} Les op\'erateurs $\partial$ et $u^+$ satisfont 
$a^+-1=u^+\partial$ sur les deux espaces 
 $\Pi(\pi,0)^*$ et $\O(\Sigma_n)^{\rho}$. Puisque 
 $\Phi$ est $G$-\'equivariant, il commute \`a $a^+$ et $u^+$. On en d\'eduit que 
 $u^+ \Phi(\partial l)=u^+ \partial \Phi(l)$ pour tout 
$l \in \Pi(\pi,0)^*$. Puisque $u^+$ est injectif sur $\O(\Sigma_n)^{\rho}$, on en d\'eduit que 
$\Phi$ commute avec $\partial$. Il commute donc aussi avec $(\partial-x)^{-1}$ pour tout 
$x\in \qp$. On en d\'eduit que $\Phi(f.l)=f\Phi(l)$ pour tout $f\in \mathcal{S}$ et tout 
$l\in \Pi(\pi,0)^*$, et le r\'esultat s'ensuit gr\^ace \`a la densit\'e de $\mathcal{S}$ dans 
$\O(\Omega)$ et au point b) de la proposition \ref{continuit\'e}.
\end{proof}

\section{Surjectivit\'e de $\Phi$}\label{surjectivit\'e}

   Ce chapitr\'e est consacr\'e \`a la preuve du r\'esultat suivant, preuve qui doit beaucoup \`a des discussions avec Laurent Fargues et Pierre Colmez. 
   
   \begin{theoreme}\label{surjectif tour} 
   Tout morphisme $G$-\'equivariant continu et non nul 
$\Phi: \Pi(\pi,0)^*\to \O(\Sigma_n)^{\rho}$ est surjectif. 
   \end{theoreme}
   
   En appliquant \cite[lemma 3.6]{STInv} et le th\'eor\`eme \ref{existemorphisme2}, on en d\'eduit que 
le $D(G)$-module $\O(\Sigma_n)^{\rho}$ est coadmissible. En d\'ecomposant $\O(\Sigma_n)$ selon l'action de 
$D^*$, on obtient la coadmissibilit\'e du $D(G)$-module $\O(\Sigma_n)$. 

    La preuve du th\'eor\`eme \ref{surjectif tour} se fait en deux \'etapes: on \'etablit d'abord que 
    l'image de $\Phi$ est dense en utilisant des r\'esultats de Kohlhaase \cite{Kohl} sur le lien entre certains fibr\'es 
    $G$-\'equivariants sur la tour de Drinfeld et certains fibr\'es $D^*$-\'equivariants sur la tour de Lubin-Tate. Un argument d'analyse fonctionnelle 
    permet alors de conclure que $\Phi$ est surjectif. Les deux \'etapes utilisent de mani\`ere cruciale la structure de 
    $\O(\Omega)$-module sur $\Pi(\pi,0)^*$ et le fait que $\Phi$ est $\O(\Omega)$-lin\'eaire. Nous utiliserons aussi syst\'ematiquement les 
    r\'esultats du chapitre 3 de \cite{STInv} concernant les modules coadmissibles sur une alg\`ebre de Fr\'echet-Stein (qui sera dans notre cas l'alg\`ebre des fonctions rigides analytiques
    sur une vari\'et\'e Stein sur $\qp$).

\subsection{Les tours jumelles}

Si $X$ vit sur $\qp$, $\breve{X}$ d\'esigne l'extension des scalaires $X \hat{\otimes}_{\qp} \breve{\mathbf{Q}}_p$. 
\begin{lemme}\label{produitcomplete}
Soit $F: X \to Y$ un morphisme d'espaces de Fr\'echet, tel que $\breve{F} : \breve{X} \to \breve{Y}$ soit d'image dense. Alors $F$ est 
d'image dense.
\end{lemme}
\begin{proof}
Soit $l$ une forme lin\'eaire continue s'annulant sur l'image de $F$. Alors pour tout $a\in \breve{\mathbf{Q}}_p$
$$\breve{l}(\breve{F}(x\hat{\otimes} a))=\breve{l}(F(x)\hat{\otimes} a)=a\cdot l(F(x))=0.$$
Donc $\breve{l}\circ \breve{F}$ s'annule sur l'image de $X\otimes_{\qp} \breve{\mathbf{Q}}_p$ dans 
$\breve{X}$. Cette image \'etant dense, on a $\breve{l}\circ \breve{F}=0$ et comme 
$\breve{F}$ est d'image dense on a $\breve{l}=0$. Mais alors pour tout 
$y\in Y$ on a $l(y)=\breve{l}(y\hat{\otimes} 1)=0$, d'o\`u le r\'esultat. 
\end{proof}

Avant de prouver que $\Phi$ est surjective, on va montrer que $\Phi$ est d'image dense. D'apr\`es le lemme pr\'ec\'edent, il suffit de le faire pour $\breve{\Phi}$.
\begin{proposition}\label{image dense}
L'image de $\breve{\Phi}$ est dense dans $\O(\breve{\m}_n)^{\rho}$.
\end{proposition}
\begin{proof}

Notons $W = \overline{\mathrm{Im}(\breve{\Phi})}$ l'adh\'erence de l'image de $\breve{\Phi}$ : c'est un sous-$\O(\breve{\Omega})$-module coadmissible (sur $\O(\breve{\Omega})$) de $\O(\breve{\m}_{n})^{\rho}$, puisqu'un sous-module ferm\'e d'un module coadmissible est lui-m\^ eme coadmissible et puisque $\O(\breve{\m}_n)$ est coadmissible, en tant que 
$\O(\breve{\Omega})$-module projectif de type fini \cite{STInv}. On a donc une suite exacte $G$-\'equivariante de $\O(\breve{\Omega})$-modules coadmissibles
\[ 0 \to W \to \O(\breve{\m}_{n})^{\rho} \to \O(\breve{\m}_{n})^{\rho}/W \to 0. \]
Comme $\breve{\Omega}$ est une vari\'et\'e Stein, cette suite exacte revient exactement \`a la donn\'ee d'une suite exacte de faisceaux coh\'erents $G$-\'equivariants sur $\breve{\Omega}$, que nous noterons
\[ 0 \to \mathcal{F} \to \mathcal{G} \to \mathcal{F}' \to 0. \]
Le faisceau $\mathcal{G}$ est un fibr\'e vectoriel, puisque $\breve{\m}_{n}$ est un rev\^ etement \'etale de $\breve{\Omega}$. Les faisceaux coh\'erents $\mathcal{F}$ et $\mathcal{F}'$ aussi : comme $\breve{\Omega}$ est une courbe lisse, il suffit en effet de voir qu'ils sont sans torsion. Or le support de la partie de torsion d'un faisceau coh\'erent $G$-\'equivariant sur $\O(\breve{\Omega})$ est un sous-ensemble discret de $\breve{\Omega}$ stable par $G$ : c'est donc l'ensemble vide. La suite exacte ci-dessus est donc une suite exacte de \textit{fibr\'es $G$-\'equivariants sur $\breve{\Omega}$}. 
\\

Pour montrer que ces deux fibr\'es sont isomorphes, nous allons faire appel aux r\'esultats de \cite{Kohl}. Commen\c cons par fixer quelques notations. On note $\breve{\m}_{0}^{(0)}$ la composante de $\breve{\m}_0$ correspondant au lieu o\`u la quasi-isog\'enie est de hauteur $0$, et $\breve{\m}_{n}^{(0)}$ ses rev\^ etements. On note aussi $\breve{\mathcal{L}}_{0}^{(0)}$ l'espace de Lubin-Tate (\cite{LT}, c'est une boule ouverte de rayon $1$) et $\breve{\mathcal{L}}_{n}^{(0)}$ ses rev\^ etements (galoisiens de groupe $G_0/G_n$).

L'observation de base de Kohlhaase est que l'anneau $\O(\breve{\m}_{n}^{(0)})$, resp. $\O(\breve{\mathcal{L}}_{n}^{(0)})$, est $(\breve{F}_n,G_n)$-r\'egulier\footnote{Rappelons que $F_n=\qp(\mu_{p^n})$.}, resp. $(\breve{F}_n,D_n)$-r\'egulier, au sens de Fontaine. Ceci va permettre de d\'efinir de cat\'egories de fibr\'es \'equivariants int\'eressants, \og \`a la Fontaine\fg{}.

Si $\mathcal{N}$ est un fibr\'e vectoriel $G_0$-\'equivariant de rang $r$ sur $\breve{\m}_{0}^{(0)}$ et si $N=H^0(\breve{\m}_{0}^{(0)}, \mathcal{N})$, on pose 
 \[ \Lambda_n(\mathcal{N}) = (\O(\breve{\m}_{n}^{(0)}) \otimes_{\O(\breve{\m}_{0}^{(0)})} N)^{G_n}.\]
De la r\'egularit\'e de $\O(\breve{\m}_{n}^{(0)})$, on d\'eduit par un argument standard que l'application naturelle
\begin{eqnarray} \O(\breve{\m}_{n}^{(0)}) \otimes_{\breve{F}_n} \Lambda_n(\mathcal{N}) \to \O(\breve{\m}_{n}^{(0)}) \otimes_{\O(\breve{\m}_{0}^{(0)})} N  \label{fibre lt} \end{eqnarray}
est injective, pour tout $n$. En particulier, $\Lambda_n(\mathcal{N})$ est un $\breve{F}_n$-espace vectoriel avec action semi-lin\'eaire de $G_0/G_n$, de dimension inf\'erieure ou \'egale \`a $r$. 

De m\^eme, si $\mathcal{N}$ est un fibr\'e vectoriel $D_0$-\'equivariant sur $\breve{\mathcal{L}}_{0}^{(0)}$, et si $N=H^0(\breve{\mathcal{L}}_{0}^{(0)}, \mathcal{N})$, on pose
\[ \Delta_n(\mathcal{N}) = (\O(\breve{\mathcal{L}}_{n}^{(0)}) \otimes_{\O(\breve{\mathcal{L}}_{0}^{(0)})} N)^{D_n}. \]
L'application naturelle
\begin{eqnarray}\O(\breve{\mathcal{L}}_{n}^{(0)}) \otimes_{\breve{F}_n} \Delta_n(\mathcal{N}) \to \O(\breve{\mathcal{L}}_{n}^{(0)}) \otimes_{\O(\breve{\mathcal{L}}_{0}^{(0)})} N  \label{fibre dr} \end{eqnarray}
est injective pour tout $n$. 

\begin{definition} Un fibr\'e vectoriel $G_0$-\'equivariant $\mathcal{N}$ sur $\breve{\m}_{0}^{(0)}$ est dit \textit{de Lubin-Tate} s'il existe $n\geq 0$ tel que l'application $\eqref{fibre lt}$ soit un isomorphisme. On d\'efinit alors $\mathbf{D}_{\rm LT}(\mathcal{N})$ comme le fibr\'e $D_0$-\'equivariant sur $\breve{\mathcal{L}}_{0}^{(0)}$ dont l'espace des sections globales est le $\O(\breve{\mathcal{L}}_{n}^{(0)})^{G_0}=\O(\breve{\mathcal{L}}_{0}^{(0)})$-module
$$ H^0(\breve{\mathcal{L}}_{0}^{(0)},\mathbf{D}_{\rm LT}(\mathcal{N}))=( \O(\breve{\mathcal{L}}_{n}^{(0)}) \otimes_{\breve{F}_n} \Lambda_n(\mathcal{N}) )^{G_0}. $$
Un fibr\'e vectoriel $D_0$-\'equivariant $\mathcal{N}=\tilde{N}$ sur $\breve{\mathcal{L}}_{0}^{(0)}$ est dit \textit{de Drinfeld} s'il existe $n\geq 0$ tel que l'application $\eqref{fibre dr}$ soit un isomorphisme. On d\'efinit alors $\mathbf{D}_{\rm Dr}(\mathcal{N})$ comme le fibr\'e $G_0$-\'equivariant sur $\breve{\m}_{0}^{(0)}$ dont l'espace des sections globales est le $\O(\breve{\m}_{n}^{(0)})^{D_0}=\O(\breve{\m}_{0}^{(0)})$-module
\[ H^0(\breve{\m}_{0}^{(0)}, \mathbf{D}_{\rm Dr}(\mathcal{N}))=( \O(\breve{\m}_{n}^{(0)}) \otimes_{\breve{F}_n} \Delta_n(\mathcal{N}) )^{D_0}. \]
\end{definition}

Un fibr\'e $G$-\'equivariant $\mathcal{N}$ sur $\breve{\Omega}$ est dit de \textit{Lubin-Tate} si $(\pi_{\mathrm{Dr}}^{(0)})^*(\mathrm{Res}_{G_0}^G(\mathcal{N}))$ est de Lubin-Tate au sens pr\'ec\'edent, o\`u $\pi_{\mathrm{Dr}}^{(0)} : \breve{\m}_{0}^{(0)} \to \breve{\Omega}$ est l'application des p\'eriodes de $\breve{\m}_{0}$ restreinte \`a la composante connexe $\breve{\m}_{0}^{(0)}$.

De m\^eme, un fibr\'e $D^*$-\'equivariant $\mathcal{N}$ sur $\breve{\mathbf{P}}^1$ est dit \textit{de Drinfeld} si $(\pi_{\rm LT}^{(0)})^*(\mathrm{Res}_{D_0}^{D^*}(\mathcal{N}))$ est de Drinfeld au sens pr\'ec\'edent, o\`u $\pi_{\rm LT}^{(0)} : \breve{\mathcal{L}}_{0}^{(0)} \to \breve{\mathbf{P}}^1$ est l'application des p\'eriodes de Gross-Hopkins \cite{GH}.

Si $\mathcal{N}$ est un fibr\'e de Lubin-Tate sur $\breve{\Omega}$ (resp. si $\mathcal{N}$ est un fibr\'e de Drinfeld sur $\breve{\mathbf{P}}^1$), le fibr\'e $\mathbf{D}_{\rm LT}(\mathcal{N})$ est $D^*$-\'equivariant  (resp. le fibr\'e $\mathbf{D}_{\rm Dr}(\mathcal{N})$ est $G$-\'equivariant) et de plus il descend en un fibr\'e $D^*$-\'equivariant encore not\'e $\mathbf{D}_{\rm LT}(\mathcal{N})$ sur $\breve{\mathbf{P}}^1$ (resp. en un fibr\'e $G$-\'equivariant encore not\'e $\mathbf{D}_{\rm Dr}(\mathcal{N})$ sur $\breve{\Omega}$).

\begin{theoreme}[Kohlhaase]\label{kohlhaase}

a) Les foncteurs $\mathbf{D}_{\rm LT}$ et $\mathbf{D}_{\mathrm{Dr}}$ r\'ealisent des \'equivalences de cat\'egories quasi-inverses l'une de l'autre entre la cat\'egorie des fibr\'es $G$-\'equivariants sur $\breve{\Omega}$ dits \textit{de Lubin-Tate} et celle des fibr\'es $D^*$-\'equivariants sur $\breve{\mathbf{P}}^1$ dits \textit{de Drinfeld}.

b) Si $\rho$ est une repr\'esentation lisse de dimension finie de $D^*$, le fibr\'e $\rho^* \otimes \O_{\breve{\mathbf{P}}^1}$ est un fibr\'e de Drinfeld et 
on a un isomorphisme canonique 
$$\mathbf{D}_{\rm Dr}(\rho^* \otimes \O_{\breve{\mathbf{P}}^1})=\O(\breve{\m}_{n})^{\rho}.$$

c) Les cat\'egories des fibr\'es de Drinfeld et des fibr\'es de Lubin-Tate sont stables par sous-objet et quotient, et les foncteurs $\mathbf{D}_{\rm LT}$ et $\mathbf{D}_{\mathrm{Dr}}$ pr\'eservent les suites exactes. 
\end{theoreme}

\begin{remarque}
Les m\'ethodes de Kohlhaase \cite{Kohl} sont \'el\'ementaires et \textit{n'utilisent pas} l'isomorphisme entre les tours de Lubin-Tate et de Drinfeld, d\^u \`a Faltings et Fargues \cite{faltings, FGL}.
\end{remarque}

Revenons \`a la preuve de la proposition \ref{image dense}. Le fibr\'e $\mathcal{G}$ est un fibr\'e de Lubin-Tate, d'apr\`es le deuxi\`eme point du th\'eor\`eme, puisque $\mathcal{G}=\mathbf{D}_{\mathrm{Dr}}(\rho^* \otimes \O_{\breve{\mathbf{P}}^1})$. Le troisi\`eme point du th\'eor\`eme prouve que $\mathcal{F}$ et $\mathcal{F}'$ sont aussi de Lubin-Tate, et que l'on a une suite exacte de fibr\'es $D^*$-\'equivariants sur $\breve{\mathbf{P}}^1$ : 
\[ 0 \to \mathbf{D}_{\rm LT}(\mathcal{F}) \to \rho^* \otimes \O_{\breve{\mathbf{P}}^1} \to \mathbf{D}_{\rm LT}(\mathcal{F'}) \to 0. \]
Choisissons $\lambda$ un entier positif suffisamment grand pour que les fibr\'es $\mathbf{D}_{\rm LT}(\mathcal{F}) \otimes \O(\lambda)$ et $\mathbf{D}_{\rm LT}(\mathcal{F}') \otimes \O(\lambda)$ soient \`a pentes positives. En tordant par $\O(\lambda)$, on a donc une injection $D^*$-\'equivariante
\[ 0 \to H^0(\mathbf{D}_{\rm LT}(\mathcal{F}) \otimes \O(\lambda)) \to \rho^* \otimes \mathrm{Sym}^{\lambda}(L^2) \]
Or, comme $\rho^*$ est lisse, le membre de droite est une repr\'esentation irr\'eductible et cette fl\`eche est donc un isomorphisme. En outre, la suite exacte longue de cohomologie donne l'annulation de $H^0(\mathbf{D}_{\rm LT}(\mathcal{F}') \otimes \O(\lambda))=0$. Comme $\mathbf{D}_{\rm LT}(\mathcal{F}') \otimes \O(\lambda)$ est un fibr\'e de pentes positives, il est nul et donc 
$\mathbf{D}_{\rm LT}(\mathcal{F})$ est en fait isomorphe \`a $\rho^* \otimes \O_{\breve{\mathbf{P}}^1}$. Par cons\'equent, $\mathcal{F}$ est isomorphe \`a $\mathcal{G}$. Ainsi, on a bien $W=\O(\breve{\m}_{n})^{\rho}$. 
\end{proof}

\begin{remarque}
La preuve ci-dessus est simple mais ne fonctionne plus en dimension sup\'erieure (la th\'eorie de Kohlhaase n'est pas limit\'ee \`a la dimension $1$) ; toutefois, elle s'applique encore si l'on remplace $\qp$ par une extension finie. Elle utilise de fa\c con cruciale le fait que $\rho$ est lisse. Si $V(\mathbf{X})$ est le module de Dieudonn\'e rationnel du groupe formel de hauteur $2$ sur $\overline{\mathbf{F}}_p$, $V(\mathbf{X})$ est une repr\'esentation irr\'eductible de degr\'e $2$ de $D^*$ et pourtant l'on a une suite exacte de fibr\'es $D^*$-\'equivariants sur $\p1$ (\cite{GH}) :
\[ 0 \to \O_{\p1}(-1) \to V(\mathbf{X}) \otimes \O_{\p1} \to \O_{\p1}(1) \to 0, \]
qui tir\'ee en arri\`ere \`a $\breve{\mathcal{L}}_{0}^{(0)}$ redonne la suite exacte de fibr\'es \'equivariants fournie par la th\'eorie de Grothendieck-Messing.
\end{remarque}

Notons qu'au passage on a d\'emontr\'e la
\begin{proposition}\label{irreductible}
Le $\O(\Omega)$-module $\O(\Sigma_n)^{\rho}$ avec action semi-lin\'eaire de $G$ est topologiquement irr\'eductible.
\end{proposition}
\begin{proof}
Soit $V \subset \O(\Sigma_n)^{\rho}$ un sous $\O(\Omega)$-module ferm\'e non nul et stable sous l'action de $G$. Comme on l'a vu dans la preuve du lemme \ref{produitcomplete}, $\breve{V}$ est lui-m\^eme non nul et le paragraphe pr\'ec\'edent montre que le $\O(\breve{\Omega})$-module $\O(\breve{\Sigma}_n)^{\rho}$ avec action semi-lin\'eaire de $G$ est irr\'eductible. En particulier, $\breve{V}$ en est un sous-espace dense et par le lemme \ref{produitcomplete} on en d\'eduit que $V$ est dense dans $\O(\Sigma_n)^{\rho}$ ; comme il est ferm\'e, il est \'egal \`a tout l'espace. 
\end{proof}

\begin{remarque}
 Nous ne pr\'etendons pas avoir d\'emontr\'e que le $G$-module topologique $\O(\Sigma_n)^{\rho}$ est irr\'eductible ! Toutefois il l'est effectivement : combiner \cite{Colmezpoids} et le th\'eor\`eme \ref{main1}.
\end{remarque}

\begin{proposition}\label{Phi scalaire}
 Tout endomorphisme $\O(\Omega)$-lin\'eaire et $G$-\'equivariant de $\O(\Sigma_n)^{\rho}$ est scalaire.
\end{proposition}

\begin{proof}
On raisonne comme avant: il suffit de le v\'erifier apr\`es avoir \'etendu les scalaires \`a 
$\breve{\mathbf{Q}}_p$, dans quel cas cela suit de l'\'equivalence de cat\'egories de Kohlhaase et du fait que les endomorphismes 
du fibr\'e $D^*$-\'equivariant $\rho^* \otimes \O_{\breve{\mathbf{P}}^1}$ sont scalaires (par un argument semblable \`a celui utilis\'e pour d\'emontrer son irr\'eductibilit\'e). 
\end{proof}

\subsection{Un r\'esultat d'analyse fonctionnelle}

Si l'on savait que $\Pi(\pi,0)^*$ \'etait coadmissible comme $\O(\Omega)$-module, on pourrait conclure directement que $\Phi$ est surjective, car un morphisme continu entre modules coadmissibles sur une alg\`ebre de Fr\'echet-Stein est automatiquement d'image ferm\'ee \cite{STInv}. Mais il semble difficile de montrer un tel \'enonc\'e de coadmissibilit\'e - nous l'obtiendrons uniquement comme cons\'equence de notre r\'esultat principal ! Pour conclure que l'image de $\Phi$ est effectivement $\O(\Sigma_n)^{\rho}$ tout entier, il reste donc \`a d\'emontrer la
\begin{proposition}\label{ferme}
Soit $X$ une vari\'et\'e Stein sur $\qp$. Soit $N$ un $\O(X)$-module projectif de type fini et soit $Y$ un $\O(X)$-module qui soit un espace de Fr\'echet. Tout morphisme continu $F: Y \to N$ qui est $\O(X)$-lin\'eaire et d'image dense est surjectif. 
\end{proposition}

\begin{proof} Notons $N':=\mathrm{Im}(F) \subset N$. 
Par hypoth\`ese il existe un $\O(X)$-module $N''$ tel que $N \oplus N''= \O(X)^d$, $d>0$. En rempla\c cant
$N$ par $N \oplus N''=\O(X)^d$, $Y$ par\footnote{Noter que $N''$ (et donc $Y\oplus N''$) est naturellement muni d'une structure d'espace de Fr\'echet, car c'est un sous-$\O(X)$-module ferm\'e du Fr\'echet $\O(X)^d$.} $Y\oplus N''$ et $F$ par $F\oplus {\rm id}$, on peut donc supposer $N$ libre de rang fini. En raisonnant dans une base, on est m\^ eme ramen\'e au cas o\`u $N=\O(X)$, $N'=I$ est un id\'eal de $\O(X)$.

    Supposons que l'on peut trouver un nombre fini d'\'el\'ements $f_1,...,f_m \in I$ tels que $V(f_1,\dots,f_m)$ soit vide. 
L'id\'eal $J=(f_1,\dots,f_m)$ de $\O(X)$ est de type fini, donc coadmissible et donc ferm\'e \cite{STInv} dans $\O(X)$. De plus, si $(U_j)_j$ est un recouvrement Stein de $X$, alors 
$V(J\cap \O(U_j))=\emptyset$, donc $J\cap \O(U_j)=\O(U_j)$ par le Nullstellensatz affino\"ide, pour tout $j$. On en d\'eduit que $J$ est dense, et comme il est ferm\'e, on a $J=\O(X)$ et donc aussi $I=\O(X)$, ce qui finit la preuve de la proposition. 

      Il reste donc \`a montrer l'existence de $f_1,...,f_m$.
        On choisit les $f_i$ par r\'ecurrence. On prend $f_1\in I$ non nulle. Supposons $f_1,\dots,f_{i-1}$ choisis et $V(f_1,\dots,f_{i-1})$ non vide. 
Donnons-nous un recouvrement de Stein $(U_j)$ de $X$. Pour chaque $j$, $V(f_1,\dots,f_{i-1}) \cap U_j$ est affino\"ide et n'a donc qu'un nombre fini de composantes irr\'eductibles\footnote{Soit $A$ une alg\`ebre affino\"ide. Les composantes irr\'eductibles de l'espace affino\"ide $\mathrm{Sp} A$ sont les ensembles analytiques $\mathrm{Sp} A/\mathfrak{p}$, o\`u $\mathfrak{p}$ est un id\'eal premier minimal de $A$ ; comme $A$ est noeth\'erienne, il n'y en a qu'un nombre fini.}. On choisit pour chaque composante un point ferm\'e quelconque de cette composante. R\'ep\'etant cette op\'eration pour chaque $j$, on obtient une suite d\'enombrable $(z_n)_{n \in \mathbf{N}}$ de points de $X$. Pour chaque $k$, on choisit un \'el\'ement $\xi_k \in \O(X)$ s'annulant en $z_j$, pour tout $j=0,\dots,k-1$ et ne s'annulant pas en $z_k$ (l'existence de $\xi_k$ est facile si l'on souvient que, $X$ \'etant Stein, il existe une immersion ferm\'ee $\iota$ de $X$ dans l'espace affine ; il suffit de prendre $\xi_k$ de la forme $P \circ \iota$, o\`u $P$ est un polyn\^ome bien choisi). 

On va construire $f_i \in I$ ne s'annulant en aucun des $z_n$. Par hypoth\`ese, $I$ est dense dans $\O(X)$, il existe donc un \'el\'ement $h_n \in I$ ne s'annulant pas sur l'orbite de $z_n$, pour chaque $n \geq 0$. Par d\'efinition $h_n=F(y_n)$, avec $y_n \in Y$. Comme $Y$ est un espace de Fr\'echet, sa topologie est d\'efinie par une famille d\'enombrable de semi-normes $(q_n)_{n\geq 0}$. Posons $c_0=1$, puis choisissons $c_n \in \mathbf{Q}_p$ par r\'ecurrence sur $n$, de sorte que
$ \sum_{k=0}^n c_k  \xi_k h_k$
ne s'annule pas en $z_n$, ce qui est possible puisque $h_n(z_n) \neq 0$, et de sorte que
\[ |c_n| \max_{k=0,\dots,n} q_k ( \xi_n  y_n) \leq 2^{-n} \]
(cette expression a un sens, puisque $Y$ est un $\O(X)$-module). 
 Notons $$f_i=F( \sum_{k=0}^{\infty} c_k \xi_k  y_k)\in I,$$
 la somme
$\sum_{k=0}^{\infty} c_k \xi_k  y_k$
\'etant convergente dans $Y$. 
 Soit $n \geq 0$. Alors pour tout $k > n$, 
\[ F(c_k \xi_k y_k)(z_n)=c_k \xi_k(z_n) h_k(z_n) = 0, \]
et donc
\[ f_i(z_n)= \sum_{k=0}^n c_k \xi_k(z_n)  h_k (z_n) \neq 0, \]
par choix de la suite $c$. Donc $f_i$ ne s'annule en aucun des $z_n$. 

Soit $j\geq 0$. Si $Z$ est une composante irr\'eductible de l'affino\"ide $V(f_1,\dots,f_{i-1}) \cap U_j$, il existe $n$ tel que $z_n \in Z$ par choix de la suite $z$, mais par hypoth\`ese $f_i$ ne s'annule pas en $z_n$. Ainsi, $V(f) \cap Z$ est un ferm\'e strict de $Z$ et donc $\dim V(f) \cap Z < \dim Z$. Ceci valant pour chaque composante irr\'eductible de $V(f_1,\dots,f_{i-1}) \cap U_j$, on en d\'eduit que
\[ \dim V(f_1,\dots,f_i) \cap U_j \leq  \dim V(f_1,\dots,f_{i-1}) \cap U_j-1 \]
Comme $(U_j)$ est un recouvrement ouvert affino\"ide admissible de $X$, on en d\'eduit en passant \`a la limite sur $j$ que 
\[ \dim V(f_1,\dots,f_i) < \dim V(f_1,\dots,f_{i-1}), \]
comme voulu.
\end{proof}

    Le th\'eor\`eme \ref{surjectif tour} r\'esulte alors de la proposition pr\'ec\'edente\footnote{Noter que $\O(\Sigma_n)^{\rho}$ est un facteur direct de $\O(\Sigma_n)$ en tant que $\O(\Omega)$-module, et que 
$\O(\Sigma_n)$ est projectif de type fini comme $\O(\Omega)$-module \cite[prop. A5]{Kohl}} (avec 
$X=\Omega$, $Y=\Pi(\pi,0)^*$ et $N=\O(\Sigma_n)^{\rho}$)
et des propositions \ref{commute} et \ref{image dense}. 

\section{Injectivit\'e de $\Phi$ et fin de la preuve}\label{injectivit\'e}
 
Nous allons expliquer dans cette partie la preuve de l'injectivit\'e de $\Phi$ (ce qui terminera la d\'emonstration du th\'eor\`eme \ref{main1}) et celle du th\'eor\`eme \ref{main2}. Cela passe par l'introduction d'une nouvelle repr\'esentation de $G$, not\'ee $\Pi(\pi,2)$, dont le dual est construit \`a partir de la repr\'esentation $\Pi(\pi,0)^*$ et de l'op\'erateur $\partial$, ainsi que par l'\'etude  de l'action de $u^+$ sur $\Pi(\pi,2)$, et plus pr\'ecis\'ement du $G$-module $\Pi(\pi,2)^{u^+=0}$, reposant sur la th\'eorie du mod\`ele de Kirillov de Colmez rappel\'ee dans la section \ref{kirillovcolmez}. Nous commen\c cons donc par l\`a. Notons que l'injectivit\'e de $\Phi$ est une cons\'equence imm\'ediate de l'irr\'eductibilit\'e de $\Pi(\pi,0)$, qui est d\'emontr\'ee dans \cite{Colmezpoids}, mais nous aurons besoin de la plupart des constructions et r\'esultats techniques de ce chapitre dans la preuve du th\'eor\`eme \ref{main2}. Ces m\^emes r\'esultats permettent de d\'emontrer l'injectivit\'e de 
$\Phi$ sans utiliser son irr\'eductibilit\'e. 

\subsection{La repr\'esentation $\Pi(\pi,2)$}

    Nous avons d\'ej\`a observ\'e que la $G$-repr\'esentation 
        $\Omega^1(\Sigma_n)$ s'obtient directement \`a partir
        de $\O(\Sigma_n)$, en tordant par le cocycle 
       $c\in Z^1(G,\O(\Omega)^*)$, 
       $$c(g)=\det g\cdot (a-cz)^{-2}, \quad \text{si} \quad g= \left(\begin{matrix} a & b \\ c & d\end{matrix}\right)\in G.$$
        Nous allons construire un analogue $\Pi(\pi,2)^*$ de 
      $\Omega^1(\Sigma_n)$ \`a partir de $\Pi(\pi,0)^*$, par le m\^eme proc\'ed\'e : le travail d\'ej\`a effectu\'e rend la marche \`a suivre \'evidente.
      
           La premi\`ere partie du th\'eor\`eme \ref{adjoint} permet de d\'efinir une action de $G$ sur $\Pi(\pi,0)^*$ en posant 
        $$g*l=\det g\cdot (a-c\partial)^{-2}(g.l) \quad \text{si} \quad  g=\left(\begin{matrix} a & b \\ c & d\end{matrix}\right)\in G.$$
        Cette repr\'esentation se r\'ealise sur l'espace vectoriel topologique $\Pi(\pi,0)^*$, mais pour \'eviter les confusions il convient d'introduire une variable formelle 
        $dz$ et poser 
        \begin{eqnarray} \Pi(\pi,2)^*=\Pi(\pi,0)^*\,dz, \quad \text{avec} \quad g(l\,dz)=(\det g\cdot (a-c\partial)^{-2} (g.l)) \,dz \label{definition pi deux} \end{eqnarray} 
                pour $l\in \Pi(\pi,0)^*$ et $g= \left(\begin{matrix} a & b \\ c & d\end{matrix}\right)\in G$. Le th\'eor\`eme \ref{adjoint} et le corollaire \ref{closed image} se reformulent alors comme suit:
        
        \begin{proposition}\label{fleche d}
          $\Pi(\pi,2)^*$ est une repr\'esentation de $G$ et l'application 
          $$d: \Pi(\pi,0)^*\to \Pi(\pi,2)^*, \quad d(l)=-u^+(l) \,dz$$
          est $G$-\'equivariante, d'image ferm\'ee. 
        \end{proposition}

          \subsection{Le noyau de $u^+$ sur $\Pi(\pi,2)$} 
          
          Nous reprenons les notations des sections \ref{rappel phi Gamma} et \ref{kirillovcolmez}. Le but de ce paragraphe est d'\'etablir le th\'eor\`eme \ref{kirillov} ci-dessous.
          
          \begin{proposition}\label{plonger avec dependance} 
           Soit $V\in \mathcal{V}(\pi)$. Le plongement $v\mapsto \phi_v$ de $\Pi(V)^{P-\rm fini}$ dans 
           ${\rm LP}(\qpet, D_{\rm dif}^{-}(V))^{\Gamma}$ induit un plongement $P$-\'equivariant 
           $$\Psi_V: (\Pi(V)^{\rm an}/\Pi(V)^{\rm lisse})^{u^+=0}\to {\rm LP}(\qpet, t^{-1}N_{\rm dif}^+(V)/N_{\rm dif}^+(V))^{\Gamma}.$$
          \end{proposition}
          
          \begin{proof}
             En utilisant le th\'eor\`eme \ref{Ann}, on obtient 
             $$(\Pi(V)^{\rm an}/\Pi(V)^{\rm lisse})^{u^+=0}=(\Pi(V)^{\rm an})^{a^+u^+=(u^+)^2=0}/\Pi(V)^{\rm lisse}.$$
             Le plongement $\Pi(V)^{P-\rm fini}\subset {\rm LP}(\qpet, D_{\rm dif}^{-}(V))^{\Gamma}$ induit un plongement\footnote{On \'ecrit $t$ pour l'op\'erateur de multiplication par 
             $t$ sur $D_{\rm dif}^{-}(V)$.}
             $$(\Pi(V)^{\rm an})^{a^+u^+=(u^+)^2=0}\subset {\rm LP}(\qpet, (t^{-2}D_{\rm dif}^{+}(V)/D_{\rm dif}^+(V))^{\nabla t=0})^{\Gamma}.$$
             En utilisant la description explicite $$N_{\rm dif}^+(V)=L_{\infty}[[t]]\otimes_{L} D_{\rm dR}(V)\quad \text{et} \quad D_{\rm dif}^+(V)=tN_{\rm dif}^+(V)+L_{\infty}[[t]]\otimes_{L} {\rm Fil}^0(D_{\rm dR}(V)),$$
            un calcul imm\'ediat montre que l'inclusion 
             $N_{\rm dif}^+(V)\subset t^{-1}D_{\rm dif}^+(V)$ induit un isomorphisme canonique $\Gamma$-\'equivariant 
             $$t^{-1}N_{\rm dif}^+(V)/D_{\rm dif}^+(V)\simeq (t^{-2}D_{\rm dif}^{+}(V)/D_{\rm dif}^+(V))^{\nabla t=0}.$$
             On obtient donc un plongement $v\mapsto \phi_v$
             $$(\Pi(V)^{\rm an})^{a^+u^+=(u^+)^2=0}\subset {\rm LP}(\qpet, t^{-1}N_{\rm dif}^+(V)/D_{\rm dif}^+(V))^{\Gamma}.$$
             En combinant ceci avec le th\'eor\`eme \ref{Ann} on obtient le r\'esultat voulu. 
          \end{proof}
          
          Soit $V\in \mathcal{V}(\pi)$ et fixons un isomorphisme $\alpha: D_{\rm pst}(V)\simeq M(\pi)$. Il induit un isomorphisme de $\Gamma$-modules
          $\alpha_{\rm dif}: N_{\rm dif}^+(V)\simeq L_{\infty}[[t]]\otimes_{L} M_{\rm dR}(\pi)$ et donc un isomorphisme de $\Gamma$-modules 
           $$\alpha_{\rm dif}: t^{-1}N_{\rm dif}^+(V)/N_{\rm dif}^+(V)\simeq L_{\infty}(-1)\otimes_{L} M_{\rm dR}(\pi).$$
           L'isomorphisme $\alpha$ induit aussi un plongement $\alpha_{\p1}: (\Pi(V)^{\rm an}/\Pi(V)^{\rm lisse})^{*}\to tN_{\rm rig}(\pi)\boxtimes \p1$ (via l'isomorphisme
           $tN_{\rm rig}(V)\boxtimes \p1\simeq tN_{\rm rig}(\pi)\boxtimes \p1$ induit par $\alpha$) et par d\'efinition $\alpha_{\p1}:  (\Pi(V)^{\rm an}/\Pi(V)^{\rm lisse})^{*}\to \Pi(\pi,0)^*$ est un isomorphisme. On note $$\xi_{V,\alpha}: \Pi(\pi, 0)\simeq \Pi(V)^{\rm an}/\Pi(V)^{\rm lisse}$$ l'isomorphisme induit par la transpos\'ee de $\alpha_{\p1}$ compos\'e avec l'isomorphisme 
            $$[(\Pi(V)^{\rm an}/\Pi(V)^{\rm lisse})^{*}]^*\simeq \Pi(V)^{\rm an}/\Pi(V)^{\rm lisse}.$$
            En utilisant le plongement $\Psi_V$ introduit dans la proposition pr\'ec\'edente on obtient un plongement 
            $$\iota_{V,\alpha}: \Pi(\pi, 0)^{u^+=0}\to {\rm LC}(\qpet, L_{\infty}(-1)\otimes_{L} M_{\rm dR}(\pi)), \quad \iota_{V,\alpha}=\alpha_{\rm dif}\circ \Psi_V\circ (\xi_{V,\alpha}|_{\Pi(\pi,0)^{u^+=0}}).$$

          \begin{theoreme}\label{plonger sans d\'ependance}
               Le plongement $\iota_{V,\alpha}$ introduit ci-dessus est ind\'ependant, \`a homoth\'etie pr\`es, du choix de $V\in \mathcal{V}(\pi)$ et de l'isomorphisme 
               $\alpha: D_{\rm pst}(V)\simeq M(\pi)$. De plus, pour tout choix de $V$ et $\alpha$ 
           le diagramme suivant de $P$-repr\'esentations est commutatif \`a scalaire pr\`es 
               
$$\begin{array}[c]{ccc}
\Pi(\pi,0)^{u^+=0} &\stackrel{}{\rightarrow}& {\rm LC}(\qpet, L_{\infty}(-1)\otimes_{L} M_{\rm dR}(\pi))^{\Gamma}\\
\downarrow\scriptstyle{}&&\downarrow\scriptstyle{}\\
(\Pi(V)^{\rm an}/\Pi(V)^{\rm lisse})^{u^+=0} &\stackrel{}{\rightarrow}& {\rm LC}(\qpet, t^{-1}N_{\rm dif}^+(V)/N_{\rm dif}^+(V))^{\Gamma}
\end{array}.$$
 En particulier, on dispose d'un plongement $P$-\'equivariant canonique \`a scalaire pr\`es :
$$\iota: \Pi(\pi, 2)^{u^+=0}\to {\rm LC}(\qpet, L_{\infty}\otimes_{L} M_{\rm dR}(\pi))^{\Gamma}.$$
   \end{theoreme}
   
   \begin{proof}  La derni\`ere assertion d\'ecoule  
   imm\'ediatement de la premi\`ere, en utilisant le fait que comme $P$-repr\'esentations, $\Pi(\pi,0)$ et $\Pi(\pi,2)$ ne diff\`erent que par torsion par le caract\`ere $\left( \begin{smallmatrix}  a & b \\ 0 & 1 \end{smallmatrix}  \right) \mapsto a$. 
   
    L'ind\'ependance (\`a homoth\'etie pr\`es) par rapport \`a $\alpha$ (\`a $V$ fix\'e) est imm\'ediate (changer $\alpha$ en $x\alpha$ avec $x\in L$ a pour effet le changement de 
    $\iota_{V,\alpha}$ en $x^2\iota_{V,\alpha}$, car $\xi_{V,x\alpha}=x\xi_{V,\alpha}$), mais l'ind\'ependance par rapport \`a $V$ l'est nettement moins. 
    Nous aurons besoin du r\'esultat suivant, qui demande quelques pr\'eliminaires. Par construction on a 
    $\Pi(\pi, 0)^*\subset tN_{\rm rig}(\pi)\boxtimes \p1$ et m\^eme $\Pi(\pi,0)^*\subset tN^{]0,r_n]}\boxtimes \p1$ pour $n$ assez grand (d\'ependant de $V$ uniquement). Cela permet de d\'efinir 
    des applications 
    $$i_{j,n}: \Pi(\pi,0)^*\to t(N_{\rm rig}(\pi))_{\rm dif}^+=tL_{\infty}[[t]]\otimes_{L} M_{\rm dR}(\pi), \quad i_{j,n}=\varphi^{-n}\circ {\rm Res}_{\zp}\circ  \left(\begin{matrix} p^{n-j} & 0 \\ 0 & 1\end{matrix}\right)$$
    pour $n$ assez grand et $j\in\mathbf{Z}$. Notons que $i_{j,n}(L)=\alpha_{\rm dif}(i_{j,n}(l))$ si $l\in (\Pi(V)^{\rm an}/\Pi(V)^{\rm lisse})^*$ satisfait $L=\alpha_{\p1}(l)$ (il suffit de suivre les d\'efinitions). Enfin, on \'ecrit $\{\,\,\}_{\rm dif,V}$ pour l'accouplement entre $tN_{\rm dif}^+(V)$ et $t^{-1}N_{\rm dif}^+(V)/N_{\rm dif}^+(V)$ induit par l'accouplement entre 
    $D_{\rm dif}(\check{V})^+[1/t]=D_{\rm dif}(V)^+[1/t]$ et $D_{\rm dif}(V)^{+}[1/t]$.
    
    \begin{lemme} Soit $\Pi(\pi,0)^{u^+=0}_{c}$ le sous-espace de $\Pi(\pi,0)^{u^+=0}$ engendr\'e par les $(1-n)v$ avec
    $n\in  \left(\begin{smallmatrix} 1 & \qp \\ 0 & 1\end{smallmatrix}\right)$ et $v\in \Pi(\pi,0)^{u^+=0}$. 
    Si $L\in \Pi(\pi,0)^*$ et $v\in \Pi(\pi,0)^{u^+=0}_c$, alors pour tout $n$ assez grand
    $$L(v)=\sum_{j\in\mathbf{Z}} \{\alpha_{\rm dif}^{-1}(i_{j,n}(L)), \alpha_{\rm dif}^{-1}(\iota_{V,\alpha}(v))(p^{-j})\}_{\rm dif,V}.$$
    \end{lemme}
    
    \begin{proof}
      Ecrivons $L=\alpha_{\p1}(l)$ avec $l\in (\Pi(V)^{\rm an}/\Pi(V)^{\rm lisse})^*$. L'\'egalit\'e 
      $L(v)=l(\xi_{V,\alpha}(v))$ d\'ecoule de la d\'efinition de $\xi_{V,\alpha}$. 
      Ensuite, le vecteur $v_1=\xi_{V,\alpha}(v)$ est dans le sous-espace de 
      $ (\Pi(V)^{\rm an}/\Pi(V)^{\rm lisse})^{u^+=0}\subset \Pi(V)^{P-\rm fini}/\Pi(V)^{\rm lisse}$ engendr\'e par les $(1-n)y$ avec 
      $y\in \Pi(V)^{P-\rm fini}/\Pi(V)^{\rm lisse}$. On en d\'eduit que $v_1\in \Pi(V)^{P-\rm fini}_c/\Pi(V)^{\rm lisse}$, ce qui permet d'utiliser le th\'eor\`eme 
      \ref{dualKir} et obtenir ainsi 
      $$L(v)=l(v_1)=\sum_{j\in\mathbf{Z}} \{i_{j,n}(l), \Psi_V(v_1)(p^{-j})\}_{\rm dif,V}.$$
      On conclut en utilisant les \'egalit\'es $i_{j,n}(L)=\alpha_{\rm dif}(i_{j,n}(l))$ et 
      $ \Psi_V(v_1)=\alpha_{\rm dif}^{-1}(i_{V,\alpha}(v))$.
      
    \end{proof}
    
     Consid\'erons maintenant deux repr\'esentations $V_1,V_2\in \mathcal{V}(\pi)$ et deux isomorphismes 
     $\alpha_1: D_{\rm pst}(V_1)\simeq M(\pi)$, $\alpha_2: D_{\rm pst}(V_2)\simeq M(\pi)$. Soit 
     $u=\alpha_2^{-1}\circ \alpha_1: D_{\rm pst}(V_1)\simeq D_{\rm pst}(V_2)$. En appliquant le lemme pr\'ec\'edent avec
     $V=V_1$, puis avec $V=V_2$ et en comparant les r\'esultats on obtient pour tous $L\in \Pi(\pi,0)^*$, $v\in  \Pi(\pi,0)^{u^+=0}_c$ et $n$ assez grand
     $$\sum_{j\in\mathbf{Z}} \{\alpha_{1, \rm dif}^{-1}(i_{j,n}(L)), \alpha_{1, \rm dif}^{-1}(\iota_{V_1,\alpha_1}(v))(p^{-j})\}_{\rm dif,V_1}=$$ $$\sum_{j\in\mathbf{Z}}
     \{u_{\rm dif}(\alpha_{1, \rm dif}^{-1}(i_{j,n}(L))), u_{\rm dif}(\alpha_{1, \rm dif}^{-1}(\iota_{V_2,\alpha_2}(v))(p^{-j}))\}_{\rm dif,V_2}.$$
      Puisque les accouplements $\{\,\,\}_{\rm dif, V_j}$ sont induits par des isomorphismes 
      $\wedge^2(D_{\rm pst}(V_j))\simeq L$, $u_{\rm dif}$ est compatible avec ces accouplements \`a un scalaire $C$ pr\`es. Ainsi, l'\'egalit\'e pr\'ec\'edente s'\'ecrit
      $$\sum_{j\in\mathbf{Z}} \{\alpha_{1,\rm dif}^{-1}(i_{j,n}(L)), \alpha_{1,\rm dif}^{-1}( \iota_{V_1,\alpha_1}(v)-C\iota_{V_2,\alpha_2}(v))(p^{-j})\}_{\rm dif, V_1}=0.$$
 
       Nous avons maintenant besoin du 
       
       \begin{lemme}
       Soit $(x_j)_{j\in\mathbf{Z}}$ une suite d'\'el\'ements de $t^{-1}N_{\rm dif}^+(V_1)/N_{\rm dif}^+(V_1)$, \`a support fini et telle que pour tout
       $L\in \Pi(\pi,0)^*$ et tout $n$ assez grand 
       $$\sum_{j\in\mathbf{Z}} \{\alpha_{1,\rm dif}^{-1}(i_{j,n}(L)), x_j\}_{\rm dif, V_1}=0.$$
       Alors $x_j=0$ pour tout $j$.
       \end{lemme}
       
       \begin{proof} On a donc pour tout $l\in (\Pi(V_1)^{\rm an}/\Pi(V_1)^{\rm lisse})^*$ 
       $$\sum_{j\in\mathbf{Z}} \{i_{j,n}(l), x_j\}_{\rm dif, V_1}=0.$$
              La proposition \ref{inclusion Kir} permet de construire un vecteur $v\in \Pi(V_1)^{P-\rm fini}_c$ tel que 
       $\phi_v(p^{-j})=x_j$ pour tout $j$. Le th\'eor\`eme \ref{dualKir} combin\'e avec l'hypoth\`ese montre que 
       $l(v)=0$ pour tout $l\in (\Pi(V_1)^{\rm an}/\Pi(V_1)^{\rm lisse})^*$, autrement dit $v\in \Pi(V_1)^{\rm lisse}$. On conclut en utilisant la proposition 
       \ref{P-lisse}.
       \end{proof}
       
        On d\'eduit de ce qui pr\'ec\`ede que $\iota_{V_1,\alpha_1}-C\iota_{V_2,\alpha_2}$ s'annule sur $\Pi(\pi,0)^{u^+=0}_{c}$, autrement dit que 
       l'image de $\iota_{V_1,\alpha_1}-C\iota_{V_2,\alpha_2}$ est contenue dans le sous-espace des $ \left(\begin{smallmatrix} 1 & \qp \\ 0 & 1\end{smallmatrix}\right)$-invariants de 
              ${\rm LC}(\qpet, L_{\infty}(-1)\otimes_{L} M_{\rm dR}(\pi)^{\Gamma}$. Comme ce sous-espace est nul, cela permet de conclure. 
       
   \end{proof}

Pour pouvoir d\'ecrire plus compl\`etement $\Pi(\pi,2)^{u^+=0}$, il faut maintenant comprendre le lien entre $\Pi(\pi,2)$ et les $\Pi^{\rm an}$, pour $\Pi \in \mathcal{V}(\pi)$. 
  \begin{proposition}\label{uplus plongement} Soit $\Pi\in \mathcal{V}(\pi)$. 
Le choix d'un isomorphisme $\Pi(\pi,0)\simeq \Pi^{\mathrm{an}}/\Pi^{\mathrm{lisse}}$ d\'etermine un plongement $G$-\'equivariant de $(\Pi^{\mathrm{an}})^*$ dans $\Pi(\pi,2)^*$, qui fait de $(\Pi^{\mathrm{an}})^*$ un sous-espace $G$-stable de $\Pi(\pi,2)^* $ contenant $d(\Pi(\pi,0)^*)$.
\end{proposition}

\begin{proof} L'application 
$$d: (\Pi^{\mathrm{an}})^*\to (\Pi^{\rm an}/\Pi^{\rm lisse})^*\simeq \Pi(\pi,0)^*\to \Pi(\pi,2)^*,$$
la premi\`ere fl\`eche \'etant $l\mapsto -u^+(l)$ et la derni\`ere \'etant simplement $l \mapsto l \,dz$ 
est bien d\'efinie et un plongement car $u^+$ est injectif d'image ferm\'ee sur $\Pi(\pi,0)^*$. 
Il reste \`a montrer que $d$ est $G$-\'equivariante, ce qui revient \`a montrer que 
$$u^+(gl)=\det g\cdot (a-c\partial)^{-2} g(u^+l)$$
pour tout $l\in (\Pi^{\mathrm{an}})^*$ et tout $g\in G$. Si 
$c=0$, d\'ecoule de l'identit\'e 
$$u^+\cdot \left(\begin{matrix} a & b \\ 0 & d\end{matrix}\right)=\frac{d}{a} \left(\begin{matrix} a & b \\ 0 & d\end{matrix}\right)\cdot u^+$$
dans $D(G)$ et du fait que $(\Pi^{\mathrm{an}})^*$ est un $D(G)$-module. Ainsi, $F$ est \'equivariante pour l'action du Borel $B$ de $G$ et il suffit de tester que 
$d(wl)=wd(l)$ pour $l\in (\Pi^{\mathrm{an}})^*$. Cela revient \`a $u^+wl=-\partial^{-2} w(u^+l)$, ou encore (en appliquant $w$ et en utilisant la relation 
$w\partial^{-2}w=\partial^2$ sur $\Pi(\pi,0)^*$)
$u^-l=-\partial^2 u^+l$. Cela d\'ecoule de la proposition \ref{numer}. 
\end{proof}

\begin{theoreme}\label{kirillov} Le sous-espace $\Pi(\pi,2)^{u^+=0}$ de $\Pi(\pi,2)$ est stable par $G$ et 
on a un isomorphisme, canonique \`a scalaire pr\`es, de repr\'esentations de $G$ :
\[ \Pi(\pi,2)^{u^+=0} \simeq \pi \otimes M_{\rm dR}(\pi). \]
\end{theoreme}
\begin{proof}
La stabilit\'e de $\Pi(\pi,2)^{u^+=0}$ par $G$ vient du fait que $\Pi(\pi,2)^{u^+=0}$ est dual du conoyau de la fl\`eche $G$-\'equivariante $d : \Pi(\pi,0)^* \to \Pi(\pi,2)^*$ (voir les propositions \ref{imagefermee} et \ref{fleche d}). Montrons tout d'abord que $\Pi(\pi,2)^{u^+=0}$ est une repr\'esentation lisse de $G$. Commençons par 

\begin{lemme}
Pour tout $v \in \Pi(\pi,2)$ l'application 
\[ f_v : (a,b,c,d) \mapsto \left(\begin{matrix} a & b \\ c & d\end{matrix}\right).v \]
est localement analytique en chacun des arguments $a, b, c, d$ au voisinage de $(1,0,0,1)$.
\end{lemme}

\begin{proof}
La repr\'esentation $\Pi(\pi,0)$ est localement analytique et comme repr\'esentations du Borel $B$, $\Pi(\pi,2)$ et $\Pi(\pi,0)$ ne diff\`erent que par torsion par le caract\`ere de $B$, $ \left(\begin{smallmatrix} a & b \\ 0 & d\end{smallmatrix}\right) \mapsto d/a$. Donc $f_v$ est localement analytique en $a, b$ et $d$. De plus,
\[ f_v(1,0,x,1) = \left(\begin{matrix} 1 & 0 \\ x & 1\end{matrix}\right).v= w\left(\begin{matrix} 1 & x \\ 0 & 1\end{matrix}\right)w.v = wf_{w.v}(1,x,0,1), \]
ce qui permet de conclure. 
\end{proof}

  En utilisant le lemme pr\'ec\'edent et un calcul imm\'ediat, on obtient 
  les formules suivantes pour l'action de 
$\mathfrak{g}$ sur $\Pi(\pi,2)$ 
$$a^+=\partial u^+, \quad a^-=-\partial u^+, \quad u^-=-\partial^2 u^+.$$
Ainsi, un vecteur tu\'e par $u^+$ est automatiquement tu\'e par $\mathfrak{g}$. Une nouvelle application du lemme pr\'ec\'edent permet de conclure que 
$\Pi(\pi,2)^{u^+=0}$ est une repr\'esentation lisse de $G$. 

Choisissons maintenant $V_1,V_2\in \mathcal{V}(\pi)$ non isomorphes (cela correspond au choix de deux filtrations diff\'erentes sur $M_{\rm dR}(\pi)$). 
Fixons 
des isomorphismes 
\[ \Pi(\pi,0) \simeq \Pi(V_1)^{\rm an}/\Pi(V_2)^ {\rm lisse} \quad \text{et} \quad \Pi(\pi,0) \simeq \Pi(V_2)^{\rm an}/\Pi(V_2)^{\rm lisse}. \]
D'apr\`es la proposition \ref{uplus plongement}, on peut donc voir $(\Pi(V_1)^{\rm an})^*$ et $(\Pi(V_2)^{\rm an})^*$ comme deux sous-espaces de $\Pi(\pi,2)^*$ contenant $d(\Pi(\pi,0)^*)$. Notons $X_1$ et $X_2$ les images de ces sous-espaces dans le quotient $\Pi(\pi,2)^*/d((\Pi(\pi,0)^*)$. Comme repr\'esentation de $G$
$$X_j\simeq (\Pi(V_j)^{\rm an})^*/\Pi(\pi,0)^*\simeq \pi^*.$$
 En particulier, $X_1$ et $X_2$ sont des repr\'esentations irr\'eductibles de $G$
et donc $X_1$ et $X_2$ sont \'egaux s'ils sont d'intersection non nulle. S'ils \'etaient \'egaux, on aurait $(\Pi(V_1)^{\rm an})^*=(\Pi(V_2)^{\rm an})^*$ (puisque ces deux sous-espaces contiennent $d(\Pi(\pi,0)^*)$), et un des r\'esultats principaux de \cite{CD} entra\^ine que $\Pi(V_1)\simeq \Pi(V_2)$, ce qui contredit le fait que $V_1$ et $V_2$ ne sont pas isomorphes. On en d\'eduit que $X_1$ et $X_2$ sont en somme directe, ce qui donne dualement une surjection $G$-\'equivariante
\[  \Pi(\pi,2)^{u^+=0}=(\Pi(\pi,2)^*/d(\Pi(\pi,0)^*))^*  \twoheadrightarrow \pi^{\oplus 2}. \]
Comme on vient de voir que $\Pi(\pi,2)^{u^+=0}$ est une $G$-repr\'esentation lisse et comme $\pi$ est supercuspidale, ce morphisme se scinde et on peut donc \'ecrire :
\[ \Pi(\pi,2)^{u^+=0} = \pi^{\oplus 2} \oplus \xi, \]
o\`u $\xi$ est une repr\'esentation lisse de $G$.  

 Le th\'eor\`eme \ref{plonger sans d\'ependance} permet de plonger $\Pi(\pi,2)^{u^+=0}$ dans  ${\rm LC}(\qpet,L_{\infty} \otimes M_{\rm dR}(\pi))^{\Gamma}$, de fa\c con $P$-\'equivariante. Si $X$ est une repr\'esentation de $P$, on note $X_c$ le sous-espace engendr\'e par les \'el\'ements de la forme $(1-n).v$, avec $n\in  \left(\begin{smallmatrix} 1 & \qp \\ 0 & 1\end{smallmatrix}\right) $ et $v \in X$. On a donc
\[ \pi_c^{\oplus 2} \oplus \xi_c \hookrightarrow (({\rm LC}(\qpet,L_{\infty})^{\Gamma})_c)^{\oplus 2} = ({\rm LC}_c(\qpet,L_{\infty})^{\Gamma})^{\oplus 2}. \]
Or $\pi_c$ et ${\rm LC}_c(\qpet,L_{\infty})^{\Gamma}$ sont isomorphes comme $P$-repr\'esentations et irr\'eductibles (c'est un r\'esultat standard de la th\'eorie du mod\`ele de Kirillov des repr\'esentations supercuspidales de $G$). On a donc n\'ecessairement $\xi_c=0$. En d'autres termes, tout vecteur de $\xi$ est fix\'e par l'action de l'unipotent sup\'erieur : mais il n'y a pas de tel vecteur non nul dans ${\rm LC}(\qpet,L_{\infty})^{\Gamma}$ ! On en d\'eduit que $\xi=0$, et $\Pi(\pi,2)^{u^+=0}$ et $\pi \otimes M_{\rm dR}(\pi)$ sont donc deux repr\'esentations isomorphes de $G$. Il reste donc pour conclure \`a exhiber un isomorphisme $G$-\'equivariant canonique, \`a scalaire pr\`es, entre ces repr\'esentations.

Ce qui pr\'ec\`ede donne en particulier que l'image de $\Pi(\pi,2)^{u^+=0}$ dans ${\rm LC}(\qpet, L_{\infty}\otimes_{L} M_{\rm dR}(\pi))^{\Gamma}$ est le sous-espace ${\rm LC}_c(\qpet, L_{\infty}\otimes_{L} M_{\rm dR}(\pi))^{\Gamma}$. La th\'eorie du mod\`ele de Kirillov usuelle pour les repr\'esentations lisses donne en outre un isomorphisme $P$-\'equivariant $\pi \otimes M_{\rm dR}(\pi) \simeq {\rm LC}_c(\qpet, L_{\infty}\otimes_{L} M_{\rm dR}(\pi))^{\Gamma}$. On obtient donc en composant l'inverse de cet isomorphisme avec la fl\`eche fournie par le th\'eor\`eme \ref{plonger sans d\'ependance} un isomorphisme $P$-\'equivariant, canonique \`a scalaire pr\`es
\[ \Pi(\pi,2)^{u^+=0} \simeq \pi \otimes M_{\rm dR}(\pi). \]
Cet isomorphisme est automatiquement $G$-\'equivariant (car les deux membres sont isomorphes comme $G$-repr\'esentations \`a $\pi \oplus \pi$, et $\pi$ est irr\'eductible comme $P$-repr\'esentation).
\end{proof}

En retour, le th\'eor\`eme pr\'ec\'edent permet de d\'ecrire les positions relatives des sous-repr\'esentations $(\Pi^{\rm an})^*$ \`a l'int\'erieur de $\Pi(\pi,2)^*$. Soit $\mathcal{L}$ une filtration sur $M_{\mathrm{dR}}(\pi)$. Il lui correspond une repr\'esentation $V_{\mathcal{L}} \in \mathcal{V}(\pi)$, qui vient avec un isomorphisme $D_{\rm pst}(V_{\mathcal{L}}) \simeq M(\pi)$. 
\begin{corollaire}\label{preimage}
La pr\'eimage de $\mathcal{L}^{\perp} \otimes \pi^* \subset (\Pi(\pi,2)^{u^+=0})^*$ dans $\Pi(\pi,2)^*$ est $(\Pi_{\mathcal{L}}^{\mathrm{an}})^*$ (vu dans $\Pi(\pi,2)^*$ par la proposition \ref{uplus plongement}).
\end{corollaire}
\begin{proof}
Il s'agit de v\'erifier que le quotient de $d((\Pi_{\mathcal{L}}^{\mathrm{an}})^*)$ par $d(\Pi(\pi,0)^*)$ est isomorphe \`a $\mathcal{L}^{\perp} \otimes \pi^*$, vu comme sous-espace du conoyau de $d : \Pi(\pi,0)^* \to \Pi(\pi,2)^*$ par le th\'eor\`eme \ref{kirillov}. L'isomorphisme $\alpha : D_{\rm pst}(V_{\mathcal{L}}) \simeq M(\pi)$ identifie ce quotient \`a $(\Pi_{\mathcal{L}}^{\mathrm{lisse}})^*$. Dans la construction de l'isomorphisme du th\'eor\`eme \ref{kirillov}, on peut choisir $V=V_{\mathcal{L}}$. Admettons qu'on a alors un diagramme commutatif
$$\begin{array}[c]{ccc}
\Pi(\pi,2)^{u^+=0} &\stackrel{}{\longrightarrow}& \Pi_{\mathcal{L}}^{\mathrm{lisse}}\\
\downarrow\scriptstyle{}&&\downarrow\scriptstyle{}\\
{\rm LC}_c(\qpet, L_{\infty} \otimes M_{\rm dR}(\pi))^{\Gamma} &\stackrel{}{\rightarrow}& {\rm LC}_c(\qpet, L_{\infty}[[t]] \otimes M_{\rm dR}(\pi)/\mathrm{Fil}^0(L_{\infty}((t)) \otimes M_{\rm dR}(\pi)))^{\Gamma}
\end{array}$$
o\`u les fl\`eches verticales sont des isomorphismes (celle de droite s'obtient en composant le mod\`ele de Kirillov usuel de $\Pi_{\mathcal{L}}$ avec $\alpha_{\rm dif}$). Comme de plus, toutes les identifications effectu\'ees sont canoniques \`a scalaire pr\`es, le r\'esultat s'ensuit.

Reste \`a justifier la commutativit\'e du diagramme pr\'ec\'edent. Comme on a choisi $V=V_{\mathcal{L}}$ dans \ref{kirillov}, il s'agit, par d\'efinition de $\iota_{V_{\mathcal{L}},\alpha}$, de justifier que le diagramme dont les fl\`eches sont $ \left(\begin{smallmatrix} 1 & \qp \\ 0 & 1\end{smallmatrix}\right)$-\'equivariantes
$$\begin{array}[c]{ccc}
(\Pi_{\mathcal{L}}^{\rm an}/\Pi_{\mathcal{L}}^{\rm lisse})^{u^+=0}=(\Pi_{\mathcal{L}}^{\rm an})^{a^+u^+=(u^+)^2=0}/\Pi_{\mathcal{L}}^{\rm lisse} &\stackrel{}{\overset{u^+}\longrightarrow}& \Pi_{\mathcal{L}}^{\mathrm{lisse}}\\
\downarrow\scriptstyle{\Psi_{V_{\mathcal{L}}}}&&\downarrow\scriptstyle{}\\
{\rm LC}_c(\qpet, L_{\infty} \otimes t^{-1}N_{\rm dif}^+(V_{\mathcal{L}})/N_{\rm dif}^+(V_{\mathcal{L}}))^{\Gamma} &\stackrel{}{\overset{t}\rightarrow}& {\rm LC}_c(\qpet, L_{\infty} \otimes N_{\rm dif}^+(V_{\mathcal{L}})/D_{\rm dif}^+(V_{\mathcal{L}}))^{\Gamma}
\end{array}$$
est commutatif, puisque, par d\'efinition (voir la proposition \ref{fleche d}), la fl\`eche $(\Pi_{\mathcal{L}}^{\rm an})^* \to \Pi(\pi,2)^*$ \'etait d\'eduite de $d=-u^+ : (\Pi_{\mathcal{L}}^{\rm an})^* \to (\Pi_{\mathcal{L}}^{\rm an}/\Pi_{\mathcal{L}}^{\rm lisse})^*$. Comme les deux fl\`eches verticales sont induites par le plongement $\Pi(V_{\mathcal{L}})^{P-\mathrm{fini}} \to \mathrm{LP}(\qpet, D_{\rm dif}^{-}(V_{\mathcal{L}}))^{\Gamma}$, la commutativit\'e de ce diagramme se r\'eduit simplement au fait que ce plongement est $ \left(\begin{smallmatrix} 1 & \qp \\ 0 & 1\end{smallmatrix}\right)$-\'equivariant.
\end{proof}

\subsection{D\'emonstrations des th\'eor\`emes \ref{main1} et \ref{main2}}

La proposition \ref{commute} montre que l'on peut aussi voir $\Phi$ comme un morphisme
\[ \Phi : \Pi(\pi,2)^* \to \Omega^1(\Sigma_n)^{\rho}, \]
ce que l'on fait dans la suite de cette section.

\begin{proposition}\label{injectif}
Le morphisme $\Phi$ est injectif.
\end{proposition}
\begin{proof}
Soit $v \in \Pi(\pi,2)^*$ tel que $\Phi(v)=0$. En particulier, l'image de $\Phi(v)$ dans le quotient $H_{\mathrm{dR}}^1(\Sigma_n)^{\rho}$ est nulle. La compos\'ee de $\Phi$ et de la surjection sur la cohomologie de de Rham est $G$-\'equivariante donc se factorise par le quotient de $\Pi(\pi,2)^*$ par l'image de $u^+$ : comme celle-ci est ferm\'ee par la proposition \ref{imagefermee}, ce quotient s'identifie, comme on l'a d\'ej\`a vu, \`a $((\Pi(\pi,2))^{u^+=0})^*$ et on a donc une fl\`eche $G$-\'equivariante \textit{surjective} $((\Pi(\pi,2))^{u^+=0})^* \to H_{\mathrm{dR}}^1(\Sigma_n)^{\rho}$. Or on a d\'ej\`a vu que, comme $G$-repr\'esentations, $((\Pi(\pi,2))^{u^+=0})^* \simeq (\pi^*)^{\oplus 2}$ et que $H_{\mathrm{dR}}^1(\Sigma_n)^{\rho}$ a pour quotient $(\pi^*)^{\oplus 2}$\footnote{Dualement, cela revient \`a dire que $H_{\mathrm{dR,c}}^1(\Sigma_n)^{\rho^{\vee}}$ contient $\pi^{\oplus 2}$ comme sous-objet. Le th\'eor\`eme \ref{derham} dit exactement cela, avec sous-objet remplac\'e par quotient ; mais cela suffit, puisque $\pi$ est supercuspidale, donc est un objet projectif de la cat\'egorie des repr\'esentations lisses de $G$ \`a caract\`ere central trivial.}. La fl\`eche consid\'er\'ee ne peut donc \^etre qu'un isomorphisme. Par cons\'equent, $v \in \mathrm{Im} u^+$ : il existe $v' \in \Pi(\pi,2)^*$ tel que $v'=u^+. v$. Or
\[  0 = \Phi(v) = \Phi(u^+. v') = u^+. \Phi(v'). \]
Comme $\mathrm{Ker} (u^+)=0$ sur $\Omega^1(\Sigma_n)^{\rho}$, on en d\'eduit que $\Phi(v')=0$. R\'ep\'etant l'argument, on voit qu'un vecteur $v$ d'image nulle par $\Phi$ est en fait dans $\mathrm{Im} (u^+)^j$ pour tout $j$. 

Voyons l'\'el\'ement $v'$ ci-dessus comme un \'el\'ement de $\Pi(\pi,0)^*$, ce qui est loisible puisque les espaces vectoriels topologiques sous-jacents \`a $\Pi(\pi,0)^*$ et $\Pi(\pi,2)^* $ sont les m\^emes. L'op\'erateur $u^+$ agit de la m\^eme mani\`ere sur $\Pi(\pi,0)^*$ et $\Pi(\pi,2)^*$ et est injectif, donc $v'$ est dans $\mathrm{Im} (u^+)^j$ pour tout $j$. De plus, comme $u^+ : \Pi(\pi,0)^* \to \Pi(\pi,2)^*$ est $G$-\'equivariant et comme $w.v$ est aussi dans le noyau de $\Phi$, on a aussi que $w.v'$ est dans $\mathrm{Im} (u^+)^j$ pour tout $j$. Voyons $v'$ dans $t N_{\mathrm{rig}}(\pi) \boxtimes \mathbf{P}^1$ : $v'=(z_1,z_2)$. Comme $v' \in \mathrm{Im} (u^+)^j$ pour tout $j >0$, $z_1$ est infiniment divisible par $t$. On voit imm\'ediatement que cela force $z_1=0$ en prenant une base de $N_{\rm rig}(\pi)$ sur $\mathcal{R}$. Mais $w.v'$ a la m\^eme propri\'et\'e et $w.v'=(z_2,z_1)$. On a donc aussi $z_2=0$ et donc $v'=0$. Donc $v$ est lui-m\^eme nul. 
\end{proof}

\begin{remarque}
On verra plus loin (d\'emonstration du th\'eor\`eme \ref{corollaireter}) que l'intersection des $\mathrm{Im}((u^+)^j)$, $j\geq 0$, sur $\Pi(\pi,0)^*$ (ou $\Pi(\pi,2)^*$) est en fait nulle, mais cela fait appel \`a un r\'esultat d\'elicat de \cite{Dthese}, que l'argument pr\'ec\'edent permet d'\'eviter.
\end{remarque}

Ceci termine la preuve du th\'eor\`eme \ref{main1}. 

\begin{proof}[D\'emonstration du th\'eor\`eme \ref{main2}]
L'isomorphisme $\Phi$ induit un isomorphisme
\[ (\Pi(\pi,2)^{u^+=0})^* \simeq H_{\rm dR}^1(\Sigma_n)^{\rho}. \]
De plus, un tel isomorphisme $\Phi$ est unique \`a scalaire pr\`es, d'apr\`es les propositions \ref{commute} et \ref{Phi scalaire}. En combinant ceci avec le th\'eor\`eme \ref{kirillov}, on en d\'eduit que l'on a un isomorphisme canonique \`a scalaire pr\`es
\[ H_{\rm dR}^1(\Sigma_n)^{\rho} \simeq \pi^* \otimes M_{\rm dR}(\pi)^*. \]
Une fois ceci acquis, la deuxi\`eme partie du th\'eor\`eme se d\'eduit trivialement du corollaire \ref{preimage}.
\end{proof}

\begin{remarque}\label{poids generaux}
Tous les r\'esultats de cet article s'\'etendent aux fibr\'es $\O(k)$, $k \in \z$. Dans l'\'enonc\'e du th\'eor\`eme \ref{main1}, il suffit de remplacer $\Pi(\pi,0)$ par la repr\'esentation de $G$ dont le dual est l'espace topologique $\Pi(\pi,0)^*$ avec action de $G$ tordue par $(a-c \partial)^{-k}$. 

Pour le th\'eor\`eme \ref{main2}, il faut cette fois-ci consid\'erer, pour $k\geq 0$, la suite exacte
\[ 0 \to \O(-k)(\Sigma_n)^{\rho} \overset{(u^+)^{k+1}}{\longrightarrow} \O(k+2)(\Sigma_n)^{\rho} \otimes \det{}^{k+1} \to H_{\rm dR}^1(\Sigma_n)^{\rho} \otimes \mathrm{Sym}^k \to 0, \]
afin de d\'ecrire les vecteurs localement analytiques des repr\'esentations de Banach attach\'ees aux repr\'esentations \`a poids $0, k+1$. L'existence de cette suite exacte se justifie comme dans \cite[p. 95-97]{SS}.
\end{remarque}

\begin{remarque}\label{difference}
La conjecture originale de Breuil-Strauch \cite{BS} \'etait formul\'ee de fa\c con l\'eg\`erement diff\'erente. Nous expliquons maintenant le lien avec les r\'esultats de cet article.

Les auteurs de \cite{BS} travaillent avec le rev\^etement de Drinfeld de niveau $1+\varpi_D \O_D$, et nous noterons $\Sigma_{1+\varpi_D \O_D}$ le mod\`ele sur $\qp$ de son quotient par l'action de $p^{\z}$. Soit $K_0=\mathbf{Q}_{p^2}$ et $K=\mathbf{Q}_{p^2}(\sqrt[p^2-1]{-p})$. Dans \cite{teitel}, Teitelbaum a construit un mod\`ele formel semi-stable minimal $\widehat{\Sigma}_{1+\varpi_D \O_D,K}$ de $\Sigma_{1+\varpi_D \O_D,K}$, qui donne par uniformisation $p$-adique un mod\`ele semi-stable $\widehat{\mathrm{Sh}}_{(1+\varpi_D\O_D)K^p,K}$ de la courbe de Shimura $\mathrm{Sh}_{(1+\varpi_D \O_D)K^p,K}$. Les r\'esultats de \cite{GKrig}, qui \'etendent la th\'eorie de Hyodo-Kato \`a des vari\'et\'es rigides non n\'ecessairement propres, permettent de d\'efinir la cohomologie log-rigide de la fibre sp\'eciale de $\widehat{\Sigma}_{1+\varpi_D \O_D}$, qui est un $(\varphi,N,\mathcal{G}_{\qp})$-$K_0$-module $H_{\rm HK}^1(\widehat{\Sigma}_{1+\varpi_D \O_D,K})$, avec un isomorphisme $H_{\rm HK}^1(\widehat{\Sigma}_{1+\varpi_D \O_D,K}) \otimes_{K_0} K \simeq H_{\rm dR}^1(\Sigma_{1+\varpi_D \O_D,K})$. On montre avec des arguments semblables\footnote{La seule chose \`a savoir est l'existence d'une suite spectrale pour la cohomologie log-rigide pour un quotient $\widehat{\Sigma}_{1+\varpi_D \O_D} \to \Gamma \backslash \widehat{\Sigma}_{1+\varpi_D \O_D}$, avec $\Gamma$ sous-groupe de $G$ comme dans le th\'eor\`eme de Cerednik-Drinfeld : pour cela voir \cite{GKrig}, chapitre $7$.} \`a ceux de la section \ref{calculdeRham} que, pour toute repr\'esentation irr\'eductible $\rho$ de $D^*/1+\varpi_D \O_D$ de correspondante de de Jacquet-Langlands $\pi$, l'on a un isomorphisme compatible aux $(\varphi,N,\mathcal{G}_{\qp})$-structures des deux membres
\[ \ho_G((H_{\rm HK}^1(\widehat{\Sigma}_{1+\varpi_D \O_D,K})^{\rho})^*,\pi) = \varinjlim_{K^p} H_{\rm HK}^1(\widehat{\mathrm{Sh}}_{(1+\varpi_D\O_D)K^p,K})[\pi(g)_f^p]^{\rho}, \]
o\`u $g$ est une forme quaternionique telle que $\pi(g)_p=\rho$, comme dans la partie \ref{construction}. Saito \cite{Saito} a montr\'e que la structure de $(\varphi,N,\mathcal{G}_{\qp})$-module donn\'ee par la th\'eorie de Hyodo-Kato classique sur l'espace vectoriel $\varinjlim_{K^p} H_{\mathrm{dR}}^1(S_{K_pK^p,K})[\pi(f)^p]^{\rho}$ \'etait celle de $M(\pi)$. On d\'eduit donc de l'isomorphisme pr\'ec\'edent une identification naturelle
\[ H_{\rm dR}^1(\Sigma_{1+\varpi_D \O_D})^{\rho} = \pi^* \otimes M_{\rm dR}(\pi)^*. \]

Soit $\mathcal{L}$ une droite de $M_{\rm dR}(\pi)$. Via cette identification, on peut voir $\mathcal{L}^{\perp} \otimes \pi^*$ comme un sous-espace de $H_{\mathrm{dR}}^1(\Sigma_{1+\varpi_D \O_D})^{\rho}$. Breuil et Strauch d\'efinissent la repr\'esentation $BS(\mathcal{L})$ de $G$ comme le dual de la pr\'eimage dans $\Omega^1(\Sigma_{1+\varpi_D \O_D})^{\rho}$ de $\mathcal{L}^{\perp} \otimes \pi^*$\footnote{Comme d'habitude, $\rho$ est fix\'ee et sous-entendue dans la notation $BS(\mathcal{L})$.}.

Soit $g$ une forme quaternionique telle que $\pi(g)_p=\rho$, comme dans la partie \ref{construction}. La forme $g$ a une repr\'esentation galoisienne associ\'ee $r$ et $r_p=r_{|_{\mathcal{G}_{\qp}}}$ est de de Rham \`a poids $0, 1$ avec $D_{\rm pst}(r_p)=M(\pi)$. Donc $r_p$ correspond au choix d'une droite, not\'ee $\mathcal{L}$, sur $M_{\rm dR}(\pi)$. On sait que $\Pi_{\mathcal{L}}^{\rm an}=\Pi(r_p)^{\rm an}$ est un quotient de $(\Omega^1(\Sigma_{1+\varpi_D \O_D})^{\rho})^*$, puisque ce dernier est isomorphe \`a $\Pi(\pi,2)$ comme $G$-repr\'esentation, et que ce quotient correspond au quotient $\pi \otimes M_{\rm dR}(\pi) \to \pi \otimes M_{\rm dR}(\pi)/\mathcal{L}'$, pour un certain $\mathcal{L'}$. Autrement dit, l'image de la fl\`eche naturelle
\[ \ho_G((\Omega^1(\Sigma_{1+\varpi_D \O_D})^{\rho})^*,\Pi(r_p)^{\rm an})  \to  \ho_G((H_{\rm dR}^1(\Sigma_{1+\varpi_D \O_D})^{\rho})^*,\pi)=M_{\rm dR}(\pi)^* \]
est la droite $(\mathcal{L}')^{\perp}$. Or, comme on l'a d\'ej\`a not\'e (remarque \ref{schneider}), cette fl\`eche correspond \`a la fl\`eche naturelle
\[ \underset{K_p} \varinjlim ~ \Omega^1(\mathrm{Sh}_{(1+\varpi_D\O_D)K^p})[\pi(g)_f^p]^{\rho} \to \underset{K_p} \varinjlim ~ H_{\rm dR}^1(\mathrm{Sh}_{(1+\varpi_D \O_D)K^p})[\pi(g)_f^p]^{\rho}. \]
Le th\'eor\`eme de comparaison \'etale-de Rham appliqu\'e \`a la cohomologie de la courbe de Shimura propre $\mathrm{Sh}_{(1+\varpi_D \O_D)K^p}$ implique que la filtration de Hodge sur la cohomologie de de Rham
\[ \underset{K_p} \varinjlim ~ H_{\rm dR}^1(\mathrm{Sh}_{(1+\varpi_D \O_D)K^p})[\pi(g)_f^p]^{\rho} \simeq M_{\rm dR}(\pi)^* \]
est donn\'ee par la droite $\mathcal{L}^{\perp}$. Autrement dit, $\mathcal{L}'=\mathcal{L}$ et pour ce $\mathcal{L}$, on a donc bien $\Pi_{\mathcal{L}}^{\rm an}=BS(\mathcal{L})$.

On recommence ensuite le m\^eme jeu avec une autre forme quaternionique correspondant \`a une filtration diff\'erente (il en existe : il suffit de faire varier la repr\'esentation r\'esiduelle). On a donc un autre $\mathcal{L}'$ pour lequel on sait que $\Pi_{\mathcal{L}'}^{\rm an}=BS(\mathcal{L})$. 

Identifions $\mathbf{P}(M_{\mathrm{dR}})$ \`a $\mathbf{P}^1(\qp)$ via le choix de la base $(\mathcal{L},\mathcal{L}')$ de $M_{\mathrm{dR}}$. Soit maintenant $\mathcal{L}''$ une filtration admissible quelconque, et $(a,b) \in \mathbf{P}^1(\qp)$ l'\'el\'ement correspondant. Alors on a 
\[ [BS(\mathcal{L}'')]= a [BS(\mathcal{L})] + b [BS(\mathcal{L}')] \]
et  
\[ [\Pi_{\mathcal{L}''}^{\rm an}]= a [\Pi_{\mathcal{L}}^{\rm an}] + b [\Pi_{\mathcal{L}'}^{\rm an}] \]
dans le groupe $\mathrm{Ext}_G^1((\O(\Sigma_n)^{\rho})^*,\pi)$. Donc $BS(\mathcal{L}'')=\Pi_{\mathcal{L}''}^{\rm an}$. 

La conjecture de Breuil-Strauch dans sa formulation originale est donc un corollaire des r\'esultats obtenus. 
\end{remarque}

\section{Compl\'ements : quelques corollaires et une question}\label{complements}

\subsection{Preuves des th\'eor\`emes \ref{corollairebis} et \ref{corollaireter}} Nous rassemblons dans cette section les preuves des corollaires annonc\'es dans l'introduction.

\begin{proof}[D\'emonstration du th\'eor\`eme \ref{corollairebis}] 
D'apr\`es \cite[p. 20 et 174]{Cbigone}, le membre de droite est l'image par $1-\varphi$ de $(tN_{\rm rig}(\pi))^{\psi=1}$. Or, on a d\'ej\`a vu que cette image est isomorphe au membre de gauche dans le th\'eor\`eme \ref{psi egal un}. 
\end{proof}

\begin{proof}[D\'emonstration du th\'eor\`eme \ref{corollaireter}] Soit $f \in \O(\Sigma_n)$ une fonction infiniment primitivable. Ecrivons
\[ f = \sum_{i=1}^r v_i \otimes f_i \]
avec $v_i \in \rho_i$, o\`u $\rho_i$ est une repr\'esentation lisse irr\'eductible de $D^*$, et $f_i \in \O(\Sigma_n)^{\rho_i}$, pour $i=1,\dots,r$. Fixons $i$ et supposons $\rho_i$ non triviale. L'hypoth\`ese sur $f$ implique que $f_i$ est dans l'image de l'op\'erateur $(u^+)^j$ sur $\O(\Sigma_n)^{\rho_i}=\Pi(\pi_i,0)^*$, o\`u $\pi_i=\mathrm{JL}(\rho_i)$ (puisque $\rho_i$ est non triviale, on peut appliquer le th\'eor\`eme \ref{main1}) pour tout $j$. En outre, si l'on choisit $\Pi\in \mathcal{V}(\pi_i)$, on peut voir $f_i$ comme un \'el\'ement de $(\Pi^{\rm an})^*$. Le th\'eor\`eme $8.4.3$ de \cite{Dthese} dit que le sous-espace de $\Pi^{\rm an}$ des vecteurs tu\'es par une puissance de $u^+$ est dense dans $\Pi^{\rm an}$ (cette propri\'et\'e \'equivaut au fait que la repr\'esentation galoisienne correspondante soit non trianguline). Cela implique en particulier que l'intersection des $\mathrm{Im}((u^+)^j)$ sur le dual est nulle. On en d\'eduit $f_i=0$. La seule possibilit\'e est donc que $\rho_i$ soit triviale, i.e. que $f \in \O(\Omega)$. 
\end{proof}
     
  \begin{remarque}
    Un \'el\'ement $z$ de l'intersection des $\mathrm{Im}((u^+)^j)$, vu comme \'el\'ement de $D_{\rm rig}\boxtimes\p1$, satisfait trivialement 
    ${\rm Res}_{\zp}(z)=0$. Ainsi, 
    au lieu d'utiliser \cite{Dthese} on aurait pu utiliser l'injectivit\'e \cite{Colmezpoids} de l'application ${\rm Res}_{\zp}: (\Pi^{\rm an})^*\to D_{\rm rig}$.
  \end{remarque}
     
\subsection{Le complexe de de Rham dans la cat\'egorie d\'eriv\'ee des $D(G)$-modules}\label{cat\'egorie d\'eriv\'ee}

Pour $n\geq 0$, notons $R\Gamma_{\mathrm{dR}}(\Sigma_n)$ le complexe de de Rham de $\Sigma_n$ :
\[ \O(\Sigma_n)  \overset{d} \longrightarrow \Omega^1(\Sigma_n). \]
D'apr\`es le th\'eor\`eme \ref{surjectif tour}, c'est un objet de la cat\'egorie d\'eriv\'ee $D^b(D(G))$ de la cat\'egorie ab\'elienne \cite{STInv} des $D(G)$-modules coadmissibles. On munit ce complexe de la filtration \og b\^ete\fg{}, de sorte que $F^i R\Gamma_{\mathrm{dR}}(\Sigma_n)=R\Gamma_{\mathrm{dR}}(\Sigma_{n})$ si $i < 0$, $F^i R\Gamma_{\mathrm{dR}}(\Sigma_n)=0$ si $i> 0$ et 
\[ F^0 R\Gamma_{\mathrm{dR}}(\Sigma_n) = [0 \to \Omega^1(\Sigma_n)]. \] 

\begin{proposition}
Pour tout $V\in \mathcal{V}(\pi)$ on a un isomorphisme \textit{d'espaces vectoriels filtr\'es} (canonique \`a scalaire pr\`es)
\[ \ho_{D^b(D(G))}((\Pi(V)^{\mathrm{\mathrm{an}}})^*,R\Gamma_{\mathrm{dR}}(\Sigma_n)^{\rho}[1]) = D_{\mathrm{dR}}(V), \]
pour tout $n$ suffisamment grand.
\end{proposition}

\begin{remarque}
a) Un r\'esultat plus satisfaisant serait de remplacer le complexe de de Rham par un complexe convenable de cohomologie log-rigide et d'affirmer l'existence d'un \textit{isomorphisme canonique de $(\varphi,N,\mathcal{G}_{\qp})$-modules filtr\'es} entre le membre de gauche et $D_{\rm pst}(V)$. Cela permettrait de r\'ecup\'erer la repr\'esentation $V$ \`a partir de $\Pi(V)^{\mathrm{\mathrm{an}}}$. 

b) L'\'enonc\'e de la proposition donne une information nettement moins pr\'ecise sur $\Pi(V)^{\rm an}$ que la conjecture de Breuil-Strauch. Au moins pour les $V\in \mathcal{V}(\pi)$ d'origine globale, Peter Scholze nous a d'ailleurs indiqu\'e un argument qui devrait permettre de d\'eduire la proposition du th\'eor\`eme de compatibilit\'e local-global d'Emerton et du lemme de Poincar\'e $p$-adique (\cite[cor. 6.13]{Shodge}).

c) Pour $k \in \z$ et $n\geq 0$, notons $R\Gamma_{\mathrm{dR}}(\Sigma_{n},k)=R\Gamma_{\mathrm{dR}}(\Sigma_n) \otimes (\mathrm{Sym}^k)^*$ le complexe de de Rham de $\Sigma_n$ tordu par la repr\'esentation alg\'ebrique $(\mathrm{Sym}^k)^*$ :
\[ \O(\Sigma_n) \otimes (\mathrm{Sym}^k)^* \overset{d \otimes id} \longrightarrow \Omega^1(\Sigma_n) \otimes (\mathrm{Sym}^k)^*. \]
On pourrait aussi munir ce complexe d'une filtration, comme le font Schneider et Stuhler dans \cite[p. 95-97]{SS}, en prenant le produit tensoriel de la filtration \og b\^ete\fg{} par une certaine filtration sur $(\mathrm{Sym}^k)^*$. On obtiendrait ainsi une filtration telle que $F^i R\Gamma_{\mathrm{dR}}(\Sigma_{n},k)=R\Gamma_{\mathrm{dR}}(\Sigma_{n},k)$ si $i < -k$, $F^i R\Gamma_{\mathrm{dR}}(\Sigma_{n},k)=0$ si $i> 0$ et 
\[ F^i R\Gamma_{\mathrm{dR}}(\Sigma_{n},k) = [0 \to \O(k+2)(\Sigma_n) \otimes \det{}^{k+1}], ~ -k\leq i \leq 0, \]
et on pourrait prouver l'exact analogue de la proposition pr\'ec\'edente avec des poids de Hodge-Tate \'egaux \`a $0, k+1$. Toutefois cela alourdit consid\'erablement les notations.
\end{remarque}

\begin{proof}
Le cas $\rho=1$ est un r\'esultat de Schraen \cite{Schraen}. En fait, Schraen raffine l'\'enonc\'e en un isomorphisme de $(\varphi,N)$-modules filtr\'es, apr\`es avoir muni le membre de gauche d'un Frobenius et d'une monodromie en s'inspirant de la th\'eorie de Hyodo-Kato. Cela lui permet de d\'efinir un scindage canonique du complexe introduit dans la remarque pr\'ec\'edente : $R\Gamma_{\mathrm{dR}}(\Sigma_0,k)^{D^*} \simeq (H_{\mathrm{dR}}^{0}(\Omega) \oplus H_{\mathrm{dR}}^1(\Omega)[-1]) \otimes (\mathrm{Sym}^k)^*$. La monodromie $N$ est un \'el\'ement du $\ext^1$ entre ces deux composantes et le membre de gauche de la proposition se d\'ecompose en tant que $\varphi$-module comme une somme directe
\[ \ext_{D(G)}^1((\Pi(V)^{\mathrm{\mathrm{an}}})^*,H_{\mathrm{dR}}^0(\Omega) \otimes (\mathrm{Sym}^k)^*) \oplus \ho_{D(G)}((\Pi(V)^{\mathrm{\mathrm{an}}})^*, H_{\mathrm{dR}}^1(\Omega) \otimes (\mathrm{Sym}^k)^*)  \]
\[                = \ext_G^1(\mathrm{Sym}^k,\Pi(V)^{\mathrm{\mathrm{an}}}) \oplus \ho_G(\mathrm{St} \otimes \mathrm{Sym}^k,\Pi(V)^{\mathrm{\mathrm{an}}}), \]
ce qui explique pourquoi en niveau $0$, des vecteurs localement alg\'ebriques apparaissent \`a la fois en sous-objet en en quotient de $\Pi(V)^{\mathrm{\mathrm{an}}}$. 

Le cas $\rho \neq 1$ se d\'eduit du th\'eor\`eme \ref{main2}. En effet, le complexe $R\Gamma_{\mathrm{dR}}(\Sigma_n)^{\rho}$ (o\`u $n$ est choisi suffisamment grand) est scind\'e quasi-isomorphe \`a $(\Pi(\pi,2)^{u^+=0})^*[-1]$ : en effet, tous les groupes de cohomologie du complexe sont nuls sauf celui de degr\'e $1$ qui vaut $(\Pi(\pi,2)^{u^+=0})^*$, d'apr\`es le th\'eor\`eme \ref{main2}. On a donc trivialement un isomorphisme d'espaces vectoriels (sans filtration)
\[ \ho_{D^b(D(G))}((\Pi(V)^{\mathrm{\mathrm{an}}})^*,R\Gamma_{\mathrm{dR}}(\Sigma_{n})^{\rho}[1]) = D_{\mathrm{dR}}(V), \]
puisque $(\Pi(\pi,2)^{u^+=0})^*=\pi^* \otimes D_{\rm dR}(V)^*$. Le fait que les filtrations co\"incident est un corollaire direct de la conjecture.
\end{proof}

\subsection{Fonctions au bord de $\Sigma_n$ et faisceau $U\to tN_{\rm rig} \boxtimes U$} \label{bord}
La dualit\'e de Serre pour les vari\'et\'es rigides Stein \cite{Chiar} donne un isomorphisme de repr\'esentations de $G \times D^*$
\[ \Omega^1(\Sigma_n)^* \simeq H_c^1(\Sigma_n,\O). \]
Comme $\Sigma_n$ est Stein, on a en outre une suite exacte
\begin{equation}
0 \to \O(\Sigma_n) \to \underset{Z} \varinjlim ~ \O(\Sigma_n - Z) \to H_c^1(\Sigma_n,\O)  \to 0 \label{liminductive}
\end{equation}
o\`u $Z$ d\'ecrit l'ensemble des r\'eunions finies d'affino\"ides admissibles de $\Sigma_n$. Notons $\tau_n$ la compos\'ee de la surjection canonique de $\Sigma_n$ sur $\Sigma_0$ avec la r\'etraction du demi-plan sur l'arbre de Bruhat-Tits, dont on fixe l'origine standard. Soit $B_i$ la boule centr\'ee en l'origine dans l'arbre de rayon $i$. Si $U$ est un ouvert de $\mathbf{P}^1(\qp)$, il lui correspond un ouvert $V$ de l'arbre, constitu\'e de la r\'eunion de toutes les demi-droites partant de l'origine aboutissant en un point de $U$, et l'on pose
\[ \mathcal{F}_{\pi}(U)= \underset{i} \varinjlim ~ \O(\tau_n^{-1}(V - B_i))^{\rho}, \]
ce qui a un sens puisque $D^*$ agit sur $\tau_n^{-1}(V - B_i)$. Cela d\'efinit un faisceau $\mathcal{F}_{\pi}$ sur $\mathbf{P}^1(\qp)$. Ce faisceau est en quelque sorte le faisceau des sections du fibr\'e structural \og au bord\fg{} de $\Sigma_n$. Il est alors tentant de formuler la conjecture suivante, qui est un prolongement naturel de la conjecture de Breuil-Strauch.

\begin{conj}
Le faisceau $\mathcal{F}_{\pi}$ sur $\mathbf{P}^1(\qp)$ est le faisceau $U\mapsto t N_{\rm rig}(\pi) \boxtimes U$ de Colmez. La suite exacte de repr\'esentations de $G$ (voir \cite{Colmezpoids})
\[ 0 \to \Pi(\pi,0)^* \to t N_{\rm rig}(\pi) \boxtimes \mathbf{P}^1(\qp) \to \Pi(\pi,2) \to 0 \]
s'identifie \`a la suite exacte d\'eduite de \eqref{liminductive} :
\[ 0 \to \O(\Sigma_n)^{\rho} \to \mathcal{F}_{\pi}(\mathbf{P}^1(\qp)) \to (\Omega^1(\Sigma_n)^{\rho})^* \to 0. \]
\end{conj}

\section{Appendice: compatibilit\'e local-global (d'apr\`es Emerton)} 
   
        Nous expliquons dans cet appendice la preuve du th\'eor\`eme \ref{locglobfort}, en suivant de mani\`ere pleinement fid\`ele Emerton
        \cite{Emcomp}. Nous reprenons les notations des sections 4.1 et 5.1 (en fixant $K^p$ et en posant $X=X(K^p)$).

\begin{lemme}\label{injective 1}  Le groupe 
  ${\rm GL}_2(\zp)$ agit librement sur $X$, avec un nombre fini d'orbites. En particulier, il existe $s>0$ tel que pour tout
  $\zp$-module topologique $M$ l'on ait 
un isomorphisme de ${\rm GL}_2(\zp)$-repr\'esentations 
$$\con(X,M)\simeq \con({\rm GL}_2(\zp), M)^{\oplus s}.$$
\end{lemme}

  Le r\'esultat suivant est un des ingr\'edients de base de la th\'eorie.

\begin{lemme}\label{dense 1} Si $\mathcal{B}$ est un facteur direct topologique (en tant que $G$-module) de $\con(X)$, alors 
$\mathcal{B}_{{\rm GL}_2(\zp)-\rm alg}$ est dense dans $\mathcal{B}$.

\end{lemme}

\begin{proof} Il suffit de le faire pour $\mathcal{B}=\con(X)$, dans quel cas cela d\'ecoule du lemme 
\ref{injective 1} et du th\'eor\`eme de Mahler (voir \cite[prop. A.3]{Pa} et \cite[prop. 2.12]{PCD} pour des r\'esultats g\'en\'eraux \`a ce sujet).
\end{proof}

  La version \og en famille\fg{} de la correspondance de Langlands locale $p$-adique pour $G$ 
  permet de construire un $A$-module orthonormalisable $\Pi^{\mathrm{univ}}$, avec action continue de $G$, tel que pour tout $\mathfrak{p}\in {\rm MaxSpec}(A[1/p])$ on ait un isomorphisme $$\Pi^{\mathrm{univ}} \otimes_A k(\mathfrak{p}) \simeq \Pi(\mathfrak{p}).$$
   L'existence de $\Pi^{\mathrm{univ}}$ est un r\'esultat profond mais standard de la th\'eorie, cf. par exemple \cite{BBourbaki, Cbigone, kisin} (rappelons que nous supposons que la repr\'esentation modulo $p$
   de $\mathcal{G}_{\qp}$ est absolument irr\'eductible). De plus, la compatibilit\'e avec la correspondance modulo $p$ montre que $$\overline{\pi}:=\Pi^{\rm univ}/\mathfrak{m}\Pi^{\rm univ}$$
   est une repr\'esentation lisse irr\'eductible de $G$ sur $k_L$. 

\begin{definition} On note 
\[ M= \ho_{A[G]}^{\rm cont} (\Pi^{\mathrm{univ}},\con(X,\O_L)_{\mathfrak{m}}),\]
 $\Pi^{\mathrm{univ}}$ \'etant muni de la topologie $\mathfrak{m}$-adique et $\con(X,\O_L)_{\mathfrak{m}}$ de la topologie induite par $\con(X, \O_L)$.  Alors $M$ est un $\O_L$-module plat, s\'epar\'e complet pour la topologie $p$-adique, et on
 note $$M^*=\ho_{\O_L}(M,\O_L)$$ le dual de Schikhof de $M$, que l'on munit de la topologie de la convergence simple. On a r\'eciproquement $M=\ho_{\O_L}^{\rm cont}(M^*,\O_L)$. 
\end{definition}

\begin{remarque}\label{stupide mais utile}
On voit imm\'ediatement qu'on a des isomorphismes canoniques
$$M[\mathfrak{p}][1/p]=k(\mathfrak{p}) \otimes_{A/\mathfrak{p}} M[\mathfrak{p}]\simeq {\rm Hom}_{G}^{\mathrm{cont}}(\Pi(\mathfrak{p}), \con(X)[\mathfrak{p}])\simeq {\rm Hom}_{A[G]}^{\rm cont}(\Pi(\mathfrak{p}), \con(X)_{\mathfrak{m}}) $$
pour tout $\mathfrak{p}\in {\rm MaxSpec}(A[1/p])$.
\end{remarque}

Dans un premier temps, nous allons commencer par prouver l'\'enonc\'e plus faible suivant.
\begin{proposition}\label{locglobfaible}
Pour tout id\'eal maximal $\mathfrak{p}$ de $A[1/p]$, 
\[ \ho_G^{\rm cont}(\Pi(\mathfrak{p}),\con(X)[\mathfrak{p}]) \neq 0. \]
\end{proposition}
\begin{proof} D'apr\`es la remarque \ref{stupide mais utile} 
 il s'agit de justifier que $M[\mathfrak{p}][1/p] \neq 0$, pour tout id\'eal maximal $\mathfrak{p}$ de $A[1/p]$.
 L'id\'ee (due \`a Emerton) est de d\'emontrer 
  cet \'enonc\'e par interpolation $p$-adique, en le prouvant pour une famille dense d'id\'eaux maximaux (form\'ee de points correspondant \`a des repr\'esentations galoisiennes cristallines). 
Cela demande quelques pr\'eliminaires. 

\begin{lemme}\label{cofg}
Le $A$-module $M^*$ est de type fini. 
\end{lemme}

\begin{proof}
Comme $M^*$ est compact, il suffit de montrer que $M^*/\mathfrak{m}M^*$ est de dimension finie sur $k_L$. Or, l'isomorphisme 
$M=\ho_{\O_L}^{\rm cont}(M^*,\O_L)$ induit un isomorphisme 
$$ (M/\pi_LM)[\mathfrak{m}]\simeq {\rm Hom}_{\O_L}(M^*, k_L)[\mathfrak{m}]\simeq {\rm Hom}_{k_L}( M^*/\mathfrak{m}M^*, k_L).$$
Il suffit donc de d\'emontrer que $(M/\pi_LM)[\mathfrak{m}]$ est de dimension finie sur $k_L$. Mais, par d\'efinition de $M$, on dispose d'une injection de $k_L$-espaces vectoriels
$$(M/\pi_LM)[\mathfrak{m}]\subset {\rm Hom}_{k_L[G]}(\Pi^{\rm univ}/\mathfrak{m}\Pi^{\rm univ}, {\rm LC}(X, k_L)_{\mathfrak{m}}).$$ 
 Comme 
$\overline{\pi}=\Pi^{\mathrm{univ}}/\mathfrak{m}\Pi^{\rm univ}$ est irr\'eductible et lisse, le choix d'un vecteur non nul quelconque 
$v$ de $\overline{\pi}$ et d'un sous-groupe ouvert $K_p$ qui le fixe fournit un plongement 
$$\ho_{k_L[G]}(\overline{\pi}, {\rm LC}(X,k_L)_{\mathfrak{m}})\subset ({\rm LC}(X,k_L)_{\mathfrak{m}})^{K_p}\subset {\rm LC}(X(K_p), k_L)$$
et le dernier espace est de dimension finie sur $k_L$ car $X(K_p)$ est fini.
\end{proof}

\begin{lemme} \label{FINITO}
Si $\mathfrak{p}$ est un id\'eal maximal de $A[1/p]$, alors $M[\mathfrak{p}][1/p]$ est un $L$-espace vectoriel de dimension finie, dual de $M^*\otimes_{A} k(\mathfrak{p})$, et il est non nul si et seulement si 
 $\mathfrak{p} \in \mathrm{Supp} ~ M^*[1/p]$. 
\end{lemme}
\begin{proof} Nous avons $$M[\mathfrak{p}]={\rm Hom}_{\O_L}^{\rm cont}(M^*, \O_L)[\mathfrak{p}]={\rm Hom}_{\O_L}^{\rm cont} (M^*/\mathfrak{p}M^*, \O_L).$$
Le lemme \ref{cofg} montre que $M^*/\mathfrak{p}M^*$ est un $A/\mathfrak{p}$-module de type fini, donc un $\O_L$-module de type fini, ce qui fournit des isomorphismes 
$$M[\mathfrak{p}][1/p]={\rm Hom}_L(M^*[1/p]/\mathfrak{p}, L)={\rm Hom}_L(M^*\otimes_A k(\mathfrak{p}), L).$$
Le reste se d\'eduit du lemme de Nakayama.
\end{proof}

  Vu le lemme pr\'ec\'edent, il suffit de montrer que $\mathrm{Supp} ~ M^*[1/p]$ est Zariski dense dans $\mathrm{Spec} ~ A[1/p]$ (il est automatiquement ferm\'e 
  puisque $M^*[1/p]$ est de type fini sur $A[1/p]$). Nous allons exhiber une famille $\mathcal{C}$ d'id\'eaux maximaux de $A[1/p]$, qui est Zariski dense dans 
  $\mathrm{Spec} ~ A[1/p]$ et contenue dans $\mathrm{Supp} ~ M^*[1/p]$. 
  \\

Soit $\sigma$ une repr\'esentation automorphe de $\ob^*(\mathbf{A})$ telle que :

a) $\sigma$ soit non ramifi\'ee en dehors de $\Sigma$.

b) $\sigma_f^{K^p} \neq 0$ et $\sigma_p^{{\rm GL}_2(\zp)} \neq 0$.

c) La repr\'esentation galoisienne associ\'ee $r_{\sigma}$ v\'erifie $\overline{r}_{\sigma}=\overline{r}$. 

L'action de $\tilde{\mathbf{T}}_{\Sigma}$ sur $\sigma_f$ d\'efinit un morphisme $\tilde{\mathbf{T}}_{\Sigma} \to L$, qui s'\'etend par continuit\'e en un morphisme $A[1/p] \to L$ (puisque $\overline{r}_{\sigma}=\overline{r}$), dont le noyau est un id\'eal maximal $\mathfrak{p}_{\sigma}$ de $A[1/p]$. 
On note $$\mathcal{C}=\{ \mathfrak{p}_{\sigma}, \sigma ~ \mathrm{comme ~ avant} \}.$$

\begin{lemme}\label{dense}
L'ensemble $\mathcal{C}$ est Zariski dense dans $\mathrm{Spec} ~ A[1/p]$.
\end{lemme} 
\begin{proof}
Notons que par d\'efinition, $A[1/p]$ agit fid\`element sur $\con(X)_{\mathfrak{m}}$. Il suffit donc de montrer que tout $t \in \cap_{\mathfrak{p} \in \mathcal{C}} \mathfrak{p}$ agit par $0$ sur $\con(X)_m$. Comme l'action est continue, il suffit pour cela de prouver que  $\sum_{\mathfrak{p} \in \mathcal{C}} \con(X)_{\mathfrak{m}}[\mathfrak{p}]$ est dense dans $\con(X)_{\mathfrak{m}}$. Comme $\con(X)_{\mathfrak{m}}$ est un facteur direct topologique de $\con(X)$, l'ensemble des vecteurs ${\rm GL}_2(\zp)$-alg\'ebriques de $\con(X)_{\mathfrak{m}}$ est dense dans $\con(X)_{\mathfrak{m}}$, d'apr\`es le lemme \ref{dense 1}. Il suffit donc de justifier que $$(\con(X)_{\mathfrak{m}})_{{\rm GL}_2(\zp)-alg} \subset \sum_{\mathfrak{p} \in \mathcal{C}} \con(X)_{\mathfrak{m}}[\mathfrak{p}].$$ Mais ceci est un corollaire imm\'ediat du lemme \ref{formes automorphes}.
\end{proof}

  On note ${\rm LP}(X)$ l'espace des vecteurs localement alg\'ebriques de $\con(X)_{\mathfrak{m}}$. 
  
\begin{lemme}\label{cas cristallin}
Pour tout $\mathfrak{p} \in \mathcal{C}$, $M[\mathfrak{p}]\otimes_{A/\mathfrak{p}} k(\mathfrak{p})\neq 0$, autrement dit $${\rm Hom}_G^{\rm cont}(\Pi(\mathfrak{p}), \con (X)[\mathfrak{p}])\ne 0.$$
De plus, ${\rm LP}(X)[\mathfrak{p}]$ est inclus dans l'image de la fl\`eche naturelle
\[ \Pi(\mathfrak{p}) \otimes \ho_{G}^{\rm cont}(\Pi(\mathfrak{p}),\con(X)[\mathfrak{p}]) \to \con(X)[\mathfrak{p}]. \]

\end{lemme}
\begin{proof}
Soit $\mathfrak{p}\in \mathcal{C}$, $\sigma$ la repr\'esentation automorphe correspondante. Comme $\sigma_p^{{\rm GL}_2(\zp)}\neq 0$, la restriction de $r_{\sigma}=r(\mathfrak{p})$ au groupe de d\'ecomposition en $p$ est cristalline. En outre, $\overline{r(\mathfrak{p})}_{|_{\mathcal{G}_{\qp}}}=\overline{r}_{|_{\mathcal{G}_{\qp}}}$ est absolument irr\'eductible, et donc en particulier $r(\mathfrak{p})_{|_{\mathcal{G}_{\qp}}}$ l'est. Notons $a<b$ ses poids de Hodge-Tate, et posons $W=\mathrm{Sym}^{b-a-1}(L^2) \otimes \det{}^a$. 

    D'apr\`es Berger-Breuil \cite{BB} et la construction de la correspondance de Langlands locale $p$-adique pour $G$, le compl\'et\'e unitaire universel de $\sigma_p\otimes W$ est pr\'ecis\'ement 
    $\Pi(\mathfrak{p})$. Comme $\sigma_p\otimes W$ est localement alg\'ebrique, on obtient 
    $${\rm Hom}_G^{\rm cont}(\Pi(\mathfrak{p}), \con(X)[\mathfrak{p}])={\rm Hom}_G^{\rm cont}(\sigma_p\otimes W, \con(X)[\mathfrak{p}])={\rm Hom}_G(\sigma_p\otimes W, {\rm LP}(X)[\mathfrak{p}]).$$
    De plus, le lemme \ref{formes automorphes} montre que le morphisme d'\'evaluation 
    $$ (\sigma_p\otimes W)\otimes {\rm Hom}_G(\sigma_p\otimes W, {\rm LP}(X)[\mathfrak{p}])\to {\rm LP}(X)[\mathfrak{p}]$$
   est un isomorphisme, ce qui permet de conclure (en utilisant l'injection $\sigma_p\otimes W\subset \Pi(\mathfrak{p})$).    \end{proof}
    
    Ceci ach\`eve la preuve de la proposition \ref{locglobfaible}.
\end{proof}    

Nous avons maintenant en main tous les ingr\'edients pour prouver le th\'eor\`eme \ref{locglobfort}. 

\begin{proof}[D\'emonstration du th\'eor\`eme \ref{locglobfort}]
Montrons tout d'abord que la fl\`eche naturelle d'\'evaluation (le produit tensoriel compl\'et\'e est pour la topologie $\pi_L$-adique)
\begin{eqnarray} (\Pi^{\mathrm{univ}} \widehat{\otimes}_{A} M) [1/p] \to \con(X)_{\mathfrak{m}} \label{evaluation} \end{eqnarray}
est un isomorphisme. Cette fl\`eche s'obtient apr\`es inversion de $p$ \`a partir de la fl\`eche :
\[ \Pi^{\mathrm{univ}} \widehat{\otimes}_{A} M \to \con(X,\O_L)_{\mathfrak{m}}. \]
Comme $\Pi^{\mathrm{univ}} \widehat{\otimes}_{A} M$ est s\'epar\'e pour la topologie $\pi_L$-adique (utiliser le lemme 3.1.16 de \cite{Emcomp}), on peut tester l'injectivit\'e de \eqref{evaluation} apr\`es r\'eduction modulo $\pi_L$ de ce morphisme :
\[ \Pi^{\mathrm{univ}}/\pi_L \otimes_{A} M/\pi_L \to  {\rm LC}(X, k_L)_{\mathfrak{m}}. \]
Nous allons montrer d'abord que ce morphisme est injectif en restriction \`a la $\mathfrak{m}$-partie. Pour cela remarquons que le lemme B.6 de \cite{Emcomp} fournit un isomorphisme 
$$(\Pi^{\mathrm{univ}}/\pi_L \otimes_{A} M/\pi_L)[\mathfrak{m}]\simeq {\rm Hom}_A(A/\mathfrak{m}, \Pi^{\rm univ}\otimes_A 
M/\pi_L)$$ 
$$\simeq {\rm Hom}_A(A/\mathfrak{m}, M/\pi_L)\otimes_{A/\pi_L} \Pi^{\rm univ}/\pi_L=(M/\pi_L)[\mathfrak{m}]\otimes_{A/\pi_L} \Pi^{\rm univ}/\pi_L$$ $$=(M/\pi_L)[\mathfrak{m}]\otimes_{k_L}\Pi^{\rm univ}/\mathfrak{m}=(M/\pi_L)[\mathfrak{m}]\otimes_{k_L} \bar{\pi}.$$
D'autre part la preuve du lemme \ref{cofg} fournit un plongement 
$$(M/\pi_L)[\mathfrak{m}]\subset {\rm Hom}_G(\bar{\pi}, {\rm LC}(X, k_L)[\mathfrak{m}]).$$
Il suffit donc de montrer que la fl\`eche
\[ \overline{\pi} \otimes_{k_L} \ho_{k_L[G]}(\overline{\pi}, {\rm LC}(X, k_L)[\mathfrak{m}]) \to  {\rm LC}(X, k_L)[\mathfrak{m}] \]
 est injective, ce qui d\'ecoule du fait que $\overline{\pi}$ est lisse irr\'eductible et ${\rm LC}(X, k_L)[\mathfrak{m}]$ est lisse. 
Ensuite, on v\'erifie sans mal que pour tout $x\in \Pi^{\rm univ}\otimes_A M$ on a $ \mathfrak{m}^N x\subset \pi_L\cdot (\Pi^{\rm univ}\otimes_A M)$ pour tout $N$ assez grand (il suffit de montrer que si $u\in M$, alors 
$\mathfrak{m}^N u\subset \pi_L\cdot M$ pour $N$ assez grand, ce qui se fait en regardant $u$ comme une forme lin\'eaire continue sur $M^*$). 
Ainsi tout \'el\'ement de $\Pi^{\mathrm{univ}}/\pi_L \otimes_{A} M/\pi_L$ est tu\'e par une puissance de $\mathfrak{m}$, ce qui permet de d\'eduire l'injectivit\'e du morphisme 
\[ \Pi^{\mathrm{univ}}/\pi_L \otimes_{A} M/\pi_L \to  {\rm LC}(X, k_L)_{\mathfrak{m}} \]
de son injectivit\'e sur la $\mathfrak{m}$-partie.

Prouvons la surjectivit\'e. Comme, d'apr\`es le lemme \ref{cofg}, $M^*$ est de type fini sur $A$, et comme $\Pi^{\mathrm{univ}}$ est un $A[G]$-module orthonormalisable admissible, l'image du morphisme \eqref{evaluation} est automatiquement ferm\'ee (combiner Proposition 3.1.3 et Lemma 3.1.16 de \cite{Emcomp}). Il suffit du coup de montrer que l'image de \eqref{evaluation} est dense. On a vu dans la preuve du lemme \ref{cas cristallin} que l'image de  $\Pi(\mathfrak{p}) \otimes \ho_{G}^{\rm cont}(\Pi(\mathfrak{p}),\con(X)[\mathfrak{p}])$ contient ${\rm LP}(X)[\mathfrak{p}]$. Par cons\'equent, l'image de notre morphisme contient $(\con(X)_{\mathfrak{m}})^{{\rm GL}_2(\zp)-\mathrm{alg}}$, qui forme un sous-espace dense. Cela finit la preuve du fait que 
$$(\Pi^{\mathrm{univ}} \widehat{\otimes}_{A} M) [1/p] \to \con(X)_{\mathfrak{m}}$$
est un isomorphisme.

Soit maintenant $\mathfrak{p}$ un id\'eal maximal de $A[1/p]$. Pour achever la preuve du th\'eor\`eme \ref{locglobfort}, on prend la $\mathfrak{p}$-partie de l'isomorphisme \eqref{evaluation}. Pour calculer la $\mathfrak{p}$-partie du terme de gauche
on utilise le lemme 3.1.17 de \cite{Emcomp}, qui fournit des isomorphismes
$$(\Pi^{\mathrm{univ}} \widehat{\otimes}_{A} M) [1/p][\mathfrak{p}]={\rm Hom}_A(A/\mathfrak{p}, \Pi^{\mathrm{univ}} \widehat{\otimes}_{A} M)[1/p]=$$
$${\rm Hom}_A(A/\mathfrak{p}, M)\widehat{\otimes}_A \Pi^{\rm univ}[1/p]=M[\mathfrak{p}]\widehat{\otimes}_{ A/\mathfrak{p}}(\Pi^{\rm univ}/\mathfrak{p}) [1/p]=M[\mathfrak{p}][1/p]\otimes_{k(\mathfrak{p})} \Pi(\mathfrak{p}),$$
la dernière égalité utilisant le fait que $M[\mathfrak{p}][1/p]$ est de dimension finie sur $k(\mathfrak{p})$ (lemme \ref{FINITO}). 
Pour finir, il suffit d'utiliser la remarque \ref{stupide mais utile} et la proposition
\ref{locglobfaible}.
\end{proof}

\end{document}